\documentclass[11pt, reqno]{article}

\usepackage{fullpage}
\usepackage{enumerate}
\usepackage{setspace}

\usepackage{amsmath,amsfonts,amssymb,graphics,amsthm, comment}
\usepackage{hyperref}
\hypersetup{
    colorlinks=true,
    linkcolor=blue,
    citecolor=red,
    urlcolor=blue,
    pdfborder={0 0 0}
}    
\usepackage{graphicx}
\usepackage{xcolor}

\usepackage{cleveref}

  \crefname{theorem}{Theorem}{Theorems}
  \crefname{thm}{Theorem}{Theorems}
  \crefname{lemma}{Lemma}{Lemmas}
  \crefname{lem}{Lemma}{Lemmas}
  \crefname{remark}{Remark}{Remarks}
  \crefname{prop}{Proposition}{Propositions}
  \crefname{proposition}{Proposition}{Propositions}
\crefname{notation}{Notation}{Notations}
\crefname{claim}{Claim}{Claims}
  \crefname{defn}{Definition}{Definitions}
  \crefname{corollary}{Corollary}{Corollaries}
  \crefname{section}{Section}{Sections}
  \crefname{figure}{Figure}{Figures}
  \crefname{question}{Question}{Questions}
  \crefname{exercise}{Exercise}{Exercises}
    \crefname{assumption}{Assumption}{Assumptions}

\newtheorem{thm}{Theorem}[section]

\newtheorem{conjecture}[thm]{Conjecture}
\newtheorem{claim}[thm]{Claim}
\newtheorem{lemma}[thm]{Lemma}
\newtheorem{corollary}[thm]{Corollary}

\newtheorem{proposition}[thm]{Proposition}

\newtheorem{question}[thm]{Question}

\numberwithin{equation}{section}

\theoremstyle{definition}
\newtheorem{remark}[thm]{Remark}
\newtheorem{example}[thm]{Example}

\def \min {\text{min}}

\def\cS{\mathcal{S}}
\def\cR{\mathcal{R}}

\def\cO{\mathcal{O}}

\def\cL{\mathcal{L}}

\def\cI{\mathcal{I}}
\def\cH{\mathcal{H}}
\def\cG{\mathcal{G}}
\def\cF{\mathcal{F}}
\def\cE{\mathcal{E}}

\def\cC{\mathcal{C}}
\def\cB{\mathcal{B}}
\def\cA{\mathcal{A}}

\def \ve {\varepsilon}

\def\P{\mathbb{P}}
\def\E{\mathbb{E}}

\def\R{\mathbb{R}}
\def\Z{\mathbb{Z}}
\def\N{\mathbb{N}}

\def\R{\mathbb{R}}

\def\1{\mathbf{1}}

\def  \p- {p\textunderscore}

\DeclareMathOperator{\w} {\mathsf{w}}

\def \cGm {\cG^{\bullet}_{\mathsf m}}
\def \Psic {\Psi_{\texttt{contr}}}
\def \Psiuc {\Psi_{\texttt{ucontr}}}

\DeclareMathOperator{\Tu}{T_\textsf{USpA}}

\usepackage{extarrows}

\date{}

\title{Minimal spanning arborescence}
\author{Gourab Ray \thanks{University of Victoria,  Research of G. Ray supported by NSERC 50311-57400} \and Arnab Sen \thanks{University of Minnesota, Research of A. Sen partially supported by Simon Foundation MP-TSM-00002716} }

\begin{document}
\maketitle
\begin{abstract}
We study the minimal spanning arborescence which is the directed analogue of the minimal spanning tree, with a particular focus on its infinite volume limit and its geometric properties. We prove that in a certain large class of transient trees, the infinite volume limit exists almost surely. We also prove that for nonamenable, unimodular graphs, the limit is almost surely one-ended assuming a certain sufficient condition that guarantees the existence of the limit.

This object cannot be studied using well-known algorithms, such as Kruskal's or Prim's algorithm,  to sample the minimal spanning tree  which has been instrumental in getting analogous results about them (Lyons, Peres, and Schramm \cite{lyons2006minimal}). Instead, we use a recursive algorithm due to Chu, Liu, Edmonds, and Bock \cite{Edmonds, Chu-Liu, bock}, which leads to a novel stochastic process which we call the \emph{loop contracting random walk}. This is similar to the well-known and widely studied loop erased random walk \cite{lawler2018topics}, except instead of erasing loops we contract them. The full algorithm bears similarities with the celebrated Wilson's algorithm to generate uniform spanning trees and can be seen as a certain limit of the original Wilson's algorithm.
\end{abstract}


\section{Introduction}\label{sec:intro}

A minimum weighted directed spanning tree {with a fixed root} on a weighted directed graph, or the \textbf{minimal spanning arborescence} is a fundamental object in theoretical computer science. 
It has widespread applications in various topics which include disease outbreaks \cite{jombart2011reconstructing}, approximating the traveling salesman problem \cite{salii2021improving}, wireless networks \cite{li2004topology}, natural language processing  \cite{mcdonald2005non,wan2009improving}, genetics \cite{horns2016lineage} and social networks \cite{zehnder2014towards}. Finding an efficient algorithm to sample the minimal spanning arborescence in a network 
has various elegant solutions, the major ones are attributed to Edmonds \cite{Edmonds}, Chu-Liu \cite{Chu-Liu} and Bock \cite{bock}, (see \cite{bother2023efficiently} for a modern exposition on the history of the algorithmic side of the problem).

The undirected version of the problem, or the \textbf{minimum spanning tree (MST)} problem is of similar fundamental importance with various celebrated algorithms to sample it, namely Kruskal's \cite{kruskal1956shortest} and Prim's \cite{prim1957shortest} algorithms. The study of the MST on graphs with i.i.d.\ random edge weights has also received a lot of attention in the probability and statistical mechanics community where the inherent elegance and simplicity of the above algorithms has played a pivotal role. 
Indeed, various intriguing results about the geometry of the minimum spanning tree can be found in recent literature in this area. For example, the local weak limits of the MST in various infinite graphs were studied by Alexander \cite{AKS_93}, Lyons, Peres, and Schramm \cite{lyons2006minimal} and, the scaling limits in the complete graph were studied in \cite{addario_minimal} and that in the plane in \cite{garban2018scaling}. This list is however far from complete; and we refer the reader to these papers for various further references. 

Despite the popularity of its cousin, the MST, the minimum spanning arborescence on directed graphs with random edge weights has not received as much attention from the probability community. It is the purpose of this paper to rectify this and initiate an investigation of this object from a statistical mechanics perspective.  

\subsection{Definitions and Main results}\label{subsec:intro}
 Take a finite, connected undirected graph $G  = (V, E)$. We fix a distinguished vertex $\partial$ in $G$ which we call the \textbf{boundary.} Let $\vec E$ denote the collection of oriented edges where each element of $E$ is assigned two of its possible orientations. For $\vec e \in \vec E$, let $\vec e_+$ denote the head of $\vec e$ and let $\vec e_-$ denote the tail of $\vec e$. An edge $\vec e$ is an \emph{outgoing edge} from a vertex $v$ if $\vec e_-  = v$. 
 An \textbf{arborescence} is a subset of the oriented edges of $\vec E$ such that every vertex has at most one outgoing edge and there are no cycles. For any vertex $v$, one can naturally define a path $(\vec e_0,\vec e_1,\ldots, \vec e_k)$ in the arborescence where $(\vec e_0)_- = v$, $(\vec e_i)_+ = (\vec e_{i+1})_-$ for $0 \le i \le k-1$ and $(\vec e_k)_+$ has no outgoing edge from it. This path is called the \textbf{future} of $v$ (note that the future of a vertex could be empty if it is isolated).
 We say $u,v$ are in the same component of the arborescence if their futures merge. 
 The set of vertices which has no outgoing edge is called the boundary of the arborescence. Each component of the arborescence has a unique boundary vertex, and the futures of all vertices in the component merge into that vertex. When we fix a set of boundary vertices $\partial$, then an arborescence with boundary $\partial$ is called the 
\textbf{spanning arborescence} of $(G, \partial)$. Alternatively, a spanning arborescence of $(G, \partial)$ is the collection of oriented edges of $\vec E$ such that every vertex except $\partial $ has exactly one outgoing edge and there are no cycles.

We assign i.i.d.\ nonnegative weights $(U_{\vec e})_{{\vec e} \in \vec E}$ coming from a continuous distribution, e.g., Exp$(1)$ or Uniform $[0,1]$ to the oriented edges of $G$. A \textbf{minimal spanning arborescence of $(G, \partial)$ 
(MSA)} is the (a.s.\ unique) spanning arborescence with minimal weight. Usually, $\partial $ will be taken to be a fixed distinguished vertex in $G$

 If the weight of both orientations of an edge is the same for every edge, then the MSA is just the minimal spanning tree, which has been widely studied (see \cite{lyons2006minimal,penrose2003random,schrijver2005history,addario_minimal,ABS_minimal} and references therein). However, in the general case, none of the main tools used to study the MST, viz, Prim's or Kruskal's algorithm and invasion percolation is applicable for the MSA. For example, it turns out that the distribution of the tree depends on the distribution of the weights chosen (see \cref{sec:basic}) unlike the MST. 
 To the best of our knowledge, this model has not been studied from a probabilistic perspective. 
 Our goal in this article is to examine the question of the existence of weak limits and the geometry of the MSA on infinite graphs.

Let $G = (V,E)$ be an infinite, connected, locally finite graph, possibly containing multiple edges, but no self-loops and let $(U_{\vec e})_{\vec e \in \vec E}$ be a collection of nonnegative i.i.d.\ weights coming from a continuous distribution. Let $G_1 \subset G_2 \subset \ldots$ be a sequence of subgraphs of vertices of $G$ which exhaust $G$, i.e., $\cup_{i \ge 1} G_i = G$. Let $G^{\w}_i$ be the graph obtained by identifying all the vertices in the complement of $G_i$ into a single vertex $\partial_i$ and erasing all the self-loops. Let $T_i$  be the MSA of $(G_i^{\w}, \partial_i)$ with weights coming from $(U_{\vec e})_{\vec e \in \vec E}$. We say that the $T_i$ converges a.s.\ if, for every $v \in G$,  almost surely, the outgoing edge from $v$ in $T_i$ remains the same for sufficiently large $i$. If this limit does not depend (almost surely) on the choice of the exhaustion, we say that the \textbf{wired boundary MSA limit} exists almost surely in $G$.

A \emph{subdivision} of a graph is obtained by replacing each edge of it by a path of finite length. A \emph{subdivision} is $M$-\emph{bounded} if each edge is replaced by a path whose length is bounded by $M$. A \emph{subdivision} is just called \emph{bounded} if it is $M$-bounded for some $M$.

Our first result establishes the existence of the wired MSA for a class of infinite transient trees with i.i.d.\ exponential weights.
\begin{thm}\label{thm:convergence}
Let $T$ be an infinite undirected tree that is transient for the simple random walk and which has no vertex of degree 2.  Let $T' = (V,E)$ be a bounded subdivision of $T$, and suppose the weights of the oriented edges in $\vec E$ are i.i.d.\ Exponential $(1)$. Then the wired MSA limit in $T'$ exists almost surely. 
\end{thm}

Let us remark that the convergence for the wired MST (in any graph) is immediate from a certain geometric characterization of the edges in the MST: namely, an edge $e$ is in the wired MST if and only if there is no cycle in the graph with $U_e$ being the maximal weight of the edges in the cycle. In fact, this characterization has proven to be a vital tool in understanding the geometry of the MST. Unfortunately, we do not know of any such nice characterization of the MSA.

The technique we employ involves an algorithm that goes back to Chu, Liu, Edmonds and Bock \cite{bock,Chu-Liu,Edmonds} (which we abbreviate as CLEB algorithm in this article). We describe this in more detail later in \Cref{sec:techniques,sec:CLEB_finite,sec:LCRW}, but for now, we mention that sampling branches of the MSA involves a process that is very similar to the well-known loop erased random walk; we call this process a CLEB walk. Let us briefly describe the CLEB walk in the special case when the weights are i.i.d.\ Exponential $(1)$:  simply do a random walk (with unit weights on the unoriented graph) and whenever we create a loop, we contract it, and continue the process (a more precise definition of what `contraction' means can be found in \Cref{sec:CLEB_finite}). We say this process is transient at a vertex $x$ if the trace of the loop contracting random walk (LCRW) converges to an infinite path. It turns out that if the CLEB walk is transient almost surely from every vertex in a graph then the wired MSA limit exists almost surely. In fact, to prove \Cref{thm:convergence}, we show the following result.
\begin{thm}\label{thm:LCRW_transient}
Let $T$ be an infinite undirected tree that is transient for the simple random walk and which has no vertex of degree 2.  Let $T'$ be a bounded subdivision of $T$. Then the LCRW is transient at every vertex of $T'$ almost surely.
\end{thm}


Once we know that the limit exists, we turn to the properties of the wired limit of the MSA. An \textbf{arborescence in an infinite graph} with boundary $\partial$ is a subcollection of oriented edges from $\vec E$ which is acyclic, and such that every vertex not in $\partial$ has at most one outgoing edge and every vertex in $\partial $ has zero outgoing edge. The arborescence is \textbf{spanning} if $\partial  = \emptyset$. The future of a vertex can be defined in the same way as in the finite case, and two vertices $u,v$ are in the same component if their futures merge. Note that for any arborescence, this forms an equivalence relation, and thus it makes sense to talk about the connected components of an arborescence on a (possibly infinite) graph. For a positive integer $k$, a (undirected) tree is called $k$-ended if the maximum number of disjoint infinite paths in it is $k$. A component of an arborescence is $k$-ended if the tree obtained by replacing each oriented edge by an unoriented edge is $k$-ended.

It is known that the wired MST is almost surely one-ended in a large class of graphs \cite{AKS_93,lyons2006minimal}. We prove this statement for nonamenable graphs, assuming that the wired MSA limit exists a.s.\ The one-endedness of the components of the MSA implies that the `past' of every vertex is finite, and is the spanning tree analogue of the results of the type `no percolation at criticality' (e.g. \cite{BLPS_noperc}), which is why this question is interesting.

Recall that an infinite, connected graph $G$ is \textbf{amenable} if 
\begin{equation}
\inf_{S \subset V}  h(S):=\inf_{S \subset V} \frac{|\partial S|}{|S|} = 0, \text{ otherwise it is \textbf{nonamenable}},
\end{equation}
where $\partial S$ is the set of vertices not in $S$ with at least one neighbor in $S$ and $|S|$ is the cardinality of $S$.

Our next result is general enough to be stated for unimodular random graphs. For a brief introduction to this subject, we refer to the beginning of \Cref{sec:end_unimod} and to \cite{AL_unimodular,curiennotes} for a more elaborate exposition of the topic. In the unimodular setup, the definition of amenability can be generalized to invariant amenability which can be roughly described as follows. A random rooted graph is amenable if one can find a sequence of invariant percolations $\omega_n$ on it with a.s.\ finite clusters such that $\E(h(\omega_n(\rho))) \to 0$ where $\omega_n(\rho)$ is the component of $\omega_n$ containing the root $\rho$ and the expectation is over the graph as well as the percolation. We call such random rooted graphs \textbf{invariantly amenable},  and \textbf{invariantly nonamenable} otherwise.
Note that almost sure nonamenability implies invariant nonamenability, while the converse is false. For example, a (unimodular version of) supercritical Galton-Watson trees conditioned to survive with a positive probability of producing a single child is amenable but invariantly nonamenable, see below for a detailed definition. We do not state the precise definitions here and refer to \cite[Section 8]{AL_unimodular} for details.  
\begin{thm}\label{thm:end}
Suppose $(G, \rho)$ is a unimodular random rooted graph that is a.s.\ infinite, connected, invariantly nonamenable and $\E(\deg(\rho))<\infty$. Given $(G, \rho)$ we put i.i.d.\ Exponential $(1)$ weights $(U_{\vec e})_{\vec e \in \vec E}$ on the set of oriented edges $\vec E$. Suppose almost surely, given such a $(G, \rho, (U_{\vec e})_{\vec e \in \vec E})$, the CLEB walk from $\rho$ is transient almost surely. Then almost surely, the wired MSA limit on $(G, \rho, (U_{\vec e})_{\vec e \in \vec E})$ exists, the limit has infinitely many infinite components and each component is one-ended.
\end{thm}
The choice of Exponential random variables in \Cref{thm:end} is simply for concreteness, we believe it is possible to extend it to more general continuous distributions by modifying parts of the proof appropriately. Next, we record the following corollary for unimodular, nonamenable trees.
\begin{corollary}\label{cor:unimod_tree}
Let $(T ,\rho)$ be a unimodular random rooted tree that is almost surely connected, locally finite, and nonamenable, and assume $\E(\deg(\rho))<\infty$. Then the wired MSA limit for i.i.d.\ Exponential $(1)$ weights on $\vec E$, where $T = (V,E)$, exists almost surely and the limit has infinitely many infinite components, and each component is one-ended almost surely.
\end{corollary}

We now present a result which is a generalization of \Cref{cor:unimod_tree}. Consider a Galton-Watson tree $T_{\text{GW}}$ with offspring distribution $Z$ with $1 < \E(Z) < \infty$. For the sake of brevity, we assume that $\P(Z=0)  =0$ (see \Cref{wlgp_0}). $T_{\text{GW}}$ comes with a natural root $\rho$ and we shall think of this random tree as a rooted random graph $(T_{\text{GW}}, \rho)$. This is the so-called supercritical Galton-Watson tree (which is infinite a.s.\ by our assumption).  It is well-known that $T_{\text{GW}}$ is not unimodular because the root vertex (vertex with no parent) is special: it has stochastically one less degree. However, there is a standard way to turn this into a unimodular random graph as follows. Take two independent copies of $T_{\text{GW}}$, join their roots by an edge, and uniformly pick one endpoint as the root $\rho$, or equivalently the root vertex has offspring distribution given by $Z+1$ and each of its children evolves according to $T_{\text{GW}}$ (see \cite[Example 1.1]{AL_unimodular}, \cite{LPP_GW} or \cite[Section 5.1]{curiennotes}). This is known as the \emph{augmented Galton-Watson tree}, and let us denote this by $(T_{\text{AGW}}, \rho)$, we inverse bias by the degree of $\rho$ to obtain the \emph{unimodular Galton-Watson tree}, denoted by $(T_{\text{UGW}}, \rho)$. We do not need the specifics of the definitions here, except the fact that $T_{\text{AGW}}$ and $T_{\text{UGW}}$ are equivalent as probability measures on rooted graphs, so we refer the reader to \cite[Section 5.1]{curiennotes} for detailed definitions and references.
\begin{thm}\label{thm:transient_GW}
On an almost sure sample of supercritical $(T_{\text{GW}}, \rho)$ or $(T_{\text{UGW}}, \rho)$ as above, the LCRW is transient from every vertex almost surely.
\end{thm}
\begin{remark}\label{wlgp_0}
The assumption $\P(Z=0)  = 0$ in \Cref{thm:transient_GW} is without any loss of generality. Indeed, if  $\P(Z=0)>0$, then the vertices with an infinite line of descent form a branching process with zero probability of producing no offspring (see \cite{AN_branching}, Chapter 12). Furthermore, for every vertex $v$ which has a finite line of descent, let $T_v$ be the set of all its descendants, which is a finite tree by definition. Then for any wired spanning arborescence containing $T_v$, the oriented edges on $T_v$ are deterministic: it forms the only spanning arborescence of the tree spanned by the vertices of $T_v$ rooted at $v$.
\end{remark}
We finish with a theorem analogous to \Cref{cor:unimod_tree}, but for Galton-Watson trees.

\begin{thm}\label{thm:end_GW}
    Let $(T, \rho)$ be either $(T_{\text{GW}}, \rho)$ or $(T_{\text{UGW}}, \rho)$ as above. Then the wired MSA limit for i.i.d.\ Exponential $(1)$ weights on $\vec E$ where $T = (V,E)$ exists almost surely, the limit has infinitely many infinite components, and each component is one-ended almost surely.
\end{thm}

\subsection{Main techniques and perspectives}\label{sec:techniques}

It is often useful to find some tractable algorithm to sample the random object we wish to study. As mentioned before, Kruskal and Prim's algorithms, and the related invasion percolation models have been extremely fruitful in studying the MST. The known algorithm to sample the MSA is a recursive one, which goes back to Chu, Liu, Edmonds and Bock \cite{bock,Chu-Liu,Edmonds} (which we abbreviate as CLEB algorithm in this article).

Recall that in Kruskal's algorithm, one simply orders the edges according to the weights. Starting from the empty graph,  we add edges one by one from the lowest to the highest weight, and if such an addition creates a cycle we do not add that edge. The resulting tree at the end of the process is the minimal spanning tree.

Now let us briefly describe the CLEB algorithm. For every vertex, reveal the minimal weight outgoing edge, and then subtract that weight from the remaining outgoing edges. If there are no cycles then we have actually revealed the MSA. If there is a cycle, `contract' it into a single vertex and remove all the self-loops thus created. Now recursively apply the algorithm to this contracted graph.

Note that Kruskal's algorithm is a function of only the order of the edges according to their weights, and not their actual value. Consequently, its law is independent of the law of the weights. On the other hand, it is unclear from the CLEB algorithm whether the law of the resulting MSA depends on the weight distribution or not (primarily because of the weight subtraction step). In fact, we prove in \Cref{sec:basic} that the law is indeed dependent on the specific distribution of the weights. This indicates that the MSA is more unwieldy than the MST, and indeed it is reflected in the nature of the results we can prove in this article, at least compared to the results on MST already present in the literature. For example, it is immediate from Kruskal's algorithm that the local weak limit for the MST exists on any infinite locally finite graph, while we do not know the same for the MSA, and only prove partial results in this article.

On the flip side a certain connection with Wilson's algorithm  \cite{wilson_algo}, a popular algorithm to sample the uniform spanning tree emerges. Indeed, it can be shown (see \Cref{sec:ECL_finite}) that the revealment of the minimal weight edges can be done step by step and in any order, and the resulting arborescence does not depend on the order in which edges are revealed and the cycles are contracted in a certain precise sense. This is very similar to Wilson's algorithm, where one reveals simple random walk steps, and if a loop is created it is `erased' rather than contracted. In fact, Wilson's algorithm works in the oriented and weighted setup as well.
Indeed, we make this connection precise in \Cref{sec:wilson}. Briefly speaking, we prove that if we consider weights $e^{-\beta U_{\vec e}}$ and let $\beta \to \infty$, then the loops erased during Wilson's algorithm are essentially the same as those contracted in CLEB algorithm if one ignores the multiplicities. Roughly, every loop is traversed by the random walk `infinitely many times' as $\beta \to \infty$ rather than being contracted. A detailed statement can be found in \Cref{prop:wilson_cleb}.

We also prove in \Cref{sec:invasion} that in the unoriented setup, the CLEB algorithm is roughly the same as invasion percolation, and hence it can be seen as an extension of Prim's algorithm in the oriented setup.

The proof of \Cref{thm:convergence}  relies on a certain comparison with hitting probabilities of simple random walk, see \Cref{prop:hitting_prop}. The core of the argument relies on the fact that contracting edges decrease effective resistances for simple random walk. The proof of \Cref{thm:end} relies on some general results coming from the theory of unimodular random graphs combined with a certain stability result of the MSA under perturbations (stated in \Cref{lem:perturbation,lem:perturbation2}). Let us mention that such stability results are often key in understanding the geometry of the optimizers (ground states) in various disordered systems. (see, e.g., \cite{NS_ground_spinglass,majority_dynamics_2020}).


\paragraph{Organization:} In \Cref{sec:basic} we prove that the law of the MSA does depend on the specific distribution of the weights. In \Cref{sec:CLEB_finite} we describe the Chu, Liu, Edmonds and Bock (CLEB) algorithm, based on which we define the `CLEB walk' and describe the CLEB walk algorithm. In \Cref{sec:CLEB_walk} we describe some tractable ways to recover some information about the MSA from a piece of CLEB walk. In \Cref{sec:invasion} we prove that the CLEB walk started from a vertex is the same as the Invasion tree in the unoriented setup (this section is independent of the rest of the article).

In \Cref{sec:ECL_infinite}, we extend the CLEB walk in the infinite setup and define notions of transience or recurrence for this process. In \Cref{sec:LCRW} we introduce the Loop contracting random walk, which is essentially the CLEB walk for i.i.d.\ Exponential weights, and enjoys certain Markovianity exploiting the memoryless property. In \Cref{sec:LCRW_SRW} we show how the Loop contracting random walk compares with the simple random walk in a tree and prove \Cref{thm:convergence,thm:LCRW_transient,thm:transient_GW}. In \Cref{sec:end} we prove \Cref{thm:end,thm:end_GW}. In \Cref{sec:wilson}, we show that the CLEB walk can be seen as a certain limit of Wilson's algorithm. Finally, in \Cref{sec:open} we finish with some open questions and simulations.

\section*{Summary of notations}
\begin{itemize}
\item $G = (V, E)$, when finite is assumed to be a connected multigraph with no self-loops. $\vec E$ denotes the set of oriented edges by including both orientations of every edge in $E$. Sometimes these oriented edges come with weights $(U_{\vec e})_{\vec e  \in \vec E}$.  $G=(V,E)$ when infinite is additionally assumed to be locally finite.

\item $\vec e$ is an oriented edge. 

\item For any $S \subset \vec E$, let $V(S)$ denote the set of endpoints of $S$.


\item The head of $\vec e$ is denoted $\vec e_+$ and the tail is denoted $\vec e_-$. 

\item Given a vertex $v$, an \emph{outgoing edge} from a vertex $v$ is an oriented edge $\vec e$ with $\vec e_- = v$.

\item $\cO(v)$: the set of outgoing edges from $v$.
\item Cycle is defined to be a collection of oriented edges $(\vec e_1, \ldots, \vec e_k)$ so that $(\vec e_i)_+  = (\vec e_{i+1})_-$ for all $1 \le i \le k-1$ and $(\vec e_k)_+ = (\vec e_0)_-$. 
\item $\mathfrak F_T(v)$: the future of $ v$ in $T$ (\Cref{sec:intro}).
\item Contraction and uncontraction procedures of a cycle are precisely defined in \Cref{sec:ECL_finite}.
\item $Y_i$: set of edges exposed in the original CLEB algorithm in step $i$. (\Cref{sec:CLEB_finite})
\item $S_i$: set of edges exposed in step $i$ of a CLEB walk (\Cref{sec:CLEB_finite}). 
\item $\tau_{N_1}$: The last time when the path obtained by contracting cycles during a CLEB walk started from $x$ is a singleton (\Cref{sec:CLEB_walk}).
\item $(G_i = (V_i, E_i))_{i \ge 0}$: the sequence of graphs obtained by performing the CLEB walk algorithm. $(U_{i, \vec e})_{\vec e \in \vec E}$ denotes the modified weights in $\vec E_i$. (\Cref{sec:CLEB_finite})
\item $P_i  = S_i \cap \vec E_i$: the path consisting of edges which are not contracted (\Cref{sec:CLEB_finite}). See also analogous definition in the context of loop contracting random walk (\Cref{sec:LCRW}).
\end{itemize}

\paragraph{Acknowledgement:} We thank Omer Angel, Louigi Addario-Berry, Russell Lyons, Yinon Spinka and Benoit Laslier for several stimulating discussions. We thank Louigi further for inspiring \Cref{sec:invasion,sec:wilson}.
\section{Dependence on distribution}\label{sec:basic}

The minimum spanning tree (MST) enjoys a nice property that the law is independent of the distribution of the weights chosen as long as they are continuous: indeed it is clear from Kruskal's algorithm \cite{kruskal1956shortest} that the MST is a function of the ordering of the edges, as long as there are no ties. However, the MSA does not depend only on the rankings of the oriented edges, but also on the actual weights, as is clear from \Cref{fig:counter_weight}, so it is unclear whether there is a `linear' way to sample the edges of the MSA.
\begin{figure}[ht]
\centering
\includegraphics[width = 0.5\textwidth]{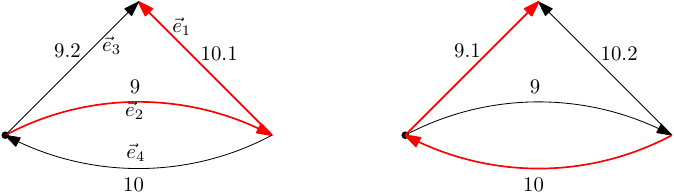}
\caption{The weights are depicted by the numbers. The ranking of the weights are the same, but have different MSAs as drawn in red}\label{fig:counter_weight}
\end{figure}

We now claim that the law of the MSA depends on the specific distribution of the weights as well. It turns out that a cycle graph is not complex enough to make this distinction, it can be shown actually that in any cycle graph, the distribution of the MSA is independent of the law (but we do not pursue this here). To witness the distinction, consider \Cref{fig:counter_weight1}.
\begin{figure}[ht]
\centering
\includegraphics[width = 0.5\textwidth]{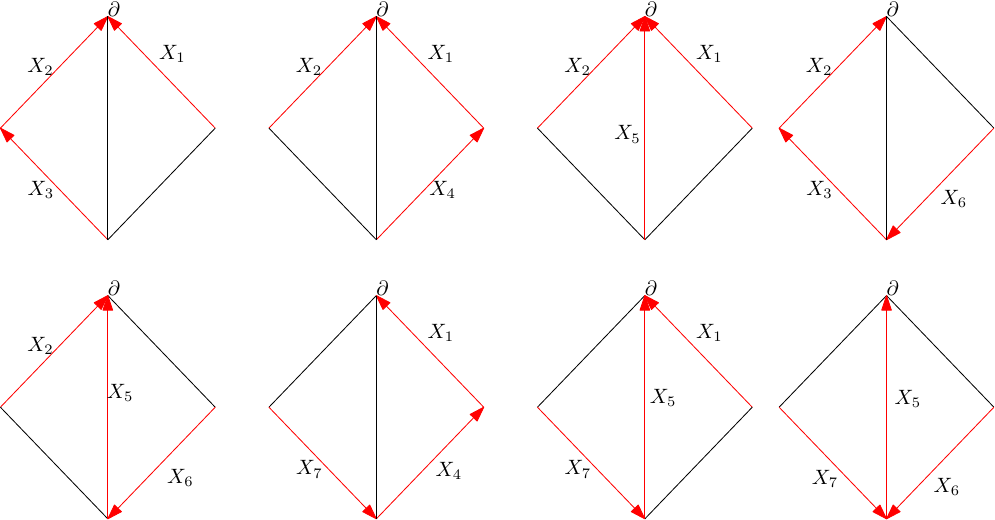}
\caption{All the spanning arborescences are marked red.}\label{fig:counter_weight1}
\end{figure}

\begin{proposition}
The distribution of the MSA with i.i.d.\ Exponential $(1)$ weights and with i.i.d.\ Unif $[0,1]$ weights are different on the graph of \Cref{fig:counter_weight1}.\end{proposition}

\begin{proof}
Let us calculate the probability that the top left figure in \Cref{fig:counter_weight1} is the MSA. This leads to the following probability computation (ignoring the redundant inequalities):
\begin{equation}
\P(X_3<X_4,X_3<X_5,X_1<X_6,X_2+X_3 < X_4+X_7,X_2+X_3<X_5+X_7) \label{eq:the_prob}
\end{equation}
We can immediately get rid of the $\{X_1<X_6\}$ part since it is independent of the rest, this contributes a factor of $1/2$ irrespective of the distribution. 

Let us now assume the distributions are Exponential $(1)$. Condition on $X_3<X_4 \wedge X_5$. Writing $X_4 = X_3+Z_4$ and  $X_5 = X_3+Z_5$ and using the memoryless property of exponential, we arrive at
\begin{multline*}
\P(X_2<X_7+Z_4,X_2<X_7+Z_5)\P(X_3 < X_4 \wedge X_5)   \\= \frac13(1-2\P(X_2>X_7+Z_4,X_2<X_7+Z_5) - \P(X_2>X_7+Z_4,X_2>X_7+Z_5))
\end{multline*}
where $Z_4,Z_5$ are again i.i.d.\ Exponential $(1)$ independent of the $X_i$s.
To compute the above probabilities, condition on $\{X_2>X_7\}$. Again writing memoryless property of Exponential and writing $X_2 = X_7+Z_2$,
\begin{align*}
\P(X_2>X_7+Z_4,X_2<X_7+Z_5) = \P(Z_4<Z_2<Z_5) & = \frac16.\\
\P(X_2>X_7+Z_4,X_2>X_7+Z_5)) = \P(Z_2 > Z_4 \vee Z_5) = \frac13.
\end{align*}
Overall, we get that \eqref{eq:the_prob} is $1/9$.

For Uniforms, by symmetry of $X_4,X_5$ we first reduce \eqref{eq:the_prob} to
$$
\P(X_3<X_4<X_5,X_2+X_3 < X_4+X_7).
$$
Next, we condition on $X_3 = u, X_4 = v$ and use the fact that for $a \in [0,1]$, $\P(Z<W+a) = a+ (1-a^2)/2$ for $Z,W \sim$ i.i.d.\ Uniform$[0,1]$ to reduce the above probability to the double integral
\begin{equation}
\int_{0<u<v<1} \left((v-u) + \frac12(1-(v-u)^2 )\right)(1-v)dudv.
\end{equation}
This is easily computed to be equal to $\frac7{60} >\frac19$.
\end{proof}
We finish this section by mentioning the remarkable result due to Frieze who prove that the total weight of the minimum spanning tree in the complete graph is approximately $\zeta(3)$  for a general class of distributions \cite{Frieze_MST_weight}. Later with Tkocz \cite{FT_complete_cost}, he proved that in the case of MSA with powers of uniform weights, this quantity is approximately 1.

\section{The Chu-Liu-Edmonds-Bock (CLEB) algorithm}\label{sec:ECL_finite}

We extensively use the Chu-Liu-Edmonds-Bock (CLEB) algorithm to sample the MSA, which is independently attributed to Chu and Liu \cite{Chu-Liu}, Edmonds \cite{Edmonds} and Bock \cite{bock} (see also Karp \cite{Karp}). We describe the original algorithm first and then describe certain variants which are useful for this paper. 

In what follows we need to work with a directed multi-graph, which we define precisely now. A directed multigraph is a collection of vertices $V$ and a collection of oriented edges $\vec E$ where the tail and the head of each $\vec e \in \vec E$ are in $V$. There could be multiple edges with the same head and tail (which we call \textbf{parallel edges}). We further assume that no oriented edge has the same head and tail (i.e., there are no self-loops). An \textbf{oriented path} is a sequence of distinct oriented edges $(\vec e_0, \vec e_1, \ldots, \vec e_{k})$ such that the tail of $\vec e_j$ is the head of $\vec e_{j+1}$ for $0 \le j \le k-1 $, and such an oriented path is a \textbf{cycle} if the tail of $\vec e_k$ is the head of $\vec e_0$. The endpoints of these edges are called the vertices of the path or cycle.

Before describing the CLEB algorithm, let us define two graph operations that will be used throughout the paper.

\begin{itemize}

\item {\textbf{ Contraction.}} Given a directed multigraph $G = (V, \vec E)$ with boundary $\partial$ and a cycle $C$ in $G$ not passing through $\partial$, this operation {\bf contracts  the cycle $C$} and outputs another directed multigraph $G' = (V', \vec E')$ with boundary $ \partial$, which we will denote by
\[ \Psic ((G, \partial), C) = (G', \partial). \]
To obtain $G'$ from $G$, we remove all vertices in $C$ and add a new vertex $v_C$ (which we can think of as the ``contracted vertex" formed by identifying all vertices in $C$ with $v_C$) while keeping the rest of the vertices unchanged (including the boundary $\partial$). In $G'$, we retain all the oriented edges of $G$ whose both endpoints are outside $C$. We then remove all the oriented edges with both endpoints in $C$. Finally, for every oriented edge in $G$ with exactly one endpoint in $C$, we replace it with the following oriented edge in $G'$.  Every edge  in $G$ with head (resp.\ tail) $u \not \in C $ and the tail (resp.\ head) in $C$,  replace it by an edge $\vec e$ in $G'$  with $\vec e_+ = u$ (resp.\ $\vec e_- = u$) and $\vec e_- = v_C$ (resp.\ $\vec e_+ = v_C$).

We remark the removal of all the edges with both endpoints in $C$ ensures that there is no self-loop in $G'$.
Let us also remark there is a natural bijection between $\vec E'$  and the edges in $\vec E$ with at most one endpoint in $C$, which yields an embedding of $\vec E'$ inside $\vec E$.   Hence we can (and will) think of the elements $\vec E'$ as elements of $\vec E$ as well. However, the tail and head of an oriented edge in $\vec E'$  may be different from those of its image in $\vec E$. Also, given any collection of edges $F \subseteq \vec E$, each edge from $F$ whose both endpoints do not belong to the vertices of $C$ can be uniquely mapped to an edge in $\vec E'$. We will denote the resulting collection of edges in $\vec E'$  obtained from this mapping by $F \cap \vec E'$. 

\begin{figure}[ht]
    \centering
    \includegraphics[scale = 0.7]{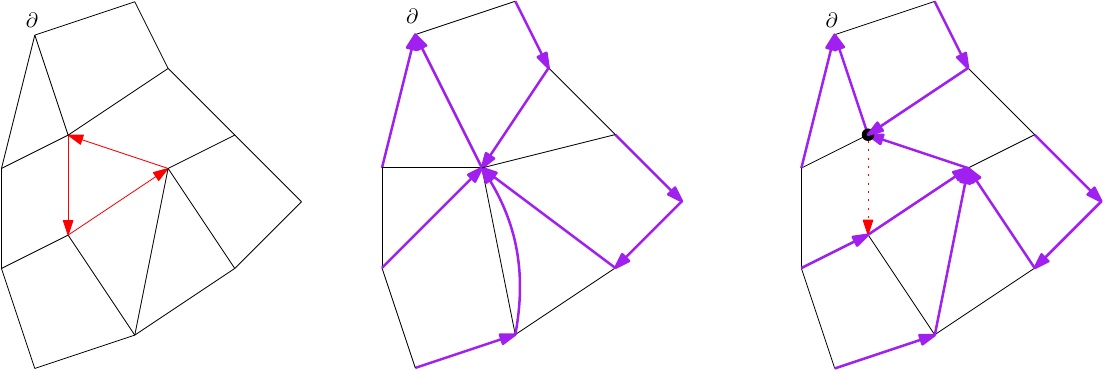}
    \caption{The contraction and uncontraction procedure. Left: The graph $G$ and the cycle $C$ is marked in red. Centre: $(G', \partial) = \Psic ((G, \partial), C)$ with a spanning arborescence $T'$ in purple. Right: The graph $G$ with spanning arborescence $T:=\Psiuc((G, \partial), (G', \partial), C,T')$. The edge in $C$ which is deleted is dotted.}
    \label{fig:enter-label}
\end{figure}

\item {\textbf{ Uncontraction.}} Suppose we have contracted a cycle $C$ in $(G, \partial)$ to obtain $(G', \partial)$, i.e., 
$\Psic ((G, \partial), C) = (G', \partial)$.  Let $T'$ be a spanning arborescence $(G', \partial)$. Given this data, the uncontraction operation yields a spanning arborescence $T$ in  $(G, \partial)$ by  {\bf uncontracting  the cycle $C$}, which we will denote by 
\[ \Psiuc ((G, \partial), (G', \partial), C, T') = T. \]

Let $T^\star \subseteq \vec E$ be the union of the edges in $C$ and the image of the edges of $T'$ in $G$. Note that $T^\star$   almost forms a spanning arborescence in $(G, \partial)$ except for one blemish. There is a unique vertex of $C$ which has exactly two outgoing edges from $T^\star$, one in $C$, say $\vec e_C$, and the other in the image of the edges in $T'$, say $\vec e_{T'}$. Define $T = T^\star \setminus \vec e_C$. It is easy to see that $T$ is a spanning arborescence of $(G, \partial)$.

\end{itemize}

\subsection{CLEB algorithm (and variants) in finite graphs}\label{sec:CLEB_finite}

  Suppose we are given a finite, connected directed multi-graph $G= (V,\vec E)$, a collection of non-negative weights $(U_{\vec e})_{\vec e \in \vec E}$, and a boundary vertex $\partial$.  

For any countable set $N$, we say a collection of weights $(W_x)_{x \in N}$ is \textbf{generic} if 
\begin{equation}
\sum_{x \in S}  n_{x} W_{x} \neq 0, \text{ for all finite $S \subset N$,} \label{eq:linear_comb}
\end{equation}
and any choice of coefficients $n_{x} \in \Z$, $x \in S$ such that $n_x \neq 0$ for some $x \in S$.
Although later we work with random weights, for now, we only assume that the collection of weights $(U_{\vec e})_{\vec e \in \vec E}$ is generic. 
Clearly, if the weight collection is i.i.d.\  from a continuous distribution, then it is generic. 
It is also clear that for a fixed boundary vertex $\partial$, a generic collection of weights produces a unique MSA for $(G, \partial)$.
\begin{figure}[t]
\centering
\includegraphics[width = 0.8\textwidth]{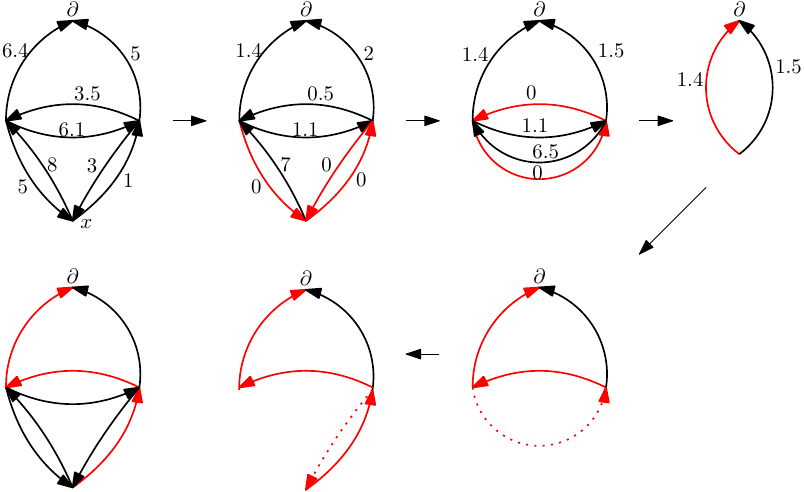}
\caption{The CLEB algorithm. The final output is in the bottom left figure.}\label{fig:CLEB_algo}
\end{figure}
\paragraph{The original CLEB algorithm}
We are given as input a graph $G = (V, \vec E)$, a boundary vertex $\partial$, and a collection of generic non-negative weights $(U_{\vec e})_{\vec e \in \vec E}$. Our goal is to output the MSA $T$ of $(G, \partial)$. We recommend referring to \Cref{fig:CLEB_algo} while reading the following description.
\medskip

 {\em Forward direction: Contracting cycles iteratively}.  Let us inductively define the sequence 
$$
J_i := (V_i, \vec E_i, Y_i, (U_{i, \vec e})_{\vec e \in \vec E_i})_{i \ge 0},
$$
where $G_i := (V_i, \vec E_i)$ is a directed multigraph  with weights $(U_{i, \vec e})_{\vec e \in \vec E_i}$ and  $Y_i  \subset \vec E$ is an increasing family  of oriented edges.  Let $\cO(v)$ denote the set of outgoing edges from $v$, i.e., $\cO(v) = \{\vec e: \vec e_- = v\}$. We start with $G_0 = G,Y_0 = \emptyset$ and $U_{0, \vec e} = U_{\vec e}$ for all $\vec e \in \vec E$.

\begin{enumerate}[(a)]
\item  For $i=1$,   take $V_1 = V$ and  $\vec E_1 = \vec E$.  For every $v \neq \partial$ in $G$, let $\pi_v$ denote $\min_{\vec e \in \cO(v)} U_{\vec e}$. Define the weights $U_{1,\vec e} := U_{\vec e} - \pi_{\vec e_-}$ for all $\vec e$ such that $\vec e_- \neq \partial$. Note that since all weights are generic, and in particular distinct, every vertex except $\partial$ has a unique zero-weight outgoing edge. Let $Y_1$ denote this collection. This completes the definition of $J_1$. 

If the zero-weight edges of $G_1$, or equivalently the edges in $Y_1$, have no cycles, we stop.
\item Otherwise, suppose we have defined $J_{i-1}$ for some $i \ge 2$ and 
there is at least one cycle consisting of zero-weight edges in the multigraph $G_{i-1}= (V_{i-1}, \vec E_{i-1})$ with weights $(U_{i-1, \vec e})_{\vec e \in \vec E_{i-1}}$.

Select one such cycle $C$ {\em arbitrarily}. Contract the cycle $C$ to obtain $G_i = (V_i, \vec E_i)$, that is,  $\Psic((G_{i-1}, \partial), C )= (G_{i}, \partial)$.
Let $v_C$ be the contracted vertex in $V_i$.
We update the weights as follows. Let $\pi_{v_C} =\min_{\vec e \in \cO(v_C)} U_{i-1,\vec e}$  and let $U_{i,\vec e} = U_{i-1,\vec e}  - \pi_{v_C}$ for all $\vec e \in \cO( v_C)$ and $U_{i, \vec e} = U_{i-1,\vec e}$ otherwise. So, we obtain a new multigraph    $G_i =(V_{i}, \vec E_{i})$ with weights $(U_{i, \vec e})_{\vec e \in \vec E_{i}}$.

As remarked before, we can view $\vec E_i$ as a subset of  $\vec E_{i-1}$.  In particular, we can view $\vec E_i$ as a  subset of $\vec E$ as well.

\item Let  $\vec e \in \vec E_i$ be the unique zero-weight edge outgoing from $v_C$ in $(V_i, \vec E_i, (U_{i, \vec e})_{\vec e \in \vec E_i})$. In other words, $\vec e$ is the unique outgoing edge from $v_C$ with minimum $U_{i-1}$-weight. Let $Y_i = Y_{i-1} \cup \{ \vec e\}$, where we abuse notation by identifying $\vec e \in \vec E_i$ by its unique image in $\vec E$. 
\item We terminate at time step $\tau$ when we are left with no zero-weight cycle in $G_\tau=(V_\tau, \vec E_\tau)$ with weights $(U_{\tau, \vec e})_{\vec e \in \vec E_\tau}$. Note that $\tau$ is necessarily finite since the size of $V_i$ is strictly decreasing. 
\end{enumerate}
We note that at every step $i$, the collection of weights $(U_{i, \vec e})_{\vec e \in \vec E_i}$ is generic and 
each vertex of $G_i$ except $\partial$  has exactly one outgoing edge with zero $U_i$-weight.\\

 {\em Backward direction: Uncontracting cycles iteratively}.  Once the algorithm terminates, let $T_\tau$ be the set of zero-weight edges in $G_\tau$. Since each vertex barring $\partial$ has exactly one outgoing edge
with zero $U_\tau$-weight and since $T_\tau$ does not contain any cycle by definition, $T_\tau$ is a spanning arborescence of $(G_\tau, \partial)$.  

We then inductively find the spanning arborescence  $T_i$ of $(G_i, \partial)$ for $i = \tau, \tau - 1, \ldots, 1$ by `uncontracting' the cycles one at a time chosen in the backward direction. More precisely, if we perform the contraction $\Psic((G_{i-1}, \partial), C_{i-1} )= (G_{i}, \partial)$ in the $i$-th time step in the forward direction, then we obtain  the spanning arborescence $ T_{i-1} =  \Psiuc ((G_{i-1}, \partial), (G_i, \partial), C_{i-1}, T_i)$ from $T_i$ while uncontracting $C_{i-1}$.  

Finally, we output $T^* = T_1$, which will be a spanning arborescence of $(G, \partial)$.

\begin{proof}[Proof that the original CLEB algorithm produces the MSA]  We claim $T^*$ to be the MSA of $(G, \partial)$. A proof can be found in  \cite{bock,Edmonds,Chu-Liu} or \cite{Karp}, but for the sake of completeness we provide a quick proof here.
In fact, we claim that at each step,  the spanning arborescence $T_i$ is the MSA of $(G_i, \partial)$ with weights $(U_{i, \vec e})_{\vec e \in \vec E_{i}}$.
The proof proceeds by induction on the step number $i $, starting from $\tau$ and running backward in time. Note $T_{\tau}$ is the MSA for $(G_\tau, \partial)$ with weights $(U_{\tau, \vec e})_{\vec e \in \vec E_{\tau}}$  since $T_\tau$ has weight 0 and any edge not in $T_\tau$ has strictly positive weight by the generic assumption. So the base case is true.

Assume now that $T_{i}$ is the MSA for $(G_{i}, \partial)$ with weights $(U_{i,\vec e})_{\vec e \in \vec E_i}$. 
 For a spanning arborescence $T$ of $G = (V, \vec E)$ with the boundary $\partial$ and the weights $(U_{\vec e})_{\vec e \in \vec E}$, let $w(G, (U_{\vec e})_{\vec e \in \vec E}, T)$
denote the weight of $T$.  Let  $T^*_{i-1}$ be the MSA for $G_{i-1}$ with weights $(U_{i-1, \vec e})_{\vec e \in \vec E_{i}}$ and let $C_{i-1}$ be the cycle in $G_{i-1}$ which we contracted  to obtain $G_{i}$.  Recall the  natural embedding of $\vec E_{i}$ into $\vec E_{i-1}$ and let $S$ be the collection of edges in $\vec E_{i}$ whose images in  $\vec E_{i-1}$ belong to $T^*_{i-1}$. Note that $S$ consists of oriented edges such that every non-boundary vertex except $v_{C_{i-1}}$ has exactly one outgoing edge, while $v_{C_{i-1}}$ may potentially have several outgoing edges. Observe  that in $G_{i-1}$, there is at least one vertex $v$ in $C$ with an oriented path in $T^*_{i-1}$ starting from $v$, ending in $\partial$ and containing no other vertex from $C$. For example, we can take the path in $T^*_{i-1}$ started from any vertex in $C$ to $\partial$, which might hit $C$ several times, and take the portion of this path from the last hit of $C$ to $\partial$. Take the outgoing edge from one such $v$ in $T^*_{i-1}$, keep it (or rather its copy in $G_{i}$) in $S$ and remove every other outgoing edge from $v_C$ which came from $T^*_{i-1}$. Call the resulting collection of edges $T'$, which can easily be seen to be a spanning arborescence of $G_{i}$. Moreover, 
\begin{equation}\label{eq:wt_bd1}
 w(G_{i-1}, (U_{i-1, \vec e})_{\vec e \in \vec E_{i-1}}, T^*_{i-1}) \ge w(G_i, (U_{i-1,\vec e})_{\vec e \in \vec E_i}, T')  \ge w(G_i, (U_{i-1,\vec e})_{\vec e \in \vec E_i}, T_i), 
 \end{equation}
since  $T'$ is constructed from $T^*_{i-1}$ by potentially deleting some outgoing edges from $v_{C_{i-1}}$ in $G_i$. The second inequality follows from the facts that $T_i$ is the MSA for $(G_{i}, \partial)$ with weights $(U_{i,\vec e})_{\vec e \in \vec E_i}$ and that subtracting the same quantity from the weights of all outgoing edges of a vertex does not change the MSA. 

Now observe that since $T_{i-1}$ is obtained from $T_i$  by uncontracting the cycle $C_{i-1}$, the extra edges in $T_{i-1}$  belong to the cycle $C_{i-1}$, which have zero $U_{i-1}$-weights. Therefore, 
\begin{equation}\label{eq:wt_bd2}
 w(G_{i}, (U_{i-1, \vec e})_{\vec e \in \vec E_{i}}, T_i ) = w(G_{i-1}, (U_{i-1,\vec e})_{\vec e \in \vec E_{i-1}}, T_{i-1})  \ge w(G_{i-1}, (U_{i-1,\vec e})_{\vec e \in \vec E_{i-1}}, T^*_{i-1}). 
  \end{equation}
  Combining \eqref{eq:wt_bd1} and \eqref{eq:wt_bd2} yields equality everywhere and by the uniqueness of the MSA for generic weights, it follows that $T_{i-1} = T^*_{i-1}$.   Thus the induction step is complete and so is the proof.
\end{proof}

\bigskip

Note that in the original CLEB algorithm, in each iterative step, we choose to reveal the minimum weight outgoing edges from all the vertices at once, and the choice of the cycle we contract is arbitrary. Instead,  we can choose the vertices one by one from which to reveal the minimum weight edge in any manner we like, and contract a cycle whenever they are created. This idea was originally proposed by Tarjan \cite{tarjan}.   This sequential revealment of edges has certain advantages and it bears similarities with the celebrated  Wilson's algorithm that is used to study uniform spanning trees (more on this in \Cref{sec:wilson}). 

\paragraph{The sequential CLEB algorithm.}  {\em Forward direction.} Let us inductively define a sequence $$H_j := (V_j, \vec E_j, S_j, (U_{j,\vec e})_{\vec e \in \vec E_j})_{j \ge 0}$$ 
where $G_j:= (V_j, \vec E_j)$ is a directed multigraph with weights $(U_{j,\vec e})_{\vec e \in \vec E_j}$ and $S_j$ is a subset of $ \vec E$.  We shall call $S_j$ the set of \textbf{exposed edges} by the algorithm at step $j$. For $j = 0$, we take $G_0 = G$, $S_0 = \emptyset$ and $U_{0,\vec e}$ = $U_{\vec e}$ for $\vec e \in \vec E_0 = \vec E$. 

The definition will be such that $G_{j}$ either is the same as $G_{j-1}$ or is obtained by contracting a single cycle of $G_{j-1}$. Therefore, we can view $\vec E_j \subseteq \vec E_{j-1} \subseteq \vec E$ and $S_j \cap \vec E_j$ as collection of edges in $\vec E_j$.  
At every step $j$, $S_j \cap \vec E_j$ is an arborescence of $G_j$, which is clearly true for $j=0$.   To carry out the inductive definition, we assume that we have generated $H_0, H_1, \ldots, H_{j-1}$ for $j \ge 1$. 

\begin{enumerate}[(a)]
\item {\textsf {(Choosing a vertex.)}} In step $j$, pick a vertex $v$ in $V_{j-1}$ which has no outgoing edge in $S_{j-1} \cap \vec E_{j-1}$. The choice of this vertex may be a deterministic function of $(H_{k})_{k \le j-1}$ and possibly some additional independent source of randomness. In any case, note that $v$ is the boundary vertex of the component of the arborescence $S_{j-1} \cap \vec E_{j-1}$ it is in. If there is no such vertex $v$, stop the algorithm and go to the last step.

\item {\textsf {(Revealing the minimum outgoing edge of that vertex.)}}  Let $\pi_v$ be the minimum weight among the outgoing edges of $v$. Let $U'_{j-1,\vec f} = U_{j-1,\vec f} - \pi_v$ for all $\vec f \in \vec E_{j-1}$ with $\vec f_- = v$ and $U'_{j-1,\vec f} = U_{j-1,\vec f}$ for all other $\vec f \in \vec E_{j-1}$. Let $\vec e \in E_{j-1}$ be the unique oriented edge out of $v$ with zero $U'$-weight, and define $S_j =  S_{j-1} \cup \vec e$. By the inductive assumptions on $ S_{j-1} \cap \vec E_{j-1}$, $S_j  \cap \vec E_{j-1}$ either is a disjoint union of arborescence or contains a single cycle. The first case occurs when the endpoints of $\vec e$ belong to different components of the arborescence $S_{j-1} \cap \vec E_{j-1}$, and the second case occurs when they belong to the same component.

\item  If $S_j  \cap \vec E_{j-1}$ has no cycle, declare $V_{j} = V_{j-1}$, $\vec E_{j} = \vec  E_{j-1} $ and $U_{j,\vec e} = U'_{j-1,\vec e}$ for all $\vec e \in \vec E_j$.  Note that  $S_j \cap \vec E_j$ is trivially an arborescence since 
we are joining two components of $S_{j-1} \cap\vec E_{j-1}$ into a new arborescence $S_j \cap \vec E_j$.

\item If $S_j  \cap \vec E_{j-1}$ has a cycle $C$ (which must be unique), contract the cycle $C$ in $G_{j-1}$  to obtain $G_j = (V_j, \vec E_j)$, that is, $\Psic((G_{j-1}, \partial), C )= (G_{j}, \partial)$.
Then define the new weights as $U_{j,\vec e} = U'_{j-1,\vec e}$ for all $\vec e \in \vec E_j$. Since we remove the cycle from $S_j  \cap \vec E_{j-1}$ via contraction,  $S_j \cap \vec E_j$ becomes again an arborescence.

\end{enumerate}
 {\em Uncontracting cycles iteratively}. Once the algorithm stops, say at step $\vartheta$, the arborescence  $S_\vartheta \cap \vec E_\vartheta$ contains an outgoing edge of each vertex in $V_{\vartheta}$. Hence, it is a spanning arborescence of $(G_\vartheta, \partial)$.  During the algorithm, we also obtain an ordered sequence of cycles $C_0,  C_1, \ldots, C_{k-1}$ which are contracted  (step (d)) at time steps $1 \le i_1 < i_2  < \cdots < i_k \le \vartheta$, i.e., $\Psic((G_{i_{\ell}-1}, \partial), C_{\ell-1} )= (G_{i_\ell}, \partial)$ where  $i_0=0$. We iteratively uncontract the cycles in reverse order in the same way as in the original CLEB algorithm to obtain the spanning arborescences $T_\ell$ of $(G_{i_\ell}, \partial)$ for $\ell = k, k-1, \ldots, 0$. More precisely, we set  $T_k = S_\vartheta \cap \vec E_\vartheta$ and $ T_{\ell -1} =  \Psiuc ((G_{i_{\ell}-1}, \partial), (G_{i_\ell}, \partial), C_{\ell -1}, T_{\ell})$. 
  
Finally, we output $T^* = T_0$, which will be a spanning arborescence of $(G, \partial)$.

\begin{proof}[Proof that the sequential CLEB algorithm produces the MSA]
Define $T_k= S_\vartheta \cap \vec E_\vartheta$ as above, which is a spanning arborescence of $G_\vartheta$.  Since the $U_{\vartheta}$-weight of every edge belonging to  $T_k$  is zero, it is the MSA of $G_\vartheta$ with weights $(U_{\vartheta, \vec e})_{\vec  e \in \vec E_{\vartheta} }$.  By induction, assume that we have computed $T_\ell$,  the MSA  in $G_{i_\ell}$ with weights $U_{i_\ell}$. Define $ T_{\ell -1} =  \Psiuc ((G_{i_{\ell}-1}, \partial), (G_{i_\ell}, \partial), C_{\ell -1}, T_{\ell})$, which is a spanning arborescence.  Since the $U_{i_{\ell}-1}$-weights of the edges of $C_{\ell-1}$ are all zero, the same argument as used in the proof of the validity of the original CLEB algorithm yields that $T_{\ell - 1}$ is the MSA of $G_{i_{\ell}-1}$. For each $j = i_{\ell}-1, i_{\ell}-2, \ldots, i_{\ell-1} +1,  i_{\ell-1}$,  the spanning arborescence $T_{\ell - 1}$ remains the MSA of $G_j = G_{i_{\ell}-1}$ with weights 
$(U_{j, \vec e})_{ \vec e \in \vec E_j}$ since $U_j$-weights can be obtained from $U_{j+1}$-weights by adding the same quantity from all outgoing edges of a certain vertex. This proves the induction hypothesis. 
\end{proof}

\bigskip
We emphasize that in the sequential CLEB algorithm,  the choice of the vertex in step (b) can depend on the edges revealed up to the previous step, and may involve additional randomness as well. This property will be exploited heavily. For future reference, we now define a useful choice of ordering for the sequential CLEB which we will call the \textbf{CLEB walk algorithm} which will be reminiscent of Wilson's algorithm. This idea was first proposed by Gabow, Galil, Spencer and Tarjan \cite{gabow86}. As Wilson's algorithm performs successive loop erased random walks, \textbf{CLEB walk algorithm}  is executed by performing successive \textbf{CLEB walks}. The \textbf{CLEB walk}  can also be thought of as an analogue of the so-called \textbf{invasion percolation}  (see \Cref{sec:invasion}) which is widely used to study the unoriented minimal spanning tree. \\

Suppose we have a directed multigraph $(G, \partial)$ with boundary $\partial$ and with generic weights. The \textbf{CLEB walk} started at a vertex $x$ is generated by running the sequential CLEB algorithm on $(G, \partial)$  with a particular choice of vertices starting with $x$ until a certain stopping time $\upsilon$ which we describe below.  This gives us the process $H_0, H_1, \ldots, H_\upsilon$.

For time step $j=1$, we choose the vertex $v$ (in step (a)) to be $x$. Then inductively in step $j \ge 2$, we choose the vertex as follows: 
\begin{itemize}
\item if a cycle is not contracted in time step $j-1$, i.e., $G_j = G_{j-1}$,  choose the head of the edge in $(S_j \cap \vec E_{j-1})  \setminus (S_{j-1} \cap  \vec E_{j-1}) $. If this head is $\partial$, stop the process.
\item  Otherwise, if a cycle $C$ is contracted in $G_{j-1}$ to produce $G_j$, then choose the vertex $v_C$ which is the new vertex in $G_j$ obtained by contracting the vertices of $C$. 
\end{itemize}

It can be easily seen by induction that   $S_j \cap \vec E_j$ is a single oriented path in $G_j$. Call this path $P_j$ for future reference. Also, we call $S_\upsilon$ to be the set of exposed edges by the CLEB walk. \\


Let us now describe the \textbf{CLEB walk algorithm}.  Recall that for any $S \subseteq \vec E$, let $V(S)$ denote the set of endpoints of $S$.
\begin{itemize}
\item Let us enumerate the vertices $x_1,\ldots, x_n$ of $V \setminus \partial$  in some order. Generate a CLEB walk in $G = (V, \vec E)$ with boundary $\partial$ starting from $x_1$ and let $\Gamma_{x_1}$ be the set of exposed edges. Let $B_{1} = V(\Gamma_{x_1})$. Note that $\partial \in B_{1}$.
\item Having defined $B_{1}, \ldots, B_{j}$,  if $\cup_{1 \le m \le j} B_{m} = V$, stop the process. Otherwise, choose the vertex $x_{i}$ with the smallest index which is not in $\cup_{1 \le m \le j} B_{m}$. Perform the CLEB walk on $G$ 
with boundary $\cup_{ 1 \le m \le j }B_{m} $ starting from $x_i$. Set $B_{j+1}  = V(\Gamma_{x_{i}})$ for this CLEB walk.
\end{itemize}
After this, we uncontract the cycles as in the sequential CLEB algorithm to get a spanning arborescence.
Of course, since the sequential CLEB algorithm produces the MSA, so does the CLEB walk algorithm.

\begin{figure}
\centering
\includegraphics[scale = 0.7]{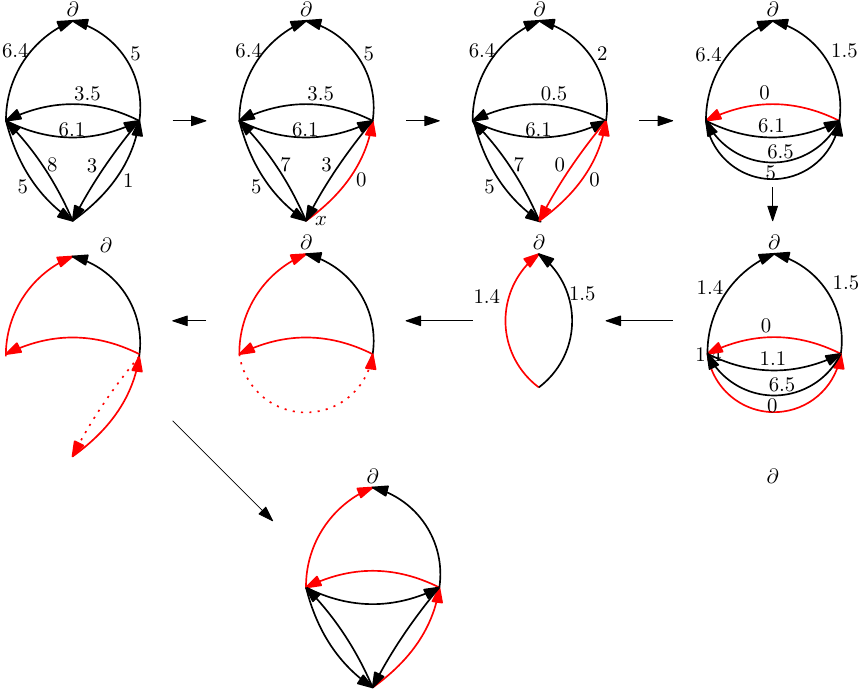}
\caption{CLEB walk started at $x$.}
\end{figure}


Observe that the set of edges exposed in the sequential CLEB algorithm (denoted by $\cup S_i$) is a superset of the MSA, and so is the set of edges exposed in the original CLEB algorithm (denoted by $\cup Y_i$). 
The goal of the next lemma is to prove that these set of edges is the same. This is again analogous to the loop soup representation of the erased loops in Wilson's algorithm.
This lemma will not be used for later results in this article, but given the fact that loop soups have found widespread use since their discovery, we record this result here for future reference. In fact, we will prove a stronger statement in the following \Cref{lem:order_invariance}.
Notice that every edge revealed in the original CLEB algorithm has a certain `level' which essentially
tells us the number of contractions which occurred to create its tail vertex. The same can be defined for the sequential CLEB algorithm.  We will use colors to mark these levels and  then show that the set of colored edges revealed in both algorithms is the same.

Let us now make this more precise.
Recall the notation $Y_i$ from the original CLEB algorithm. Let $F_i =Y_{i} \setminus Y_{i-1} $ be the new edge `exposed' in step $i$  and let $F_1 = Y_1$.  Clearly, $\cup_{i = 1}^\tau Y_i = \cup_{i =1}^\tau F_i$. 
Color edges in $F_1$ with color 1. For $i \ge 2$, note that $F_i$ is the smallest weight outgoing edge from the vertex $v_{C_{i-1}}$ in $G_i$  obtained by contracting a cycle $C_{i-1}$ in $G_{i-1}$. Color $F_i$ with color $k+1$ where $k$ is the largest color in the edges of the cycle $C_{i-1}$.

In the sequential CLEB algorithm, we can also color the edges of $S_j$ as follows. If the edge added in step (b) came out of a vertex in $V$, color it $1$. Otherwise, if it came out of a vertex $v_C$ formed by contracting a cycle $C$, color it $k+1$   where $k$ is the largest color in the edges of $C$.



\begin{lemma}[Equivalence of exposed edges in original and sequential CLEB]
\label{lem:order_invariance}
For generic weights, we have  
\[ \cup_{1 \le i \le \tau} F_i  = \cup_{1 \le j \le \vartheta} S_j \]
 as the set of colored, oriented edges. 

\end{lemma}
\begin{proof}
We will apply induction on the number of edges of the graph and we will prove a slightly general result. We first enrich the coloring scheme as follows: we assume that the color of the minimum weight edge out of each vertex in $V$ is some $c_v \in \N$ fixed from beforehand and then the colors of the remaining edges are defined using the same iterative procedure. We also allow some of the weights to be 0 but with the following restrictions: firstly, there can be at most one outgoing edge from a vertex with weight 0 and secondly, the non-zero weights of the oriented edges form a generic collection. Call this an \textbf{eligible} collection of weights. We prove that with this generalized version of the coloring scheme and weights, the result is true. Of course, once we have this, the lemma follows by choosing $c_v =1$ for all $v \in V$.

 For a graph with a single edge, the result is trivial. Now suppose that for all connected, simple graphs with $k$ edges and choice of colors $(c_v)_{v \in V}$ and a collection of eligible weights the result is true and now we prove the same result for a graph $G$ with $k+1$ edges. Let $\cC:= (C_1,\ldots,C_\tau)$ be the ordered sequence of cycles contracted in the original CLEB algorithm. Let $\cE = (\vec e_1,\vec e_2, \ldots, \vec e_\vartheta)$ be the sequence of oriented edges revealed in the sequential CLEB (observe that in the above notation $\cup_{1 \le i \le j} \vec e_i = S_j$ for all $j$). We also assume that as we expose edges from $\cE$, we also reveal their colors.

Let $i_1$ be the first index when $S_{i_1}$ has a cycle. We can write $S_{i_1} = P \cup C$ where $C$ is a cycle and $P$ is a path that ends at $C$ ($P$ could be empty). Contract the cycle $C$ into a vertex $v_C$ and for notational convenience, denote $G' = G_{i_1}$ and $U_{i_1}  = U'$ and simply observe that $U'$ is an eligible collection of weights. 

Since $C$ is a cycle consisting of minimum weight outgoing edges in $G$, it is present in $Y_1$ in the original CLEB algorithm, and hence it belongs to the sequence $\cC$. Let $C_k =C$ for some $1 \le k \le \tau$. Now the crucial observation is that none of $C_1,\ldots, C_{k-1}$ can have a common vertex with $P \cup C$. Indeed suppose for some $ j \le k$, $C_j$ is the first cycle which has a common vertex $w$ with $P \cup C$. Since in the original CLEB algorithm, $P \cup C \in Y_1$ and none of its edges is contracted before $C_j$, this can only mean $C_j = C_k = C$.

Now we will choose a sequence of edges to reveal for sequential CLEB and a  sequence of cycles to contract for original CLEB algorithms in $G'$ with weights $U'$ and apply the induction hypothesis. First note that $U'$ is an eligible collection of weights as all the edges in $P$ have weight 0. Choose colors $c'_v = c_v$ if $v \notin C$ and choose color $c_{v_C}$ to be $m+1$ if $m$ is the largest color in $C$.  For the sequential CLEB, we simply remove the oriented edges in $C$ from $\cE$ and call this new sequence $\cE'$. This is a valid choice for sequential CLEB as the edges in $P$ have $U'$-weight 0, and the edges chosen after $\vec e_{i_1}$ match with those in $\cE$ and $U_{i_1}  = U'$. For the original CLEB algorithm, observe that $P \cup C \in Y_1$ by definition. We now define a new order of cycles to contract in $G'$ with weights $U'$ by deleting the cycle $C$ from $\cC$. Observe that the latter is a valid choice since 
$\cup_{1 \le i \le k-1} C_i$ does not intersect $P \cup C$. By induction hypothesis, for these choices, the set of colored edges revealed in both algorithms match in $G'$. Thus the sets of colored edges match in $G$ for both the algorithms as the edges of $P \cup C$ are exposed and have the same color in both algorithms. 
\end{proof}

\subsection{CLEB walk and MSA}\label{sec:CLEB_walk}
Given a CLEB walk started from a vertex, what portion of the MSA can be extracted from it? This is the content of the next lemma. Recall that in Wilson's algorithm, the whole uniform spanning tree branch can be obtained by loop erasing the simple random walk started from $x$ until it hits $\partial$. We will prove an analogous result for the CLEB walk algorithm as well.

\begin{figure}[ht]
\centering
\includegraphics[scale = 0.6]{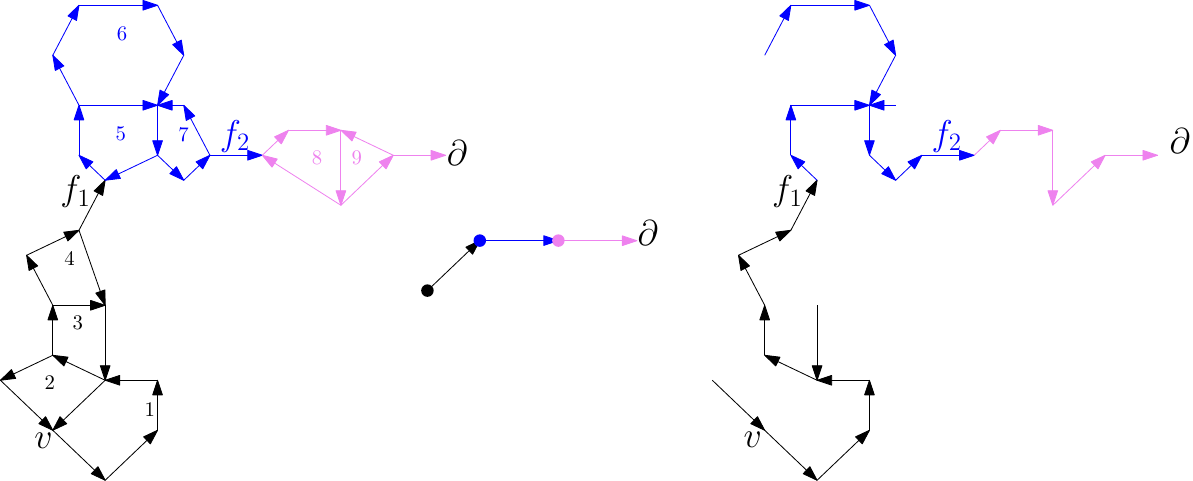}
\caption{Left: The black part is $S_{\tau_{N_1}+1}$, the blue part is  $S_{\tau_{N_2}+1} \setminus S_{\tau_{N_1}+1}$ and the violet part is $S_{\tau_{N_3}+1} \setminus S_{\tau_{N_2}+1}$. The sequence of loops contracted is numbered in order. Centre: The path $P_{\tau}$. Colors indicate the color of the edges which are contracted into a single vertex. Right: $E_{N_1}$ is shown in black, $E_{N_2}$ in blue and $E_{N_3}$ in violet. Together, they form $\Gamma_v$}\label{fig:S}
\end{figure}

Throughout the rest of this subsection, we assume that $G$ is a finite, connected, simple graph with a boundary vertex $\partial$ and $(w_{\vec e})_{\vec e \in \vec E}$ is a collection of generic weights.
Perform the CLEB walk started from $v$. Let $\tau_0=0$ and suppose $\tau_k$ be the $k$th time $S_{\tau_k } \cap \vec E_{\tau_k}$ is empty (i.e., a step when we contract every exposed edge) and let $N_1$ be the largest such $k$. In other words, the oriented edge $\vec f_1 := S_{\tau_{N_1}+1} \setminus S_{\tau_{N_1}}$ is never contracted. Let $ \cL_1$ denote the ordered sequence of loops contracted by time $\tau_{N_1}$. Now uncontract the loops of $\cL_1$ in reverse order starting from the single edge graph  $(V(\vec f_1), \vec f_1)$ with boundary $(\vec f_1)_+$  and $\vec f_1$ being its spanning arborescence.
Let $E_{N_1}$ be the spanning arborescence of the graph spanned by $V(S_{\tau_{N_1}+1})$ with boundary $(\vec f_1)_+$ obtained after the uncontraction process. 

We can now iterate this process. Inductively, let $\tau_{N_{j-1}+k}$ be the $k$th time after $N_{j-1}$ such that $$S_{\tau_{N_{j-1} + k}} \cap \vec E_{\tau_{N_{j-1}+k}} = S_{\tau_{N_{j-1}+1}} \cap \vec E_{\tau_{N_{j-1}+1}}.$$ Let $N_j$ be such that  $N_{j} - N_{j-1}$ be the largest such $k $. Let $ \vec f_j := S_{\tau_{N_j}+1} \setminus S_{\tau_{N_j}}$; and $E_{N_j}$ is the spanning arborescence of the graph spanned by $V(S_{\tau_{N_j}+1})$ with boundary $(\vec f_j)_+$ obtained by uncontracting the ordered set of cycles $\cL_j$ formed between steps $\tau_{N_{j-1}}$ and $\tau_{N_{j}}$ (starting with the single edge graph $(V(\vec f_j), \vec f_j)$ and its spanning arborescence $\vec f_j$). Let $\tau_\partial$ be the smallest $t$ such that the head of $f_t$ is $\partial$.  Let 
\[ \Gamma_v := {\textstyle \bigcup_{1 \le j \le \tau_\partial}}\big( E_{N_j} \cup \{ f_j\} \big)  \ \ \text{ and }  \ \ B_v = V(\Gamma_v).\]
 Observe that by definition, $\Gamma_v$ is a spanning arborescence of $(V(S_{N_{\tau_\partial}+1}), S_{N_{\tau_\partial}+1})$ with boundary~$\partial$.
\begin{lemma}\label{lem:recovery}
The following is true:
\begin{enumerate}
\item[(a)] The outgoing edge from $v$ in the MSA is in $S_{\tau_{N_1} +1}$,
\item[(b)] $\Gamma_v$ contains only one incoming edge into $\partial$, 
\item[(c)] every element of $\Gamma_v$ is in the MSA.
\end{enumerate}
 \end{lemma}
\begin{proof}
Observe that after sampling $S_{\tau_{N_1} + 1}$, we can complete the CLEB walk algorithm, and by definition, none of the cycles contracted after step $\tau_{N_1}+1$ can involve any vertex of $V(S_{\tau_{N_1+1}})$.
  Thus when we iteratively uncontract the cycles contracted after $\tau_{N_1 +1}$, no edge is deleted from $\cL_1$. Thus by definition, the loops in $\cL_1$ are uncontracted in the CLEB walk algorithm in the same order. Thus $E_{N_1}$ is a subset of the edges of the MSA. Since $E_{N_1}$ is a spanning arborescence of the graph spanned by $V(S_{\tau_{N_1}+1})$ and it is a subset of $S_{\tau_{N_1}+1}$, it contains an outgoing edge from  $v$, which is therefore in the MSA as well. This completes the proof of part (a).

  Part (b) is immediate since the set of edges revealed in   $S_{N_{\tau_\partial}+1}$ has only one incoming edge into~$\partial$.

  For part (c) one can iterate the same argument as in part (a) More precisely, we can complete the CLEB walk algorithm (in some prescribed order) after $N_{\tau_\partial} +1$ and since none of the cycles formed after $N_{\tau_\partial} +1$ can involve any vertex in $B_v$, we conclude that the uncontraction procedure in the CLEB algorithm uncontracts the cycles in $(\cL_1,\cL_2,\ldots, \cL_{\tau_\partial})$ in reverse order. This last uncontraction produces $\Gamma_v$ by definition and thus $\Gamma_v$ is a subset of the MSA, thereby concluding the proof.
    \end{proof}


We now record an immediate consequence of \Cref{lem:recovery} regarding the connectivity profile of $T$. Given a spanning arborescence $T$, let us introduce the notation $\mathfrak F_T(v)$ to be the \textbf{future} of $v$ in $T$, which we recall is the oriented path in $T$ started from $v$ until it hits $\partial $. Notice that in an arborescence, $\mathfrak F_T(u)$ and $\mathfrak F_T(v)$ merge into a single path leading to $\partial$ and let us consider the first vertex at which they merge.
Let us call this vertex $u \wedge v$, noting that $u \wedge v = u$ or $u \wedge v = v$ or $u \wedge v = \partial$ are all possible. Now introduce the notation
$$
u \stackrel{T}{\longleftrightarrow} v \text{ if }u \wedge v \neq \partial.
$$ 
\begin{lemma}\label{lem:connection}
Using the above notations, for the MSA T, $B_u   \cap B_v \neq \{\partial\}$ if and only if $u  \stackrel{T}{\longleftrightarrow} v$.
\end{lemma}
\begin{proof}
Suppose $B_u$ is sampled first. Now perform the CLEB algorithm started from $v$ with boundary $B_u$ and call the set of exposed edges $\tilde B_v$.  Suppose the head of the last exposed edge in $\tilde B_v$ is $y \in B_u$ and let $\Gamma_u,\tilde \Gamma_v$ be as in \Cref{lem:recovery}, where $\tilde \Gamma_v$ is obtained by uncontracting the loops of $\tilde B_v$ in the graph $G$ with boundary $B_u$. Applying \Cref{lem:recovery} to the graph $(G, B_u)$ and \Cref{lem:order_invariance}, we see that $\tilde \Gamma_v$ is contained in the MSA of $(G, \partial)$. Furthermore by the definition of $y$ and \Cref{lem:recovery} item b., $\mathfrak F_T(v) \cap B_u = y$. Thus $\mathfrak F_T(u) \cap \mathfrak F_T(v) = \mathfrak F_T(u) \cap \mathfrak F_T(y) $ where $T$ is the MSA of $(G, \partial)$. Note that by item b. of \Cref{lem:recovery}, 
 $\mathfrak F_T(y) \cap \mathfrak F_T(u)  = \{\partial\}$ if and only if $y = \partial$. 
\end{proof}

\subsection{CLEB walk and invasion percolation}\label{sec:invasion}
We take a quick detour to prove that in the unoriented setup, the CLEB walk actually reduces to the well-known invasion percolation. This section can be read independently of the rest of the article, except we borrow notations from \Cref{sec:CLEB_walk}.

In this section, we modify our setup of \Cref{sec:CLEB_finite}: we assume that   {\em both orientations of an edge $e$ have the same weight, but otherwise they satisfy \eqref{eq:linear_comb}}. In this setup, one can unambiguously define the minimum spanning tree (MST) on $G$ where the weight of every (unoriented) edge is the common weight of both its orientations. It is easy to see that if we fix any vertex in the graph to be the boundary and orient all the edges of the MST towards this vertex, we obtain the MSA. It is well-known that invasion percolation produces the MST (this is also known as Prim's algorithm \cite{prim1957shortest}). It is natural to thus compare the CLEB walk algorithm and invasion percolation in this setup. We prove below that indeed one can recover the invasion percolation process from the CLEB walk.

First of all, a moment's thought reveals that there is no issue in running the CLEB walk, namely there are no ties to deal with in any step: indeed the weight of the outgoing edges from any vertex in any step, are all distinct, and hence subtracting linear combinations of other weights from them cannot produce any ties by \eqref{eq:linear_comb}. However, the validity of the CLEB walk algorithm is not immediate.

Now let us recall the invasion percolation process.
Take a finite, connected, weighted graph $G$ possibly with multiple edges (no orientation yet) and assume that the weights satisfy \eqref{eq:linear_comb}. Invasion percolation started at $x$ defines a growing sequence of trees $(T_k)_{k \ge 1}$ as follows: $T_1$ is the edge with the smallest weight incident to $x$. Once $T_k$ is defined, consider all the edges with one endpoint in $T_k$ and the other outside $T_k$, and pick the one with the smallest weight among them, and add it to $T_k$ to define $T_{k+1}$. We stop if there is no further edge to add.


We shall use the notations from \Cref{sec:CLEB_finite} introduced to define the CLEB walk. We recall for convenience that for a CLEB walk started at $x$, $S_i \cap \vec E_i$ is an oriented path $P_i$ started at $x$ (before $\partial $ is hit). Recall the notation $\cO(v)$ which denotes the set of outgoing edges from $v$.

\begin{proposition}
In the above setup, perform the CLEB walk from a vertex $x$, let $\tau_1 = 1$, and inductively for $k \ge 2$ define $\tau_k$ to be the smallest $t > \tau_{k-1}$ such that the CLEB walk reveals an oriented edge whose reversal has not been exposed before time $t$. Let $\bar S_{\tau_k}$ denote the set of unoriented edges one of whose orientations is in $S_{\tau_k}$. Then $\bar S_{\tau_k} = T_k$.
\end{proposition}

\begin{proof}
The proof proceeds through induction. Start a CLEB walk from $x$. Write $P_{k} = (\vec e_1,\ldots, \vec e_{m_k}) $ and let $x_i = \vec e_i^+$ for $1 \le i \le k$ and $x_0 = x$. We claim that for all $k \ge 1$:
\begin{itemize}
\item For every $ 0 \le i \le m_k$, $U_{0, \vec e}- U_{k, \vec e}$ is the same for all $\vec e \in \cO(x_i)$.
\item $U_{k, \vec e} = U_{0,\vec e} - U_{0,\vec e_{i}}$ for all $\vec e \in \cO(x_{i-1})$, $1 \le i \le m_{k}$. 
\end{itemize}
Note the first item above is a consequence of the second except for the elements in $\cO(x_{m_k})$. 
We now prove the claim by induction.  Clearly, the claim is true for $k=1$ simply by the definition of the first step of the CLEB walk. Now assume this is true for some $k$. Let $\overleftarrow e$ denote the reversal of the orientation of the oriented edge $\vec e$. Let $\alpha_k \ge 0$ be the common value of $U_{0, \vec e}- U_{k, \vec e}$ for all $\vec e \in \cO(x_{m_k})$. In step $k+1$, the CLEB walk either exposes the reversal of $\vec e_{m_k}$ or it doesn't. In the former case, the weight subtracted in step $k+1$ from all edges in $\cO(x_{m_k})$ is $U_{k, \overleftarrow{e}_{m_k}}= U_{0, \overleftarrow{e}_{m_k}} - \alpha_k  = U_{0, \vec{e}_{m_k}} - \alpha_k$ by induction hypothesis and the assumption that the initial weight of $\vec e$ and its reversal is the same. By this observation, and the induction hypothesis, $$U_{0, \vec e} - U_{k+1, \vec e} = U_{0, \vec e} - (U_{k, \vec e}  - U_{0, \vec{e}_{m_k}} + \alpha_k) = U_{0, \vec{e}_{m_k}}  $$ of every $\vec e$ in $\cO(x_{m_k})$. Combining this with the second item of the induction hypothesis, the initial weight minus the current weight of every oriented edge coming out of the new vertex formed by contracting $(x_{m_{k-1}}, x_{m_k}, x_{m_{k-1}})$ is the same, and is $U_{0, \vec e_{m_k}}$. Since all other weights are undisturbed, induction is complete in this case.

Now suppose, the next edge taken by the CLEB walk is not the reversal of $\vec e_{m_k}$. We make a couple of observations.  Firstly,
\begin{equation}
U_{0,\vec e} - U_{0,\vec e_{i}} \ge 0 \text{ for all } \vec e \in \cO(x_{i-1}), \quad  1 \le i \le m_{k}. \label{eq:monotone_P1}
\end{equation}
Secondly, since the reversal of $\vec e_{i}$ has the same $U_0$-weight as $\vec e_i$, the claim also implies that 
\begin{equation}
U_{0, \vec e_i } \ge U_{0, \vec e_{i+1}} \text{ for all } 1 \le i \le m_k-1 \label{eq:monotone_P}.
\end{equation}

Observe that this new edge (call it $\vec e$) cannot have an endpoint which is one of $x_i$, $0 \le i \le m_{k}-2$. Indeed, if this is the case, $U_{k, \vec e} < U_{k, {\overleftarrow e}_{m_k}}$ and by the first item of the induction hypothesis, $U_{0, \vec e} < U_{0, {\overleftarrow e}_{m_k}}  = U_{0, {\vec e}_{m_k}}$. However, this means that $U_{0, \overleftarrow e}$ violates \eqref{eq:monotone_P1} in light of \eqref{eq:monotone_P}. Thus this new edge has a completely new vertex as an endpoint. The first item is true for the endpoint, the second item is not applicable. Furthermore, the second item for $x_{m_k}$ is true by a similar argument as in the first case. By induction, the claim is proved.

Now we come back to the proof of the proposition. Let $\cO'(x_i)$ denote the outgoing oriented edges from $x_i$ except those which are either $\vec e_i$ or their reversals.  Thus $\tau_{k+1}$ is the smallest time after $\tau_k$ at which one of the elements of $\cup_{1 \le i \le m_k} \cO'(x_i)$ is exposed. We claim that at time $\tau_{k+1}$ the oriented edge $\vec e$ exposed has $U_{0,\vec e}$ the smallest among all $U_{0,x}$ for $x \in \cup_{1 \le i \le m_k} \cO'(x_i)$, which is indeed what is required to prove the proposition. Suppose the edge $\vec f$ has the smallest $U_0$-weight in $\cup_{1 \le i \le m_k} \cO'(x_i)$  and $\vec f $ is in $\cO'(x_i)$ for some $i \le m_{\tau_k}$. By \eqref{eq:monotone_P} and \eqref{eq:monotone_P1},
$$
\min\{U_{0,\vec e}:\vec e \in \cO'(x_t), i+1 \le t \le m_{\tau_k}\} > U_{0, \vec f} > U_{0, \vec e_{i+1}} > \ldots>U_{0, \vec e_{m_{\tau_k}}}
$$
By the first item of the claim and the above inequalities, the CLEB walk must retrace back along the reversals of $\vec e_{m_{\tau_k}}, \ldots, \vec e_{i+1}$. After it retraces back with contraction in every step, the new vertex $v$ created has  $U_{0, \vec e} -  U_{\tau_k+m_{\tau_k}  - i, \vec e} $ equal for all $\vec e \in \cO(v)$, by first item of the claim. Now note that the reversal of $\vec e_i$ is in $\cO(v)$ but its $U_0$-weight is the same as $U_{0, \vec e_i } > U_{0, \vec e_{i+1}}>U_{0, \vec  f}$. By definition, $\vec f$ has the smallest $U_0$-weight in every other element in $\cO(v)$. Thus the edge exposed in the next step of the CLEB walk must be $\vec f$. The proof is thus complete.
\end{proof}

\subsection{CLEB algorithm in the infinite graph}\label{sec:ECL_infinite}

The notion of the CLEB algorithm extends easily to the infinite setting, the only difference from the finite setting is that the process never stops (analogous to the celebrated Wilson's algorithm `rooted at $\infty$' in transient graphs). Take $G$ to be an infinite, locally finite, connected graph possibly with multiple edges but no self-loops. Suppose that we have a collection of random weights $(U_{\vec e})_{\vec e}$ which is a.s.\ generic. As in the finite case, we can define the CLEB walk at any vertex $x$, except the process runs forever. Thus, using the notations from \Cref{sec:CLEB_finite}  we get an infinite collection of exposed edges $\cup_{i \ge 0} S_i $. Observe that since $S_i$ is monotone increasing: for every oriented edge $\vec e$, 
\begin{equation}
 1_{\vec e \in S_i} \text{ is non-decreasing in $i$ almost surely}. \label{eq:monotone}
\end{equation}
Similar to the finite setup, that $S_i \cap \vec E_i$ is in general a path in $G_i$ which we call $P_i$. We point out here that $P_i$ is empty if all the exposed edges up to step $i$ are contracted into a single vertex in $G_i$. 

We say the CLEB walk started at $x$ is \textbf{transient} almost surely if $S_i \cap \vec E_i $ converges to an infinite path started at $x$ as $i \to \infty$ almost surely, otherwise we call it \textbf{recurrent} at $x$. Note that unlike Markov chains it is unclear whether there is any 0-1 law for transience (see \Cref{qn:0-1}). Nevertheless the following is true.
\begin{lemma}\label{lem:transience_equiv}
For every vertex $x$, the CLEB walk started at $x$, $S_j \cap \vec E_j = \emptyset$ finitely often a.s.\  if and only if the CLEB walk is transient starting from every vertex $x$ a.s..
\end{lemma}
\begin{proof}
The proof of the if part is trivial. So, let us argue the only if the direction and note that the situation we need to rule out is that the length of $P_j$ is bounded by some number infinitely often with positive probability. 
Let $\ell_j$ denote the length of $S_j \cap \vec E_j$. Assume that there exists a vertex $x$ such that the CLEB walk started at $x$ is not transient, i.e., $\ell:=\liminf \ell_j< \infty$ with positive probability. Let $\cR_k = \{  \inf\{\ell_j: j \ge k\}  = \ell \}$ and let $k_0$ be smallest integer such that $\P(\cR_{k_0})>0$. 
On the event $\cR_{k_0}$,  the head of $S_{k_0} \setminus S_{k_0-1}$ is a new unexplored vertex, say  $v $,  in $G$. Since $\ell_j$ never goes below $\ell$ after time $k_0$ on the event $\cR_{k_0}$, any exposed edge from the time $k_0$ onward does not meet  any vertex from $V(S_{k_0-1})$.  Also, 
since $\ell_j$ is $ \ell$ infinitely often after time $k_0$ on the event $\cR_{k_0}$, the CLEB walk started at $v$ satisfies $S_j \cap \vec E_j = \emptyset$ infinitely often on this event, which contradicts the hypothesis. 
\end{proof}

\paragraph{CLEB algorithm rooted at infinity}
If in a graph $G$, the CLEB walk is transient a.s.\ for all vertices, we can define an infinite version of the CLEB walk, which we call \textbf{CLEB walk rooted at infinity.} To describe this algorithm, we assume that we have a deterministic collection of generic weights $(w_{\vec e})_{\vec e \in \vec E}$ such that the CLEB walk started at any vertex is transient. Let $B_G(x,r)$ denote the graph distance ball of radius $r$ in $G$.
An ordering $(x_1, x_2, \ldots)$ of the vertex set $V$ is \textbf{good} if for all $r \ge 1$, there exists an $n_r$ such that the vertex set in $B_G(x_1,r)$ is contained in $(x_1,\ldots, x_{n_r})$.

\begin{itemize}
\item Order the vertices $V$ in some good order $(x_1, x_2, \ldots, )$, sample a CLEB walk from $x_1$ and let $B_{x_1} = V(\Gamma_{x_1})$ where $\Gamma_{x_1}$ is the set of exposed edges of the CLEB walk started from $x$. 

\item Having defined $B_{x_1}=:B_1,B_{2}, \ldots, B_{j}$,  if $\cup_{1 \le m \le j} B_{m} = V$, stop the process. Otherwise, choose the vertex $x_{(j)}$ with the smallest index which is not in $\cup_{1 \le m \le j} B_{m}$. Perform the CLEB walk with boundary $\cup_{ 1 \le m \le j }B_{m} $ started at $x_{(j)}$ with $x_{(1)} = x_1 $.  Set $B_{j+1}  = V(\Gamma_{x_{i}})$ for this CLEB walk.
\end{itemize}
We emphasize that at this point, it is not clear whether $\cup_{1 \le m \le j}B_i \neq V$ for some finite $j$. Also, observe that the reverse ordering of loops in order to obtain an infinite analogue of the uncontraction process of loops is ambiguous unless the CLEB walk is a.s.\ transient from every vertex. 


Using the assumption that CLEB walk is transient from every vertex, we now define certain stopping times analogous to \Cref{sec:CLEB_walk}.
Let $\tau_0 =0$ and let $\tau_i$ be the $i$th time such that $S_{\tau_i} \cap E_{\tau_i} = \emptyset$. Let $\tau_{N_1}$ be the largest such time, which is a.s.\ finite by the a.s.\ transience assumption. Let $\cL_1$ be the ordered sequence of loops contracted up to time $\tau_{N_1}$. We now uncontract the loops of $\cL_1 $ in reverse order exactly as in \Cref{sec:CLEB_walk}. Let $E_{N_1}$ be the spanning arborescence of the graph spanned by $V(S_{\tau_{N_1}+1})$ with boundary $(\vec f_1)_+$ where $\vec f_1:=S_{\tau_{N_1}+1} \setminus S_{\tau_{N_1}} $. 

Since $S_{\tau_{N_1}}$ is never hit in after time $\tau_{N_1}$ in the CLEB walk started from $x_1$, we can iterate the same procedure for the CLEB walk after time $\tau_{N_1}+1$, again exactly as in \Cref{sec:CLEB_walk}. Inductively, if $N_{j-1} <\infty$, let $\tau_{N_{j-1}+k}$ be the $k$th time after $N_{j-1}$ such that $$S_{\tau_{N_{j-1} + k}} \cap \vec E_{\tau_{N_{j-1}+k}} = S_{\tau_{N_{j-1}+1}} \cap \vec E_{\tau_{N_{j-1}+1}}.$$ Let $N_j$ be such that  $N_{j} - N_{j-1}$ be the largest such $k $. It is not too hard to see that $\tau_{N_j} <\infty$ a.s.\ if the CLEB walk is transient from the head of $f_{j-1}$ (since the exposed edges are a.s.\ never hit, transience in the original graph implies this statement), and thus the inductive definition is valid for all $j\ge 1$. Uncontracting the loops for each $j$ separately as described in \Cref{sec:CLEB_walk}, we obtain  a sequence $(E_{N_i})_{i \ge 1}$ of spanning arborescences of the graph spanned by $V(S_{\tau_{N_{i+1}}})$  with boundary being the tail of $\vec f_i:=S_{\tau_{N_i} +1} \setminus S_{\tau_{N_i}}$. Thus we see that $\cup_{i \ge 1} E_{N_i}$ forms an infinite spanning arborescence of the graph spanned by $B_{1}$ `rooted at infinity', in the sense that there is an infinite oriented path started from $x_1 $ with the branch from all other vertices merging with this path.



 We can similarly perform the uncontraction procedure to inductively obtain spanning arborescence of the graph spanned by $B_{j+1}$ with boundary $\cup_{1 \le m \le j}B_m$ (each such process might produce an infinite arborescence, or might form a finite arborescence which stops when one of the previously exposed vertices are hit). Since the ordering of $V$ is chosen to be good, this gives us a spanning arborescence of $G$. This completes the description of the CLEB algorithm rooted at infinity.

Now we consider a collection of nonnegative random weights $(U_{\vec e})_{\vec e \in \vec E}$ which is a.s.\ generic.
At the moment, it is not clear whether this is the limit of the MSA even if the CLEB walk is a.s.\ transient from every vertex, nor is it clear that the limit is independent of the order of the vertices chosen. The next proposition tells us that indeed both are true.

\begin{proposition}\label{prop:necessary_cond}
Suppose $G$ is an infinite, locally finite, connected graph and the CLEB walk is transient a.s.\ for all vertices $x$ in $G$. Then the wired MSA limit exists almost surely in $G$ and the limit does not depend on the choice of the exhaustion chosen.

Furthermore, the wired MSA can be a.s.\ sampled using the CLEB algorithm rooted at infinity.
\end{proposition}
\begin{proof}
Let $(G_n)_{n \ge 1}$ be an exhaustion of $G$. Pick an edge $\vec e$. Pick a choice of good ordering of the vertices $(x_1, x_2, \ldots)$ and pick a (random) $R$ so large that $\vec e_-$ is an endpoint of an edge in some  set  $A(\vec e) $ of the form $S_{\tau_{N_j} }$ when $(\Gamma_{x_{(1)}}, \cdots, \Gamma_{x_{(r)}})$   has been sampled where $r$ is the maximum number such that $\{ x_{(1)}, \ldots, x_{(r)}\} \subset \{x_1, x_2, \ldots, x_R\}$. By the a.s.\ transience of the CLEB walk starting from any $x$, such a set $A(\vec e)$ must be finite almost surely. 
Since the ordering of vertices chosen is good, $R <\infty$ almost surely. Now pick a  random but finite $n$ large enough so that $(x_1,\ldots, x_R) $ is inside $G_n$ as well as $A(\vec e) \subset G_n$.   We can use the CLEB walk algorithm to sample the MSA in $G_n$ where the first $R$ vertices are $(x_1,\ldots, x_R)$. 
Therefore, by Lemma~\ref{lem:recovery}, $\vec e$ is in the MSA of $G_n$ if and only if it is in the arborescence obtained by the CLEB algorithm rooted at infinity. 
\end{proof}

For general graphs where a.s.\ transience is perhaps not guaranteed, we now prove that the connectivity profile of the vertices in the graph converges almost surely.  Recall the notation $\mathfrak F_T(v)$ which denotes the \textbf{future} of $v$ in $T$. Notice that if $u$ and $v$ belong to the same component of $T$ then $\mathfrak F_T(u)$ and $\mathfrak F_T(v)$ meet at a vertex after which they merge. Let us call this vertex $u \wedge v$, noting that $u \wedge v = u$ or $u \wedge v = v$ is possible and $u \wedge v =  \emptyset $ by convention if $\mathfrak F_T(u) \cap \mathfrak F_T(v) = \emptyset$. Introduce the notation
$$
u \stackrel{T}{\longleftrightarrow} v
$$
if $\mathfrak F_T(u) \cap \mathfrak F_T(v) \neq \emptyset$.

We now state a lemma analogous to \Cref{lem:connection}.
\begin{lemma}\label{lem:connection_inf}
Suppose the wired MSA limit exists in G almost surely, call it $T$. Then for every $u ,v$, $B_u \cap B_v \neq \emptyset$ if and only if $u \stackrel{T}{\longleftrightarrow} v$.
\end{lemma}

\begin{proof}
Let $\{G_n\}_{n \ge 1}$ be an exhaustion of $G$ and let $B_x(n)$ be the vertices exposed by running the CLEB walk started at $x$ in $G_n^{\w}$. Let $T_n$ be the MSA of $G_n^{\w}$.
Using the monotonicity of \eqref{eq:monotone}, we see that $\{B_u(n) \cap B_v(n) \neq \partial\}$ is nondecreasing in $n$ and converges to $B_u \cap B_v \neq \emptyset$. Using \Cref{lem:connection}, we note that  the former event is $\{u \stackrel{T_n}{\longleftrightarrow} v \}$. Since $T_n $ converges to $T$ almost surely, we are done.
\end{proof}

We now record an upshot of \Cref{lem:connection,lem:connection_inf}. Although we do not know in general graphs that the wired MSA limit exists, we can show that the connection profile of the finite volume MSAs converges. To illustrate this, introduce the following notation. For a finite, connected, simple graph $G$,
$$\cC(G, \partial, T) = \left\{1_{u \stackrel{T}{\longleftrightarrow} v}: u,v \in V \right\}.$$
Recall that the notation above means that $u \wedge v \neq \partial$.
Similarly define $\cC(G,T)$ for $T$ an infinite arborescence in an infinite graph $G$. Recall the notations $G_i^{\w}, \partial_i$ from \Cref{subsec:intro}.
\begin{proposition}\label{prop:connectivity_convergence}
Let $G_1 \subset G_2,\ldots$ be an exhaustion of an infinite locally finite graph $G$ and suppose $T_i$ is the MSA of $(G_i^{\w}, \partial_i)$. Then 
$$
\cC(G_i^{\w}, \partial_i, T_i) \text{ is locally almost surely non-decreasing}
$$ in the sense that for any finite collection $(u_1,v_1), \ldots, (u_k,v_k)$ in $G$, 
$$
\left(1_{u_i \stackrel{T_i}{\longleftrightarrow} v_i}: 1 \le i \le k\right) \text{ is non-decreasing almost surely}.
$$
In particular, $\cC(G^{\w}_i, \partial_i, T_i) $ converges almost surely.
\end{proposition}
\begin{proof}
We provide a proof for $k=1$, the rest is similar. Perform the CLEB walk in $G_n$ started from $u$ first and then from $v$ and suppose we have discovered the vertex set $B_u(n)$ for the process from $u$ and then $B_v(n)$ for that of $v$. Observe from the monotonicity of the CLEB walk \eqref{eq:monotone} that  the event $\{B_u(n) \cap B_v(n) \neq \{\partial\}\}$ is a.s.\ nondecreasing in $n$.  Thus, we know from \Cref{lem:connection} that this means $\{u \stackrel{T_n}{\longleftrightarrow} v\}$ is also a.s.\ nondecreasing in $n$. Taking the limit as $n \to \infty$, we are done.
\end{proof}

\subsection{Loop contracting random walk (LCRW)}\label{sec:LCRW}
Suppose $G$ is a connected, infinite, locally finite graph that may contain multiple edges but no self-loop. In this section, we assume that 
\[ \text{ the weights $(U_{\vec e})_{\vec e \in \vec E}$ are i.i.d.\ $\mathrm{Exp}(1)$.}\] The memoryless property will lead to a certain Markovian nature of the CLEB walk which will allow us to compare it with the simple random walk. 
To state the following lemma, we will borrow notations from \Cref{sec:CLEB_finite}.

\begin{lemma}\label{lem:exponential_memoryless}
Suppose $(U_{0,\vec e})_{\vec e \in \vec E}$ are i.i.d.\ $\mathrm{Exp}(1)$. Let $v_1, v_2 \ldots$ be the sequence of vertices chosen in the CLEB walk algorithm. Let $\cF_i$ be the $\sigma$-algebra generated by $(S_k, \pi_{v_k})_{k \le i \wedge \vartheta}$. Then conditioned on $\cF_i$, the law of $(U_{i,\vec e})_{\vec e \in \vec E_i \setminus S_i}$ is i.i.d.\ $\mathrm{Exp}(1)$.
\end{lemma}
\begin{proof}
The proof hinges on the standard fact that if $X_1,\ldots, X_n$ are i.i.d.\ $\mathrm{Exp}(1)$, and $X_{\text{min}} = \min\{X_1,\ldots, X_n\}$, and $J$ is the unique index such that $X_J = X_{\text{min}}$, then $J$ and  $X_{\text{min}}$ are independent. Furthermore, conditioned on $J$, $X_{\text{min}}$, the law of $(X_i -X_{\text{min}}: i \neq J )$ is i.i.d.\ $\mathrm{Exp}(1)$.

The claim in the lemma can be also easily proved using induction. In step 0, the statement is true by definition. Suppose the statement is true for step $k$. Given $k < \vartheta$ and $\cF_k$, 
 in step $k+1$, we select a vertex $v$, which has the weights $\{U_{k,e_i}\}_{1 \le i \le d}$  attached to its outgoing edges which are i.i.d.\  $\mathrm{Exp}(1)$ by induction hypothesis. Let $J \in \{1,\ldots, d\}$ be the index of the minimal weight in $\{U_{k,e_i}\}_{1 \le i \le d}$ and let $\pi_v$ be the value of this minimal weight. Then $\cF_{k+1} = \sigma (\cF_k, \pi_v, J)$. Also $\{U_{k+1,\vec e} := U_{k,\vec e} - \pi_v\}_{\vec e \in  \{\vec e_i, 1\le i \le d, i \neq J\} }$ are i.i.d.\ $\mathrm{Exp}(1)$ conditioned on $\cF_{k+1}$ by the fact in the previous paragraph. The remaining weights are unchanged, and conditioned on $\cF_{k+1} $ are clearly i.i.d.\ $\mathrm{Exp}(1)$ by induction hypothesis, and are also independent of $\{U_{k+1,\vec e}\}_{\vec e \in  \{\vec e_i, 1\le i \le d, i \neq J\}} $. This completes the proof.
\end{proof}

Applying \Cref{lem:exponential_memoryless} to the CLEB walk algorithm, it is straightforward to see that edges are exposed like a simple random walk on an undirected graph (with unit weight on every edge). If a cycle is created, it is contracted, and the walk continues in the new graph obtained after the contraction process. We call this the loop contracting random walk, which we formalize in the next definition.

Let $G = (V,E)$ be a multi-graph with boundary $\partial$ being a single vertex. 
The \textbf{loop contracting random walk (LCRW)} on a graph $(G, \partial)$  started at $x \neq \partial$ is described by the triplet $(G_i,P_i, X_i)$ where $X_i$ is a vertex in $G_i$ and $P_i$ is a simple path $G_i$ started at $x$ and ending at $X_i$. Let $X_0 = x$, $G_0 = G$ and $P_0 = x$.  Having defined $X_0,\ldots, X_i$ and $G_0,\ldots, G_i$ and $P_i = (\vec e_1,\ldots, \vec e_{k_i} ) $ with $(\vec e_1)_- = x,$ $ X_i= (\vec e_{k_i})_+$, choose an edge $\vec e$ \emph{uniformly} from one of the outgoing edges from $X_i$ in $G_i$. 
\begin{itemize}
\item  If the head of $\vec e$ is the head of $(\vec e_j)$ then contract the cycle $C_i:=(\vec e_{j+1}, \ldots, \vec e_{k_i}, \vec e )$ into a vertex $v$. Define $G_{i+1}$ to be this new graph, i.e., $(G_{i+1}, \partial)  = \Psic((G_i, \partial),C_i)$ and let  $P_{i+1}$ to be the path $(\vec e_1,\ldots, \vec e_j)$ and $X_{i+1} = v$. Here contracting a cycle has the same meaning as in \Cref{sec:CLEB_finite}.
\item  If the head of $\vec e$ is  not in $P_i$, define $P_{i+1} = (\vec e_1,\ldots, \vec e_{k_i}, \vec e)$, $G_{i+1} = G_i$ and $X_{i+1} = \vec e_+$.
\item If $X_{i+1} = \partial$, terminate the process.
\end{itemize}

Note also that each oriented edge is traversed at most once by LCRW. We emphasize that we do not need any weights on edges for this process.

\begin{lemma}\label{lem:CLEB_LCRW}
The LCRW started at $x$ has the same law as the CLEB walk started from $x$. In fact the two processes can be naturally coupled: in each step, if an oriented edge $\vec e$ is added in the CLEB walk, choose that edge for the LCRW as well. 
\end{lemma}
\begin{proof}
This is immediate from \Cref{lem:exponential_memoryless}.  
\end{proof}
{ From now on, whenever we talk about the LCRW, we assume it is naturally coupled with a CLEB walk as described in \Cref{lem:CLEB_LCRW}.}
  
As in the case of the CLEB walk, one can also define the recurrence and transience of the LCRW. We say the LCRW is transient if $P_i$ tends to an infinite path almost surely, otherwise it is recurrent. Using the coupling of \Cref{lem:CLEB_LCRW}, we immediately get
\begin{lemma}\label{eq:LCRW_tr_existence}
Suppose the weights are i.i.d.\ $\mathrm{Exp}(1)$. The statement of \Cref{prop:necessary_cond} holds if the LCRW started from $x$ is transient a.s.\ for all $x$. 
\end{lemma}

\begin{example}\label{ex:1d}
We now argue that in $\Z$, every vertex is recurrent for the LCRW (in fact the LCRW collapses to a single vertex infinitely often with probability 1). Indeed, every loop is formed by a backtrack which contracts a cycle formed by an oriented edge and its reversal. Thus the length of $P_i$ changes by $\pm 1$ with equal probability, which must hit 0 with probability 1.

We now sketch that in this case the a.s.\ limit of the MSA does not exist. Firstly, if the MSA in $\Z$ is connected, then there are only two choices for it: either all edges point to $+\infty$ or all of them point $-\infty$.
Using simple random walk theory, almost surely, we can find sequences $(M_k,N_k)_{k \ge 1}$ such that, on the one hand, the CLEB walk exits $[-M_k,M_k]$ through the left and the right most vertex exposed tends to infinity; and on the other hand, it exits $[-N_k,N_k]$ through the right and the leftmost point exposed tends to $-\infty$. In other words, it oscillates arbitrarily to the right and left, just like a simple random walk. Using \Cref{lem:connection_inf}, we note that a.s.\ limits through both sequences are connected, and because of the one-dimensional geometry, the limit along $M_k$ is the arborescence where all edges must point towards $-\infty$ (since the last edge through which it exits does point to the left and is not erased.) Similarly, the limit through $N_k$ is the arborescence where all edges point to the right.

Note however that by symmetry and recurrence of the LCRW, the limit in law of the arborescence does exist: each choice of the connected arborescence is chosen with equal probability. We leave detailed proofs of these facts to the reader.
\end{example}

\begin{remark}
Using the ideas in \Cref{ex:1d}, it is not too hard to construct a deterministic collection of weights for which the CLEB walk is recurrent in $\Z^d$ for $d \ge2$ and the limit for the MSA does not exist. 
\end{remark}

%

\section{Comparison with simple random walk}\label{sec:LCRW_SRW}
Although we cannot generally prove that transience of simple random walk implies transience of loop contracting random walk, we shall prove it in the case of bounded subdivision of trees with minimal degree 2 and consequently \Cref{thm:convergence}. The proof will go through a comparison with the simple random walk: we first prove that the probability of escaping the starting vertex is bigger for the LCRW compared to a simple random walk, in finite trees (there is no assumption on the trees here). Unfortunately, since the graph is changing in every step, there is no natural 0-1 law for transience, which is the reason we need the assumption on the structure of the trees in \Cref{thm:convergence}.

We remark that an alternate strategy to proving  \Cref{thm:convergence} could be to calculate the speed of the CLEB walker from the root (in an appropriate sense). However, since \Cref{lem:contraction_resistance} works in all finite trees, and is a useful insight into the behavior of the loop contracting random walk, we choose to go through this route. Furthermore, \Cref{lem:contraction_resistance} will be useful for our results on Galton-Watson trees as well where calculating the speed is perhaps more challenging.

Throughout this section we are in our usual setup: $G$, if finite, is a connected graph possibly with multiple edges but no self-loops, and if infinite we also assume local finiteness.  
 Let $(Y^{(G,v)}_t)_{t \ge 0}$ be the loop contracting random walk started at $v$ in a finite graph $G$ stopped when it hits $\partial$. Let $(X^{(G,v)}_t)_{t \ge 0}$ denote the simple random walk in the graph $G$ started at $v$. Let $\tau^Y_u = \tau^{Y^{(G,v)}}_u$ denote the first time  $Y$ started at $v$ hits $u$. Note that during some steps of the loop contracting random walk, the underlying graph gets updated via contraction. However, until the time the walk returns to the vertex $u$, in every updated such graph the vertex $u$ remains as is (i.e., uncontracted), so the random time $\tau^{Y^{(G,v)}}_u$ is well defined.  In the similar vein, let $\tau^{Y^{(G,v)}}_{v,+}$ denote the first time $Y$ started at $v$ returns to $v$.  Similarly define $\tau^X_u$ and $\tau^{X}_{v,+}$ for the simple random walk, and these definitions are standard.
 
 We start with a well-known fact about simple random walk which follows from electrical network theory. 
\begin{lemma}\label{lem:contraction_resistance}
Let $G$ be a finite graph as above and let $u \neq v$. Let $C$ be a simple cycle in the graph not containing $u,v$ and let $G'$ be the graph obtained by contracting the cycle. Then
$$
\P(\tau^{X^{(G,u)}}_v  < \tau^{X^{(G,u)}}_{u,+} ) \ge \P(\tau^{X^{(G',u)}}_v  < \tau^{X^{(G',u)}}_{u,+} )
$$
\end{lemma}
\begin{proof}
This simply follows from the fact that contracting cycles cannot increase effective resistance by Thomson's principle, and the degree of $u$ does not change. See \cite[Proposition 9.5 and Corollary 9.14]{markovmixing}.
\end{proof}

\begin{proposition}\label{prop:hitting_prop}
Let $S$ be a finite rooted tree, $v$ is a vertex in $S$. Suppose all the leaves of $S$ except $v$ in case $v$ is a leaf are glued into a single vertex which we denote by $\partial$. Call this new graph  $T$ and call $\partial $ its boundary vertex.
Then
$$
\P(\tau^{Y^{(T,v)}}_\partial < \tau^{Y^{(T,v)}}_{v,+} ) \ge \P(\tau^{X^{(T,v)}}_\partial < \tau^{X^{(T,v)}}_{v,+} )
$$
\end{proposition}
\begin{proof}
The proof proceeds by induction on the number of vertices of the tree $S$. If a tree has two vertices, the inequality is clear (in fact both sides equal 1). Now suppose the result is true for any tree with $k$ vertices and any vertex $v$ in it. Take a tree $S$ with $k+1$ vertices, fix $v$, glue all the leaves (except possibly $v$) into $\partial $ and call this graph  $T$. Let $N \ge 0$ be the number of edges of $T$ connecting $v$ to $\partial$. We decompose

\begin{align}
\P(\tau^{Y^{(T,v)}}_\partial < \tau^{Y^{(T,v)}}_{v,+} ) & = \frac{N}{\deg (v)} + \frac1{\deg(v)}\sum_{u \sim v, u \neq \partial} \P(\tau^{Y^{(T,u)}}_\partial < \tau^{Y^{(T,u)}}_{v} )\label{eq:break_X}\\
\P(\tau^{X^{(T,v)}}_\partial < \tau^{X^{(T,v)}}_{v,+} ) & =  \frac{N}{\deg (v)}+ \frac1{\deg(v)}\sum_{u \sim v, u \neq \partial} \P(\tau^{X^{(T,u)}}_\partial < \tau^{X^{(T,u)}}_{v} )\label{eq:break_Y}
\end{align}
Observe in the first equality, even though the edge $(v,u)$ is exposed, it does not affect the probability on the right hand side, since the exposed edge has an effect only after $v$ is hit. Thus it is enough to show that for each $u \sim v, u \neq \partial$,
\begin{equation}
\P(\tau^{Y^{(T,u)}}_\partial < \tau^{Y^{(T,u)}}_{v} ) \ge \P(\tau^{X^{(T,u)}}_\partial < \tau^{X^{(T,u)}}_{v} ),\label{eq:each_term}
\end{equation}
which is what we shall prove.

Take $u \not = \partial$ and $u \sim v$ and note that $u$ is not a leaf.
Let $S_u$ be the component of $u$ in $S$ if $v$ is removed from $S$, and let $T_u$ be the graph obtained by gluing all the leaves of $S_u$ into a single vertex (we still call this glued vertex $\partial$ admitting an abuse of notation). Let $d_u$ be the degree of $u$ in $S$. Observe the following recursion:
\begin{equation}
\P(\tau^{X^{(T,u)}}_\partial < \tau^{X^{(T,u)}}_{v} ) =  \frac{d_u-1}{d_u } \left( \P(\tau^{X^{(T_u,u)}}_\partial < \tau^{X^{(T_u,u)}}_{u,+})  + \P(\tau^{X^{(T_u,u)}}_\partial  \ge \tau^{X^{(T_u,u)}}_{u,+}) \P(\tau^{X^{(T,u)}}_\partial < \tau^{X^{(T,u)}}_{v} ) \right) \label{eq:recursion_RW}
\end{equation}
We broke the event that the random walk hits $\partial $ before $v$ into two: either it hits $\partial$ before returning to $u$ in $T_u$, or it returns to $u$ before hitting $\partial$ in $T_u$ and then hits $\partial$ before hitting $v$. The recursion above is immediate using the Markov property of random walk. The prefactor $(d_u-1)/d_u$ comes from the fact that we changed $T$ to $T_u$ in two of the terms above.

Now we apply a similar recursion for the loop contracting walk $Y$. Note that the event $\tau^{Y^{(T,u)}}_\partial < \tau^{Y^{(T,u)}}_{v} $ can be similarly broken up into two disjoint events: either $Y$ hits $\partial $ before returning to $u$ in $T_u$ or it returns to $u$ making an excursion $\cE$ in the tree $T_u$, and then hits $\partial$ before $v$ in the graph obtained by contracting $\cE$ (note again that neither $v$ nor $\partial$ is part of any contracted cycle in $\cE$, hence the events are unambiguous). Admitting an abuse of notation, let $\P(\cE)$ denote the probability of the excursion $\cE$ occurring. Noting that the excursion is simply a union of cycles, let $T_\cE$ be the tree obtained by contracting these cycles in $T_u$. Note that when $\cE$ is completed, a cycle is formed and let $u_\cE$ be the vertex obtained by contracting that cycle.  Now attach back the vertex $v$ to  $u_\cE$ using a single edge and call this graph $T'_\cE$.
\begin{equation}
\P(\tau^{Y^{(T,u)}}_\partial < \tau^{Y^{(T,u)}}_{v} )=  \frac{d_u-1}{d_u } \left( \P(\tau^{Y^{(T_u,u)}}_\partial < \tau^{Y^{(T_u,u)}}_{u,+})  + \sum_\cE \P(\cE) \P(\tau^{Y^{(T'_\cE,u_\cE)}}_\partial < \tau^{Y^{(T'_\cE,u_\cE)}}_{v} ) \right) \label{eq:recursion_Y}
\end{equation}
Now observe that $T_u$ has at most $k-1$ vertices. Thus using the induction hypothesis:
\begin{equation*}
\P(\tau^{Y^{(T_u,u)}}_\partial < \tau^{Y^{(T_u,u)}}_{u,+}) \ge \P(\tau^{X^{(T_u,u)}}_\partial < \tau^{X^{(T_u,u)}}_{u,+})
\end{equation*}

Let $T_u'$ be the graph obtained by attaching back $v$ to $T_u$.  Observe that $T'_\cE$ is the graph obtained from $T_u'$ after contracting the cycles in $\cE$. Therefore,
\begin{align*}
\P(\tau^{Y^{(T'_\cE,u_\cE)}}_\partial < \tau^{Y^{(T'_\cE,u_\cE)}}_{v}) &= \P(\tau^{Y^{(T'_\cE,v)}}_\partial < \tau^{Y^{(T'_\cE,v)}}_{v,+})\\
& \ge \P(\tau^{X^{(T'_\cE,v)}}_\partial < \tau^{X^{(T'_\cE,v)}}_{v,+})\\
& \ge \P(\tau^{X^{(T_u',v)}}_\partial < \tau^{X^{(T_u',v)}}_{v,+})\\
& = \P(\tau^{X^{(T,u)}}_\partial < \tau^{X^{(T,u)}}_{v} ).
\end{align*}
The first and last equalities are trivial since $v$ has degree 1 in $T_u'$. The second inequality follows from the induction hypothesis, and the last inequality follows from \Cref{lem:contraction_resistance}. Finally, note that the term
\begin{equation*}
\sum_\cE \P(\cE) = \P(\tau^{Y^{(T_u,u)}}_\partial  \ge  \tau^{Y^{(T_u,u)}}_{u,+}).
\end{equation*}
Thus, plugging these into \eqref{eq:recursion_Y}
\begin{align*}
\P(\tau^{Y^{(T,u)}}_\partial < \tau^{Y^{(T,u)}}_{v} ) & \ge \frac{d_u-1}{d_u }  \left( \P(\tau^{Y^{(T_u,u)}}_\partial < \tau^{Y^{(T_u,u)}}_{u,+})  + \sum_\cE \P(\cE)  \P(\tau^{X^{(T,u)}}_\partial < \tau^{X^{(T,u)}}_{v} ) \right) \\
& = \frac{d_u-1}{d_u }  \left( \P(\tau^{Y^{(T_u,u)}}_\partial < \tau^{Y^{(T_u,u)}}_{u,+})  + \P(\tau^{Y^{(T_u,u)}}_\partial  \ge  \tau^{Y^{(T_u,u)}}_{u,+})    \P(\tau^{X^{(T,u)}}_\partial < \tau^{X^{(T,u)}}_{v} ) \right) \\
& = \frac{d_u-1}{d_u }  \left( \P(\tau^{Y^{(T_u,u)}}_\partial < \tau^{Y^{(T_u,u)}}_{u,+}) (1-\P(\tau^{X^{(T,u)}}_\partial < \tau^{X^{(T,u)}}_{v} )  )  +\P(\tau^{X^{(T,u)}}_\partial < \tau^{X^{(T,u)}}_{v} )  \right)\\
& \ge  \frac{d_u-1}{d_u }  \left( \P(\tau^{X^{(T_u,u)}}_\partial < \tau^{X^{(T_u,u)}}_{u,+}) (1-\P(\tau^{X^{(T,u)}}_\partial < \tau^{X^{(T,u)}}_{v} )  )  +\P(\tau^{X^{(T,u)}}_\partial < \tau^{X^{(T,u)}}_{v} )  \right)\\
& = \frac{d_u-1}{d_u } \left( \P(\tau^{X^{(T_u,u)}}_\partial < \tau^{X^{(T_u,u)}}_{u,+})  + \P(\tau^{X^{(T_u,u)}}_\partial  \ge \tau^{X^{(T_u,u)}}_{u,+}) \P(\tau^{X^{(T,u)}}_\partial < \tau^{X^{(T,u)}}_{v} ) \right)\\
& = \P(\tau^{X^{(T,u)}}_\partial < \tau^{X^{(T,u)}}_{v} ) 
\end{align*}
thereby completing the proof. The last equality is simply \eqref{eq:recursion_RW}.
\end{proof}


We now state and prove a lemma regarding the escape probabilities of the simple random walk on trees. Essentially we claim that in a bounded subdivision of a tree with minimal degree at least 3, the escape probability from any vertex is lower bounded by a constant that is uniform over all such trees and all vertices in them.  Let $u \sim v$  be two vertices in an infinite tree $T$ where $\sim$ denotes $u $ and $v$ are adjacent, and suppose $T_{u,v}$ be the component of $u$ in  $T$ when a $v$ is removed from it. Let $X^{(G,x)}$ denote the simple random walk in a graph $G$ started from $x$ and let $\tau_{u,+}^{X^{(G,u)}}$ be as in the proof of \Cref{prop:hitting_prop}: it denotes the first time the random walk $X^{(G,u)}$  returns to $u$. We emphasize that this walk is in the unoriented, unweighted version of $G$. 
\begin{lemma}\label{lem:hitting}
 Fix $M>0$. We have\begin{equation}
\inf_{T, u,v} \P(\tau_{u,+}^{X^{(T_{u,v}, u)}} = \infty) =: \alpha(M) >0\label{eq:alpha}
\end{equation}
where the infimum is over all vertices $u,v$ in $T$ and over all trees $T$ which is an $M$-bounded subdivision of a tree with minimum degree at least 3.
\end{lemma}
\begin{proof}
    This is a standard exercise in electrical network theory, we provide a quick proof here assuming familiarity of the reader with this theory. First of all, note that replacing $T$ by a bounded subdivision can change the left hand side of \eqref{eq:alpha} by a uniformly bounded factor depending only on $M$, so we can equivalently assume that the minimum degree of $T$ is at least 3 (i.e. there is no subdivision, or $M=1$). Let $R^{G}_{\text{eff}}(x \leftrightarrow y)$ denote the effective resistance between $x$ and $y$ on the graph $G$ with with conductance 1 on every edge. 
Observe that by monotonicity of effective resistance (Thomson's principle, see \cite[Theorem 9.10]{markovmixing}) 
\begin{equation}
R^{T_{u,v}}_{\text{eff}}(u \leftrightarrow \infty) \le R^{T_2}_{\text{eff}}(x \leftrightarrow \infty) =:\alpha  \label{eq:alpha_binary}\end{equation}
where $T_2$ is the binary tree with root $x$. Let $d_u$ be the degree of $u$ in  $T_{u,v}$.  Observe that by the \cite[Proposition 9.5]{markovmixing} and the parallel law:
\begin{align}
    \P(\tau_{u,+}^{X^{(T_{u,v}, u)}} = \infty) & = \frac1{d_u} \sum_{x \sim u \text{ in $T_{u,v}$}} \frac 1{1+R^{T_{x,u}}_{\text{eff}} (x \leftrightarrow \infty)} \nonumber\\
    & \ge \frac1{d_u} \sum_{x \sim u \text{ in $T_{u,v}$}} \frac 1{1+\alpha}\label{eq:upper_alpha}\\
    & = \frac1{1+\alpha}.
\end{align}
where in \Cref{eq:upper_alpha} we used \eqref{eq:alpha_binary}. Since \eqref{eq:upper_alpha} does not depend on $u,v,T$, we have \eqref{eq:alpha} with $\alpha(1) \ge (1+\alpha)^{-1}$. This completes the proof.
\end{proof}
\begin{proof}[Proof of \Cref{thm:LCRW_transient}]
Let us iteratively remove the leaves of $T'$ to recover a subtree where every vertex has degree at least 2, call this $T_{\text{core}}$. Observe that the lemma is true for $T'$ if and only if it is true for $T_{\text{core}}$. Indeed, every connected component of $T' \setminus T_{\text{core}}$ is a finite tree and there is only one choice for the MSA on the edges of these trees once the choice of the graph in the exhaustion of $T'$ gets large enough (see \Cref{wlgp_0}). So let us assume that $T' = T_{\text{core}}$. Let us also assume that $T'$ is an $M$-subdivision of a tree with degree at least 3.

By \Cref{lem:transience_equiv} it is enough to prove that $S_i \cap \vec E_i = \emptyset$ infinitely often has probability 0 for CLEB walk started at any vertex.
Fix a vertex $v$, take an $n$ large and consider $T_n$, the tree with all vertices at most distance $n$ from $v$.  Observe that using \Cref{prop:hitting_prop} and \Cref{eq:alpha}, the probability that the LCRW from $v$ in $T_n^{\w}$ hits $\partial$ before returning to $v$ is at least $\alpha(M)$. If the LCRW returns to $v$ before hitting $\partial$, then the tree changes to some other  $T'_n$ and $v$ is replaced by some $v'$. However, observe that $T'_n$ is still an $M$-subdivision of a tree with degree at least 3. Hence by \eqref{eq:alpha}, the simple random walk still has probability at least $\alpha(M)$ to hit $\partial$ before returning to $v' $. Thus using \Cref{prop:hitting_prop} again, we see that the LCRW also has probability at least $\alpha(M)$ to hit $\partial $ before returning to $v'$. Thus the probability that LCRW returns $k$ times is at most $(1-\alpha(M))^k$ for any $n$.  Taking the limit $n \to \infty$ and then taking $k \to \infty$, we conclude.
\end{proof}

\begin{proof}[Proof of \Cref{thm:convergence}]
Follows immediately from \Cref{prop:necessary_cond,lem:CLEB_LCRW,thm:LCRW_transient}.
\end{proof}
We finish this section with the proof of \Cref{thm:transient_GW}. The idea of the proof is similar to that of \Cref{thm:convergence}, but instead of \Cref{lem:hitting}, we use the randomness of the Galton-Watson tree to our advantage.
\begin{proof}[Proof of \Cref{thm:transient_GW}]

It is well-known that the simple random walk is transient almost surely on both $T_{\text{GW}}$ and $T_{\text{UGW}}$ (see \cite{GK_GW,LPP_GW}). Thus the probability that the simple random walk from the root never returns to it is some constant $\alpha$ which depends only on the law of $Z$ (when averaged over the randomness of the tree).

Let us concentrate on $T_{\text{GW}}$, the proof for $T_{\text{UGW}}$ is similar. Recall the notations $P_i,X_i$ from \Cref{sec:LCRW}. However, instead of $G_i$, we will reveal a portion of 
the tree $T_i$ 
simultaneously performing the LCRW  as follows.
\begin{itemize}
    \item In the first step, reveal all the neighbors of $X_0$ in $T_{\text{GW}}$, the number of such neighbors has the same law as $Z$.
    \item Inductively, assume that all the neighbors of vertices in $P_i$ in $G_i$ are in $T_i$ and (hence $X_i$ is not a leaf of $T_i$). Perform a loop contracting random walk step, which is just a simple random walk step in $T_i$. If we reach a leaf, reveal all its children (according to the law $Z$) and call this new tree $T_{i+1}$.
\end{itemize}
Borrowing notations from \Cref{sec:CLEB_walk}, let $(\tau_k)_{k \ge 1}$ be the successive times when $P_i$ becomes a singleton. For each $k \ge 0$, conditionally on the $\sigma$-algebra generated by $T_{\tau_k}$ (call it $\cF_{k}$), the law of the trees attached to each neighbor of the singleton $P_{\tau_k}$ are distributed as i.i.d.\ $T_{\text{GW}}$. Thus conditionally on $\cF_k$, using \Cref{prop:hitting_prop}, the probability that the LCRW never comes back to hit $P_{\tau_k} $ is at least $\alpha$. Since this probability does not depend on $k$, the probability that the LCRW returns $n$ times is at most $(1-\alpha)^n$. Letting $n \to \infty$, the proof is complete.
\end{proof}


\section{Ends of the MSA}\label{sec:end}
In this section, we prove \Cref{thm:end,thm:end_GW}.
\subsection{Perturbating the MSA}
We first prove a stability under perturbation result for the MSA. Namely, we show that locally changing the weight in a certain way, only changes the MSA for the changed weights locally. This will allow us a way to deform the MSA to obtain a new arborescence that in turn will be an MSA for a modification of the weights. We prove two such perturbation lemmas, the second one depends on the knowledge of the CLEB walk algorithm, and the first one does not. In what follows, we only need the second one, but we still keep the statement and proof of the first perturbation lemma for potential future use. (The reader may wish to skip \Cref{lem:perturbation} as it is not used later).

The reader may refer to \Cref{fig:pert1} while reading the statement of the following lemma. 
Recall that we denote by $\cO(v)$ the set of outgoing edges from $v$. 
%
%
%
%
\begin{figure}[ht]
\centering
\includegraphics[scale = 0.7]{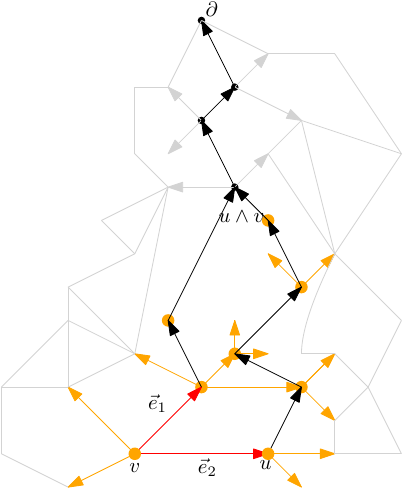}
\caption{Orange vertices denote $\Sigma$. In \Cref{lem:perturbation}, the weights of the orange outgoing edges and that of $\vec e_1$ are increased to at least $2M_\Sigma$. }\label{fig:pert1}
\end{figure}

\begin{lemma}[Perturbating the MSA: I]\label{lem:perturbation}
Suppose $G$ is a finite, connected graph, possibly with multiple edges but no self-loops and a boundary vertex $\partial$.
 Suppose $T^*(w)$ is the  MSA of $G$ with  weight function $w$ satisfying \eqref{eq:linear_comb}. Fix a vertex $v$ and suppose $\vec e_1,\ldots ,\vec e_d$ is the set of outgoing edges from $v$, let $u = (\vec e_2)_+$ and suppose $\vec e_1 \in T^*(w)$. Suppose $v \not \in \mathfrak F_{T^*}(u)$ and let $\Sigma = \Sigma(u,v)$ be the set of endpoints of the futures of $u$ and $v$ until they merge, i.e., the endpoints of 
 $$
 (\mathfrak F_{T^*}(u)  \cup \mathfrak F_{T^*}(v) ) \setminus (\mathfrak F_{T^*}(u)  \cap \mathfrak F_{T^*}(v)). $$
 except $u \wedge v$.
Let $\vec E_\Sigma = \{\cO(v): v \in  \Sigma\}$ and define $M_\Sigma := \max\{w(\vec e): \vec e \in \vec E_\Sigma \}$. Choose a new collection of weights $w'$ such that it satisfies \eqref{eq:linear_comb} and 
  \begin{equation}
 w'(\vec e ) 
 \begin{cases}
 \ge 2M_\Sigma \text{ if $\vec e \in \vec E_\Sigma \setminus (T^* \cup \{\vec e_2\})$ or $\vec e = \vec e_1$}\\\label{eq:w'}
 = w(\vec e) \text{ otherwise.}
 \end{cases}
 \end{equation}
Then $T':=(T^*(w) \setminus \{e_1\}) \cup \{e_2\} $ is the MSA for the weights $w'$. 
\end{lemma}
\begin{proof}
First of all, note that the condition $v \not \in \mathfrak F_{T^*}(u)$ makes $T'$ a spanning arborescence of $(G, \partial)$. Let $\vec E'$ be the set of all oriented edges in $\vec E$ except those outgoing from $v$. Observe that
\begin{align}
    w(\vec e_2) & \ge w(\vec e_1), \label{eq:e2e1}\\
    w(T') & = w'(T'), \label{eq:ww'T'}\\
T' \cap \vec E' = T^* \cap \vec E' \text{ and hence } w(T' \cap \vec E') &= w(T^* \cap \vec E') = w'(T' \cap \vec E') = w'(T^* \cap \vec E').\label{eq:ww'T'T*}
\end{align}
Indeed, \eqref{eq:e2e1} follows since $T^*$ is the MSA for $w$ and both \eqref{eq:ww'T'} and \eqref{eq:ww'T'T*} follow from the definitions of $T'$ and $w'$.

Pick any spanning arborescence $U$  of $(G, \partial)$ and suppose the edge outgoing from $v$ in $U$ is $\vec e_i$.  Since $T^*(w)$ is the MSA for $w$, we must have
\begin{equation*}
w(\vec e_1) + w(T^*(w ) \cap \vec  E') \le w(\vec e_i) + w(U \cap \vec E').
\end{equation*}
Therefore,
\begin{equation}
w(T^*(w ) \cap \vec E') \le w(U \cap \vec E') + M_\Sigma \le w'(U \cap \vec E')+M_\Sigma \label{eq:estimate1}
\end{equation}
since by definition, $w'(\vec e) \ge w(\vec e)$  for all $\vec e$.
Now assume $\vec e_i \neq \vec e_2$. This is actually the easy case as $w'(\vec e_i) \ge 2M_\Sigma$ by definition. Indeed,
\begin{align*}
w'(U) = w'(\vec e_i) + w'(U \cap \vec E') &\ge w(\vec e_2) +M_\Sigma +  w(U \cap \vec E') \\ & \ge w(T^*(w ) \cap \vec E' )  + w(\vec e_2)- M_\Sigma+M_\Sigma = w(T') = w'(T'),
\end{align*}
and thus $T'$ is the MSA for the weights $w'$.
Here the first inequality follows from the definition of $w'$, the second inequality follows from \eqref{eq:estimate1} and the last two equalities follow from \eqref{eq:ww'T'} and \eqref{eq:ww'T'T*}.
From now on, we assume $\vec e_i = \vec e_2$. We divide it into two cases.
\paragraph{Case 1:}
Suppose $\mathfrak F_U(u)$ and $\mathfrak F_U((\vec e_1)_+)$ is the same as $\mathfrak F_{T'}(u)$ and $\mathfrak F_{T'}((\vec e_1)_+)$ until they merge. (Here we do not gain anything from the orange edges in \Cref{fig:pert1}.) Observe the following:
\begin{equation}
  w'(U \cap \vec E') \ge w(U \cap \vec E' )  \ge w(T' \cap \vec E') = w(T^* \cap \vec E')\label{eq:relation}
\end{equation}
The last equality is true by \eqref{eq:ww'T'T*}, and if the inequality is false, then 
 \begin{equation*}
w(U \setminus \{\vec e_2 \} \cup \{\vec e_1\})  = w(\vec e_1) + w(U \cap \vec E')  < w(\vec e_1) + w(T^* \cap \vec E') = w(T^*),
\end{equation*}
a contradiction since $T^*$ is the MSA for weight $w$.  Thus in this case, \begin{equation*}
w'(U) = w'(\vec e_2) + w'(U \cap \vec E') \ge w(\vec e_2) + w(U \cap \vec E')  \ge  w(\vec e_2) + w(T' \cap \vec E') = w(T') = w'(T').
\end{equation*}
where the last equality again follows from \eqref{eq:ww'T'}. Thus in this case $T'$ is the MSA for $w'$ as well.

\paragraph{Case 2:}
Now suppose that the future of $u$ and that of $ (\vec e_1)_+$ of $U$ is not the same as that of $T'$ until they merge, i.e., there exists a vertex $y \in \Sigma$ such that the outgoing edge from $y$ in $U$ is not in $T'$. Call this edge $\vec e_y$. Note that $w'(\vec e_y) - w(\vec e_y) \ge M_S$ by definition.

\begin{align}
w'(T') & = w'(\vec e_2 )  + w'(T' \cap \vec E') \\ \nonumber
& \le w'(\vec e_2) + w(U \cap \vec E') + M_S \\
& \le w'(\vec e_2) + w'(U \cap \vec E') -M_S+M_S \\
& \le w'(U ).
\end{align}
The first inequality follows from \eqref{eq:ww'T'T*} and \eqref{eq:estimate1}. The second inequality follows since $w'(\vec e) \ge w(\vec e) $ for all $\vec e$ and $w'(\vec e_y) \ge w(\vec e_y) + M_S$. This shows that $T' $ is the MSA for the weight $w'$ in this case as well, and the proof is complete.
\end{proof}

Although \Cref{lem:perturbation} is a natural perturbation procedure, it is as local as $\Sigma$ is. We now prove a different version of \Cref{lem:perturbation} which will be useful when $\Sigma $ is not local, particularly in the case when $u \wedge v = \partial $ and later in the infinite volume case when the CLEB walk is a.s.\ transient. Unlike \Cref{lem:perturbation}, we shall make heavy use of the results in \Cref{sec:CLEB_walk} in the statement and the proof of the next lemma. To that end,  recall the notation $S_{\tau_{N_1}}$ from \Cref{sec:CLEB_walk}: roughly it is the set of edges exposed by the CLEB walk until a time after which $P_i$ is never a singleton.

\begin{lemma}[Perturbing the MSA: II]\label{lem:perturbation2}
Suppose we are in the setup of \Cref{lem:perturbation} and assume $v \not \in \mathfrak F_{T^*}(u)$.  Let $\Sigma = \Sigma (v)$ be the set of endpoints of  $S_{\tau_{N_1}+1}$ except $(f_1)_+$ where $S_{\tau_{N_1}}, f_1$ are defined as in \Cref{sec:CLEB_walk} for the CLEB walk starting at $v$. Now define $\vec E_\Sigma,M_\Sigma, w', T'$ as in \Cref{lem:perturbation}. Then $T'$ is the MSA for the weights $w'$.
\end{lemma}
\begin{proof}
Let $\partial':= V(S_{\tau_{N_1}}) \cup \{\partial\}$.
Perform CLEB walk algorithm with boundary $\partial'$, in any order. When the algorithm is complete, say at step $\vartheta=\vartheta(\partial')$ we are left with a graph $G'$ (formed by contracting cycles) and a spanning arborescence $Q$ of $G'$ with boundary $\partial'$. Note that the CLEB walk on $(G, \partial')$ for the weights $w $ and $w'$ are the same up to step $\vartheta $ since none of the outgoing edges of $S_{\tau_{N_1}}$ has been exposed up to this step. The latter observation also implies that $\vec e_1, \vec e_2, \vec f_1$ are all in $G'$ (i.e., are not contracted yet). Also, recall that the root of every component of $Q$ is in $\partial'$ and every vertex $v$ in $\partial'$ corresponds to a unique component in $Q$ whose root is $v$. Finally, by definition of $S_{\tau_{N_1}}$, the component of $\vec f_1^+$ in $G'$ has root $\partial$. Let $V_{\text{int}}(Q)$ be the set of non-boundary vertices of $Q$. The idea now is to complete the CLEB algorithm on the remaining vertices with weight $w'$ and compare it with that of weight $w$. 

Define $\text{Past}_{T^*}(v)$ be the set of vertices $u$ such that $ v \in \mathfrak F_{T^*}(u)$. We now claim the following.
\begin{claim}\label{claim:not_in_past}
If  $(\vec e_2)_+ \in V_{\text{int}}(Q)$, then $\mathfrak F_{Q}((\vec e_2)_+)$ cannot contain any element from  $\text{Past}_{T^*}(v) \cap V(S_{\tau_{N_1}}) $.
\end{claim}
\begin{proof}
We shall prove this claim by contradiction, so let us assume this is the case. Clearly, $\vec e_2$ is not in $S_{\tau_{N_1}}$. Let us complete the sequential CLEB algorithm with weight $w$ by starting the CLEB walk from $v$ with weight $w$ until a vertex outside $V(S_{\tau_{N_1}})$ is hit. By the definition of $S_{\tau_{N_1}}$, the edges exposed by this process is exactly $S_{\tau_{N_1}}$ and the process stops by hitting $(\vec f_1)_+$ (i.e., all the vertices of $V(S_{\tau_{N_1}})$ are contracted into a single vertex). This yields a contracted graph $\tilde G$ with a spanning arborescence $\tilde Q$ of $(\tilde G, \partial)$. Furthermore, $\tilde Q$ can be obtained from $Q$ by identifying all the vertices of $V(S_{\tau_{N_1}})$ into a single vertex, removing all the self-loops, and adding the edge $\vec f_1$. Let $\cL$ be the ordered sequence of loops contracted by the CLEB walk started from $v$ until $(\vec f_1)_+$ is hit. Now we uncontract the loops in $\cL$ in reverse order, starting from $((\tilde G, \partial), \tilde Q)$. By \Cref{lem:recovery}, this yields a spanning arborescence of $(G', \partial)$ which can be seen as the union of $T^* \cap S_{\tau_{N_1}} , Q$ and $\vec f_1$. Now further uncontracting the remaining loops in reverse order of course yields $T^*$ by definition. But furthermore by \Cref{lem:recovery}, this uncontraction procedure can be carried out separately for every component of $Q$ and each such procedure yields an arborescence with the same root as that component.
 Thus if $(\vec e_2)_+$ in $G'$ is a vertex in $V_{\text{int}}(Q)$ whose future contains  $\text{Past}_{T^*}(v) \cap S_{\tau_{N_1}}$, then $(\vec e_2)_+$ in $G$, which is $u$, satisfies $\mathfrak F_{T^*}(u) \ni v$. This is a contradiction by our assumption on the future of $u$ in $T^*$ and the claim is proved.
\end{proof}
Now suppose we start from $Q$, order the vertices in $V(S_{\tau_{N_1}})$ in some arbitrary order, except making sure that $v$ is the last vertex, and finish the sequential CLEB algorithm for the weight $w'$. Observe that by the choice of $w'$, for every vertex $y \neq v$ in $V(S_{\tau_{N_1}})$, we simply expose the outgoing edge from $y$ in $T^*$, which is the same as that of $T'$. Thus when every vertex except the last vertex $v$ has been explored, we expose $T^* \cap S_{\tau_{N_1}+1} \setminus \{\vec e_1\}$. Now when we choose $v$, the minimum $w'$-weight outgoing edge is $\vec e_2$. 

If $(\vec e_2)_+ \in S_{\tau_{N_1}}$, then by \Cref{lem:recovery} and our assumption, $v \not \in \mathfrak F_{T^*}(u) $, thus adding $\vec e_2$ does not create a cycle. On the other hand, if $(\vec e_2)_+ \not \in S_{\tau_{N_1}}$, by \Cref{claim:not_in_past}, adding $\vec e_2$ again does not create a cycle.  Thus this completes the contraction part of the sequential CLEB algorithm, and overall, no cycle is formed while exploring $V(S_{\tau_{N_1}})$.  In the uncontraction phase of the algorithm,  since none of the cycles contracted has any outgoing edge from $S_{\tau_{N_1}}$, and $Q$ is the same for both $w$ and $w'$, uncontracting the remaining cycles yield $T^* \setminus \{\vec e_1\} \cup \{ \vec e_2 \}$ as the MSA for $w'$. This completes the proof.
\end{proof}

\begin{remark}
It is known that in an MST, changing the weight of one edge changes the MST  at most two other edges (\cite[Lemma 3.15]{lyons2006minimal}). It is not hard to convince oneself, that if $S_{\tau_{N_1}}$ is a large set, the number of edges that changes by changing the weight of an edge can be arbitrarily large.
\end{remark}

Let us now prove an infinite volume version of \Cref{lem:perturbation,lem:perturbation2} for a generic collection of weights.

\begin{lemma}\label{lem:perturb_infinite}
Suppose $G = (V,E)$ is a locally finite, infinite, connected graph with  random weights $(W({\vec e}))_{\vec e \in \vec E}$ satisfying 
\begin{itemize}
\item $(W({\vec e}))_{\vec e \in \vec E}$ is a.s.\ generic (see \eqref{eq:linear_comb}),
\item CLEB walk is a.s.\ transient for all vertices and hence a.s.\ the wired limit of MSA exists, call it $T^*$.
\end{itemize}
Suppose $\mathsf v \in V$ and $\vec e_1,\vec e_2$ are picked measurably with respect to $W$ such that the tails of both $\vec e_1,\vec e_2$ is $\mathsf v$ and $\mathsf v \not \in \mathfrak F_{T^*}(\mathsf u)$ where $\mathsf u  = (\vec e_2)_+$ (whence $T' := T^*(W)\setminus \{\vec e_1 \} \cup \{\vec e_2\}$ is an arborescence). In case $\mathsf u \wedge \mathsf v \neq \emptyset$ pick  $\Sigma (\mathsf u, \mathsf v)$ according to \Cref{lem:perturbation} or pick $\Sigma(\mathsf v)$ according to \Cref{lem:perturbation2}. In case $\mathsf u \wedge \mathsf v =\emptyset$ pick $\Sigma(\mathsf u,\mathsf v)$ according to \Cref{lem:perturbation2}. Let $\vec E_\Sigma := \{\cO(v): v \in \Sigma\}$ and  $M_\Sigma = \max \{W(\vec e): \vec e \in \Sigma_{\vec E}\}$.
Pick a new collection of weights: 
  \begin{equation*}
 W'(\vec e ) 
 \begin{cases}
 \ge 2M_\Sigma \text{ if $\vec e \in \vec E_\Sigma \setminus (T^* \cup \{\vec e_2\})$ or $\vec e = \vec e_1$}\\
 = W(\vec e) \text{ otherwise.}
 \end{cases}
 \end{equation*}
 such that $W'(\vec e)$ is also a.s. generic.
Then $T'$ is the wired MSA limit for the weights $W'$.
\end{lemma}
\begin{proof}
Let $T^*_n$ be the MSA of $G_n^{\mathsf{w}}$ where $(G_n)_{n \ge 1}$ is an exhaustion of $G$. Let $\cE_{v,a,b}$ be the event that $\mathsf v = v, \vec e_1 = \vec a, \vec e_2 = \vec b$.  

We now claim that for almost all $\omega \in \cE_{v,a,b}$, there exists an $N(\omega)$ such that for all $n \ge N(\omega)$, $\mathsf v \not \in \mathfrak F_{T_n^*}(\mathsf u)$. To prove this claim, we partition $\cE_{v,a,b}$ into the union of two events: either $u$ and $v$ are in different components of $T^*$, or they are in the same component and $v \not \in \mathfrak F_{T^*}(u) $. In the former case, using the monotonicity of connectivity in \Cref{prop:connectivity_convergence}, we know that the components of $v$ and $u$ in $T^*_n$ are distinct for all $n \ge 1$, whence $\mathsf v \not \in \mathfrak F_{T_n^*}(\mathsf u)$ for all $n \ge 1$. In the latter case, since $T^*_n$ converges to $T^*$ a.s.\, for almost all $\omega \in \cE_{v,a,b}$, we can find find an $N(\omega)$ such that for all $n \ge N(\omega)$,  $\mathfrak F_{T^*_n}(u) , \mathfrak F_{T^*_n}(v)$ until they merge is the same as that of $\mathfrak F_{T^*}(u), \mathfrak F_{T^*}(v)$. Since $v \not \in  \mathfrak F_{T^*}(u)$, it is not in $\mathfrak F_{T^*_n}(u)$ as well for all such $n$.

We can increase $N$ further so that $\Sigma  \subset G_n$ for all $n \ge N$. Indeed, for the choice of $\Sigma$ from \Cref{lem:perturbation}, we can use the argument in the previous paragraph, and for the choice of $\Sigma$ from \Cref{lem:perturbation2}, we can use the assumption that the CLEB walk is a.s.\ transient from all vertices.
By \Cref{lem:perturbation,lem:perturbation}, $T_n'$ is in fact the MSA for the weights $(W'(\vec e))_{\vec e \in G_n^{\w}}$ for all $n \ge N$. Since $T_n'$ converges a.s.\ to $T'$, we have that $T'$ is the wired MSA limit for the weights $W'$ on $\cE_{v,a,b}$. Since $v,a,b$ was arbitrary, the proof is complete.
\end{proof}

We now argue that we can perturb weights as in \Cref{lem:perturb_infinite} so that the weights remain absolutely continuous. Let $\R_+ = [0, \infty)$. We start with a measure-theoretic fact. We say a probability density function $g$ is \textbf{good} if its support is $\R_+$ and for any $K>0$, there exists a constant $C(K)$ such that
\begin{equation}
    \sup_{u \in (0,\infty), x \in (0, K]} \frac{g(u-x)}{g(u)} \le C(K). \label{eq:good_g}
\end{equation}
A typical example of good densities for our application are exponential random variables.
\begin{lemma}\label{lem:abs_cont_easy}
Suppose ${\bf Z} = (Z_1,Z_2,\ldots,)$ is an i.i.d. sequences each having a good density $g$. Suppose ${\bf X}:= (X_1,X_2, \ldots)$ be a sequence which is independent of ${\bf Z}$ and its distribution is absolutely continuous with respect to that of ${\bf Z}$.     Suppose $\cI:= \{I_1,\ldots, I_d\} \subseteq \N$ is a.s.\ finite (and possibly empty) set of indices picked as a measurable function of ${\bf X}$  and let $X^* := \max\{X_{I_1}, \ldots, X_{I_d}\}$.  Define 
\begin{equation}
Y_{k} = 
\begin{cases}
2X^*+Z_{j} \text{ if  }k = I_j \text{ for some }  1 \le j \le d \\
X_{k}  \text{ if }k \not \in \cI
\end{cases}
\end{equation}
Then ${\bf Y}:= (Y_k)_{k \in \N}$ is absolutely continuous with respect to ${\bf Z}$.
\end{lemma}
\begin{proof}
We first observe that without loss of generality, we can assume ${\bf X}$ has the same distribution as that of ${\bf Z}$. Indeed, let $\psi$ be the measurable map as described in the lemma such that ${\bf Y}  = \psi({\bf X}, {\bf Z})$. Let ${\bf Z'} $ be an independent copy of ${\bf Z}$. Since $({\bf X}, {\bf Z})$ is absolutely continuous with respect to $({\bf Z'}, {\bf Z})$ by assumption, $\psi(({\bf X}, {\bf Z}))$ is absolutely continuous with respect to $\psi({\bf Z'}, {\bf Z})$. Thus if we show that the distribution of $\psi({\bf Z'}, {\bf Z})$ is absolutely continuous with respect to the distribution of ${\bf Z}$, we are done. In other words, we can assume ${\bf X} = {\bf Z}$ in distribution, which we do from now on.

For any sequence $x$ and $S \subset \N$ a subsequence  (finite or infinite), let $x_S$ denote the sequence $(x_i)_{i \in S}$.
Let $\mu_{\bf X}$ be the measure induced by $\bf X$ on $\R_+^\N$.
    For a measurable set $A$, $x \in \R_+^\N$ and a finite set $S \subset \N$, define
\[ A^{x, S} = \{ z_S \in [0, \infty)^S:   ( z_S  + 2\max_{s \in S} x_s , x_{S^c} )  \in A \} \subseteq \R_+^S,\]
and let $ {\bf Z}_S = (Z_x)_{x \in S}$. 
Note that
\begin{align}
\P( {\bf Y}  \in A, \cI = S) &=  \int_{x: \cI(x) = S}  \P( {\bf Z}_S \in A^{x, S})  \mu_{\bf X}(dx) \nonumber\\
&= \int_{ x: \cI(x) = S}  \int_{ z_S \in [0, \infty)^S} \1_{ \{ (z_S  + 2\max_{s \in S} x_s , x_{S^c} )  \in A \}}   \prod_{s \in S}g(z_s) dz_S \mu_{\bf X}(dx) \nonumber\\
&=  \int_{ x: \cI(x) = S }  \int_{ u_S \ge 2\max_{s \in S} x_s  } \1_{ \{ (u_S , x_{S^c} )  \in A \}}  \prod_{s \in S}g(u_s  -  2\max_{s \in S} x_s)  du_S \mu_{\bf X}(dx)\label{eq:measure1}
\end{align}
where in the final equation, we applied the change of variable $ u_s = z_s +  2\max_{s \in S} x_s$ for $s \in S$.

Now assume $\mu_{\bf X}(A) = 0$. By monotone convergence theorem and \eqref{eq:measure1}, it is enough to prove for every finite $S$ and $K>0$,
\begin{equation}
    \int_{ x: \cI(x) = S , x_S \in [0,K]^S}  \int_{ u_S \ge 2\max_{s \in S} x_s  } \1_{ \{ (u_S , x_{S^c} )  \in A \}}  \prod_{s \in S}\frac{g(u_s  -  2\max_{s \in S} x_s)}{g(u_s)} g(u_s)  du_S \mu_{\bf X}(dx)=0 \label{eq:measure2}
\end{equation}
Now we use the condition that $g$ is good, and conclude that there exists some constant $C(K)$ such that
$$
\prod_{s \in S}\frac{g(u_s  -  2\max_{s \in S} x_s)}{g(u_s)} \le (C(2K))^{|S|} \qquad \text{ for all $u_S$ and for all $x_S \in [0,K]^S$ } .
$$
Applying this and then getting rid of the conditions $\{\cI(x) = S , x_S \in [0,K]^S \}$ and $\{u_S \ge 2\max_{s \in S}x_s\}$, the left hand side of \eqref{eq:measure2} can be upper bounded by 
\begin{equation}
   (C(2K))^{|S|}  \int_{\R_+^S} \int_{\R_+^S \times \R_+^{S^c}} \1_{ \{ (u_S , x_{S^c} )  \in A \}} d\mu_{{\bf X}_S}(du_S) d\mu_{\bf X}(dx) = (C(2K))^{|S|} \mu_{\bf X}(A) = 0,
\end{equation}
as desired. In the last step we used the fact that ${\bf X}_S$ has the density $\prod_{s \in S}g(u_s)$.
\end{proof}
\begin{remark}
    The condition \eqref{eq:good_g} is assumed for the sake of brevity. It is possible to modify \Cref{lem:abs_cont_easy} in various ways for different choices of distributions, but we do not pursue such avenues in this article.
\end{remark}

\begin{corollary}\label{lem:abs_cont_W'}
Suppose we are in the setup of \Cref{lem:perturb_infinite}. Let $W:=(W(\vec e))_{\vec e}$ be i.i.d. with a good density (satisfying \eqref{eq:good_g}) and suppose it satisfies the assumptions in \Cref{lem:perturb_infinite} and let $T^*$ be the a.s.\ wired MSA limit for these weights. Let $(Z(\vec e))_{\vec e}$ be i.i.d. and distributed as $g$, independent of $W$. Let $\varphi$ be the operation $(W,T^*) \mapsto (W',T')$ defined as follows:
\begin{itemize}
\item Measurably pick a vertex $\sf v$ and edges $\vec e_1, \vec e_2$ so that $\sf v$ is the tail of both $e_1,e_2$, and such that ${\sf v}  \not \in \mathfrak F_{T^*}(\mathsf u)$ with $\mathsf u = (\vec e_2)_+$.

\item Pick $\Sigma, M_\Sigma, \vec E_\Sigma$ as in \Cref{lem:perturb_infinite}. Define 
 \begin{equation*}
 W'(\vec e ) 
=  \begin{cases}
  2M_\Sigma + Z(\vec e) \text{ if $\vec e \in \vec E_\Sigma \setminus (T^* \cup \{\vec e_2\})$ or $\vec e = \vec e_1$},\\
  W(\vec e) \text{ otherwise.}
 \end{cases}
 \end{equation*}
 \item Define $T' = T^* \setminus \{e_1\} \cup \{e_2\}$.
\end{itemize}
Let $(W_n, T_n) := \varphi^{(n)} (W,T^* ) $ where $\varphi^{(n)}$ is the $n$-fold composition of $\varphi$ where for each iteration we use an independent set of i.i.d. $(Z(\vec e)_{\vec e \in \vec E})$ distributed with density $g$. Then $W_n$ is absolutely continuous with respect to $W$ and $T_n$ is a wired MSA limit of $W_n$.
\end{corollary}
\begin{proof}
This is a consequence of the repeated application of \Cref{lem:perturb_infinite,lem:abs_cont_easy}. Indeed,  suppose $W_n$ is absolutely continuous with respect to $W$ and let $T_n$ be the wired MSA limit for the weights $W_n$ (this is true for $n=0$ by assumption). Then by absolute continuity, $W_n$ satisfies the assumptions in \Cref{lem:perturb_infinite}. Thus, $T_{n+1}$ is the wired MSA limit of $W_{n+1}$ by \Cref{lem:perturb_infinite}. Furthermore, since, $W_n$ is absolutely continuous with respect to $W$,  $W_{n+1}$ is absolutely continuous with respect to $W$ as well using \Cref{lem:abs_cont_easy}. 
\end{proof}

\subsection{Ends of MSA in unimodular graphs}\label{sec:end_unimod}
We first recall the notion of unimodular random rooted graphs.
A \textbf{rooted graph} is a locally finite, connected graph $G$ where one vertex $\rho$ is specified as the root.
We will also consider \textbf{marked rooted graphs} which are triplets $(G, \rho, m)$ where $m:E\cup V \mapsto \Omega$ is a mapping to a Polish space $\Omega$.
Two (marked) rooted graphs are equivalent if there is a graph isomorphism between them that preserves the root and the marks.

Let $\cG^\bullet$ (resp.\ $\cGm$) denote the space of equivalence classes of rooted graphs (resp.\ marked rooted graphs) endowed with the \textbf{local topology}:
informally, two marked rooted graphs are close if for some large $R$ the balls of radius $R$ are isomorphic as graphs, and the restriction of the marks to the balls are uniformly close
(see \cite{benjamini2011recurrence}).
We similarly use $\cG_{\mathsf m}^{\bullet \bullet}$ to denote the space of equivalence classes of doubly rooted marked graphs $(G, \rho_1,\rho_2, m)$.
 A random rooted marked graph $(G, \rho, M) \in \cG_{\mathsf m}^{\bullet}$ is \textbf{unimodular} if it satisfies the mass transport principle, i.e., for all measurable $f:\cG_{\mathsf m}^{\bullet\bullet} \mapsto [0, \infty)$,
\begin{equation*}
\E\left(\sum_{x \in V(G)} f((G,\rho,x,M))\right) = \E\left(\sum_{x \in V(G)} f((G,x,\rho,M)) \right).
\end{equation*}
Informally, we regard $f((G,x,y,M))$ as the mass sent from $x$ to $y$, and then unimodularity means that the expected total mass sent by $\rho$ equals the expected total mass received by $\rho$.

Suppose $(G, \rho)$ is a unimodular, random rooted directed graph where the edge set $\vec E$ is oriented in both directions. Suppose that given $(G, \rho)$, every oriented edge is endowed with i.i.d.\ random variables satisfying \eqref{eq:linear_comb}, denoted $(U_{\vec e})_{\vec e \in \vec E}$. Note that by unimodularity, the CLEB walk is almost surely transient at $\rho$ if and only if it is transient from all vertices in $G$ almost surely (see \cite[Lemma 2.3]{AL_unimodular} or \cite{curiennotes}). In that case, we say that $(G, \rho)$ is a random rooted graph which is a.s.\ transient for the CLEB walk. Recall that if $(G, \rho)$ is unimodular and a.s.\ transient for the CLEB walk, then the wired MSA limit exists almost surely (\Cref{prop:necessary_cond}). Throughout this section, we assume $(G, \rho)$ is a random rooted graph that is connected, locally finite, and infinite almost surely, which may contain multiple edges but no self-loops.
\begin{lemma}\label{lem:MSA_unimodular}
Suppose $(G, \rho)$ is a unimodular random rooted graph with i.i.d.\ weights $W:=(W_{\vec e})_{\vec e \in \vec E}$ which is a.s.\ transient for the CLEB walk. Let $T$ be the wired MSA limit on $(G, \rho)$. Then $(G, \rho, T)$ is a unimodular random rooted marked graph.
\end{lemma}
\begin{proof}
It is easy to see that $(G, \rho, W)$ is a unimodular random rooted marked graph. Since $T$ is a measurable function of $(G, \rho, W)$ and this function does not depend on the location of the root, $(G, \rho, W, T)$ is a unimodular random rooted marked graph (in fact, $(G, \rho, T)$ is what is called a graph factor of i.i.d., but we do not need this fact, see \cite[Section~2.1]{loopO1} for further discussions).
\end{proof}


\begin{lemma}\label{lem:MSA_properties}
Let $(G, \rho)$ be unimodular and a.s.\ transient for the CLEB walk, and let $T$ be the wired MSA limit on $(G, \rho)$. Then the following holds.
\begin{enumerate}[(a)]
\item The expected number of total incoming and outgoing edges in $T$ incident to $\rho$  is 2.
\item Every component of $T$ has at most 2 ends almost surely.
\item If $(G, \rho)$ is invariantly nonamenable, then $T$ has infinitely many infinite components. Furthermore, the number of infinite components that are two-ended is either 0 or $\infty$ a.s.\ and the same statement holds for the number of one-ended infinite components.
\end{enumerate}
\end{lemma}
\begin{proof}
The proofs are quite standard (see, e.g., \cite{AL_unimodular}), but we still write them here for completeness. We define a mass transport that transports mass 1 from $x$ to $y$ if there is an oriented edge of $T$ from $x$ to $y$. The mass out of $\rho$ is $1$. The expected mass in to $\rho$ is also 1 by the mass transport principle. Consequently, the expected degree of $\rho$ is 2. Part (b) follows from \cite[Theorem 6.2]{AL_unimodular}.

For part (c) note that each component of $T$ is necessarily infinite as is clear from the CLEB algorithm. Suppose the number of infinite components is finite with positive probability. Conditioned on that event, pick one component uniformly at random, and let $\tau$ be the (unoriented) bond percolation configuration obtained by opening an edge of $G$ if and only if it is occupied by some oriented edge of $\tau$.  This gives an invariant percolation $\tau$ on $(G, \rho)$, which is connected a.s.\, and the critical Bernoulli site percolation threshold equal to 1 on $\tau$ since $\tau$ has at most two ends a.s.\ by part (b). This in particular implies $(G, \rho)$ is amenable by \cite[Theorem 8.9]{AL_unimodular}, a contradiction.

For the number of two-ended or one-ended infinite components, if there are finitely many with positive probability, condition on that event and select a component uniformly at random. This leads to the same contradiction.
\end{proof}
\begin{figure}[ht]
\centering
\includegraphics[scale = 0.5]{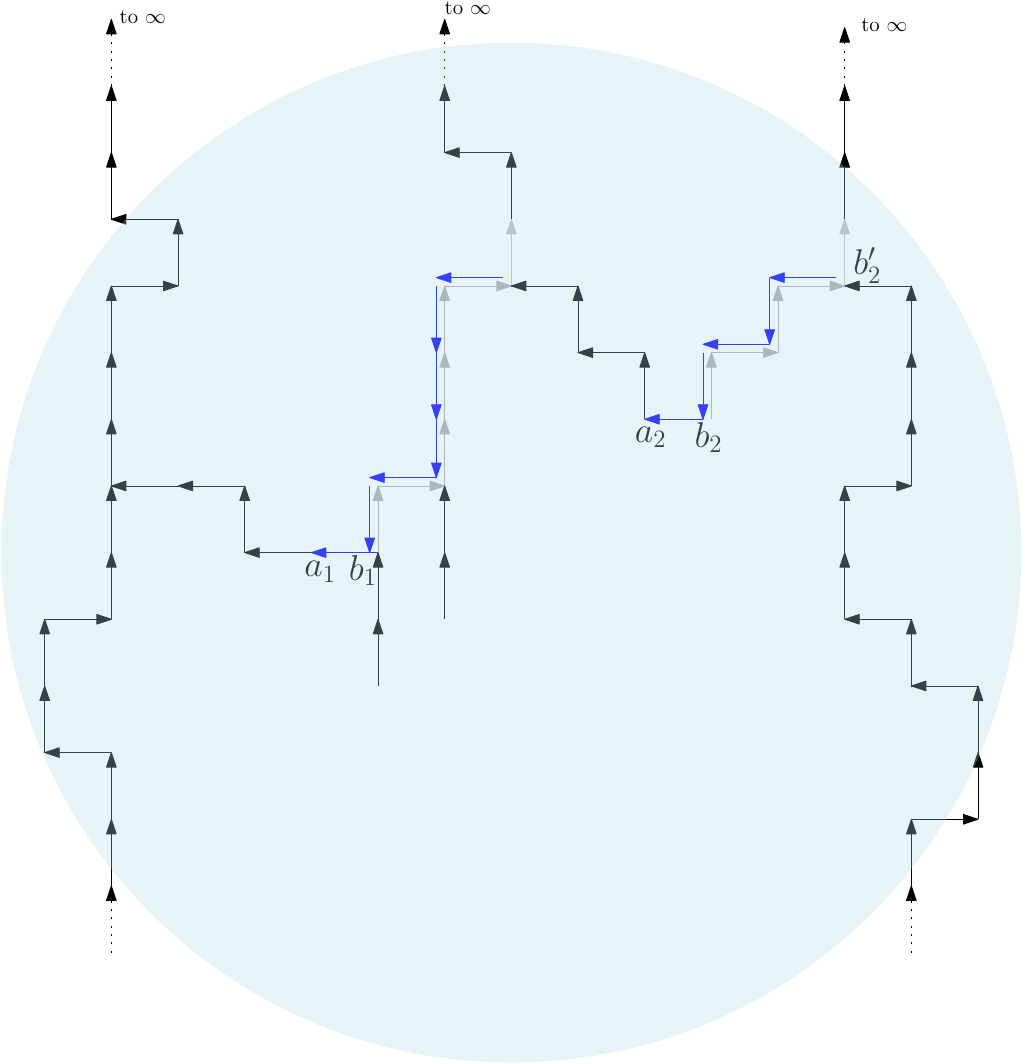}
\caption{The surgery in the proof of \Cref{thm:end}. The gray edges are the edges removed and replaced by the blue edges. The ball $B_G(\rho, M)$ is shaded blue.}\label{fig:surgery}

\end{figure}

\begin{proof}[Proof of \Cref{thm:end}]
As already explained in the paragraph before \Cref{lem:MSA_unimodular}, the wired MSA limit exists.   Also note that the number of components is infinite follows from \Cref{lem:MSA_properties}, part~(c). so we concentrate on proving that every component is a.s.\ one ended.

Recall that $B_G(\rho, a)$ denotes the graph distance ball of radius $a$ in $G$.
Using \Cref{lem:MSA_properties}, we know that each component of $T$ is at most two-ended, and suppose by contradiction that with positive probability there exists a component that is two-ended. By part (c) of \Cref{lem:MSA_properties}, we know that there are infinitely many two-ended components with probability $c>0$. Recall the notation $\mathfrak F(v)$ which denotes the future of $v$ and the notation $u \wedge v$ which denotes the vertex where $\mathfrak F(v)$ and $\mathfrak F(u)$ merge. Iteratively removing the leaves from a two-ended component leaves us with a bi-infinite path, which we call the \textbf{spine} of the component. The idea now is to merge two 2-ended components into a single component with at least $3$ ends by modifying the weights, making sure that the modification yields an absolutely continuous measure, and that the spanning arborescence remains an MSA for the modified weights. This yields a contradiction by \Cref{lem:MSA_unimodular} part (b). The tool we use to do this surgery is \Cref{lem:abs_cont_W'}.

We recommend referring to \Cref{fig:surgery} while reading the following surgery. For $M>0$, let $\cG_M$ be the event: 
\begin{itemize}
\item $B_G(\rho, M)$ intersects at least two infinite components $C$ and $C'$ which are two-ended.
\item There exist components $(C=C_1,C_2,\ldots, C_{k-1},C_k = C')$ such that there exists oriented edges $\vec e_1,\ldots, \vec e_{k-1}$ in $B_G(\rho, M)$ none of which belongs to $T$ and such that $\vec e_i $ has its tail $a_i$ in $C_i$ and head $b_i$ in $C_{i+1}$.
\item $b_i \wedge a_{i+1} \in B_G(\rho, M) $ for all $1 \le i  \le k-2$. 
\item $b_k'$ is in $B_G(\rho, M)$ where $b_k'$ is the vertex where the future of $b_k$ merges with the spine of $C'$. 
\end{itemize}
It is easy to see that if we pick $M$ large enough so that $\P(\cG_M) > c/2$. Indeed, by the almost sure existence of infinitely many two-ended components, the first item is satisfied for large $M$ with probability at least $c/2$. Now take a path $(v_1,\ldots, v_m)$ with  $v_1 \in C$ and $v_m \in C'$ inside $B_G(\rho, M)$. Fix $j_0 = 0$ and iteratively for $i \ge 0$, let $j_{i+1}$ be the largest index bigger than $j_i$ and smaller than $m$ such that $v_{j_i +1}$ and $v_{j_{i+1}}$ is in the same component of $T$. This yields a subsequence $v_1,v_{j_1},v_{j_1+1}, v_{j_{2}}, \ldots, v_{j_{k}+1}, v_{j_{k+1}}=v_m$ such that $v_{j_i+1} $ and $v_{j_{i+1}}$ is in the same component of $T$ for all $0 \le i \le k$. Let $C_i$ be the component of $v_{j_{i}+1}$. This yields the desired sequence $C_1,C_2, \ldots, C_{k+1}$ with edges $\vec e_i$  having tail  $v_{j_{i}}$ and head $ v_{j_i +1}$. Since $T$ is an almost sure limit, we can increase $M$ if needed to make sure the third and fourth items are satisfied as well with $\P(\cG_M) >c/2$.

Now do the following surgery. If we are not on $\cG_M$, do nothing. On the event $\cG_M$,
\begin{itemize}
\item add the edge $\vec f_i$ which is the reversal of $\vec e_i$.
\item remove the outgoing edges in $T$ from $b_i \wedge a_{i+1}$ for $1 \le i \le k-2$ and from $b'_{k}$,
\item reverse the orientation of all the edges in the path in $C_i$ joining $b_i $ and $b_i \wedge a_{i+1}$ for $1 \le i \le k-2$ and also that of the path joining $b_k$ and $b_k'$.
\end{itemize}
First, observe that the surgery above creates an arborescence $T'$ which has at least $3$ ends. Secondly, observe that $T'$ can be obtained from $T$ by performing a composition of operations of the form $\varphi$ from  \Cref{lem:abs_cont_W'} and extracting the second coordinate of its image. Define $W'$ to be the weights obtained from $W$ by extracting the first coordinate of the same operation. \Cref{lem:abs_cont_W'} ensures that $W'$ is absolutely continuous with respect to $W$. But the MSA under $W'$ has a component with at least 3 ends on the event $\cG_M$ which has a positive probability and since $W'$ is absolutely continuous with respect to $W$ and the MSA is a measurable function of the weights, the MSA for $W'$ is absolutely continuous with respect to the MSA of $W$ as well. This is a contradiction since there is almost surely no component of the MSA for weights $W$ with at least 3 ends. 
\end{proof}

\begin{proof}[Proof of \Cref{cor:unimod_tree}]
Recall the definition of $T_{\textsf{core}}$ from the proof of \Cref{thm:LCRW_transient}.
Take $T_{\textsf{core}}$ and iteratively delete each degree 2 vertex and join their neighbors by an edge to obtain a tree $T'$ with degree at least 3. Note further that since $T$ is nonamenable, $T_{\textsf{core}}$ is a bounded subdivision of a tree with degree at least 3, as neighboring vertices in $T'$ can only be joined by a path of uniformly bounded length in $T_{\textsf{core}}$. Thus using \Cref{thm:convergence}, wired MSA limit exists almost surely for $(T, \rho)$. Furthermore, using \Cref{thm:LCRW_transient}, we see that LCRW is a.s. transient as well, and using the equivalence of CLEB walk and LCRW in the i.i.d.\ Exponential $(1)$ setup (\Cref{lem:CLEB_LCRW}), the corollary follows from \Cref{thm:end}.
\end{proof}

We finish with the proof of \Cref{thm:end_GW}.

\begin{proof}[Proof of \Cref{thm:end_GW}.]
The LCRW on $T_{\text{UGW}}$ is a.s.\ transient as is proved in \Cref{thm:transient_GW}. Thus the wired MSA limit exists a.s.\ by \Cref{prop:necessary_cond}. It follows from \cite[Corollary 1.3]{chen_peres} that $T_{\text{UGW}}$ is invariantly nonamenable. Thus applying  \Cref{thm:end}, we conclude the proof for $T_{\text{UGW}}$.

For $T_{\text{GW}}$, consider the tree $T_{\text{AGW}}$ obtained by connecting two independent copies $T_1,T_2$ of $T_{\text{GW}}$ and connecting them by an edge $e$. Let $v_i$ be the endpoint of $e$ at which $T_i$ is attached for $i = 1,2$. Recall that $T_{\text{AGW}}$ is absolutely continuous with respect to $T_{\text{UGW}}$ (the latter is an inverse degree-biased version of the former). Let $\cC_i$ be the event that there are finitely many components for the wired MSA limit in $T_i$ and let $\cE_i$ be the event that there are finitely many components for the wired MSA limit and that there exists a component of the MSA with at least 2 ends in $T_i$. Note that $\cC_1, \cC_2$ are independent and have the same probability. The same also holds for $\cE_1$ and $\cE_2$. Our goal is to prove that the probability of $\cC_i \cup \cE_i$ is 0 for $i=1,2$.

Let $S_{\tau_{N_1},i}$ be the same as $S_{\tau_{N_1}}$ defined in \Cref{sec:CLEB_walk} for the CLEB walk started at $v_i$ in $T_i$. Since CLEB walk is a.s.\ transient in $T_i$, we have $S_{\tau_{N_1},i}$ is finite in size a.s.\ for $i=1,2$. Let $M$ be the sum of the weights of the oriented edges in $S_{\tau_{N_1},1}$ and $S_{\tau_{N_1},2}$. We can choose $m_0$ sufficiently large such that $\P((\cC_1 \cap \cC_2) \cap \{M \le m_0\}) \ge \frac12\P(\cC_1 \cap \cC_2)$. Now let $\cA$ be the event that the weight of each orientation of $e$ is bigger than $2M$ and set $\cC_{\texttt{join}} = (\cC_1 \cap \cC_2) \cap \cA. $ 
$$\P(\cC_{\texttt{join}}) \ge \P(\cA |(\cC_1 \cap \cC_2) \cap \{M \le m_0\} )\P((\cC_1 \cap \cC_2) \cap \{M \le m_0\}) \ge  \frac{C(m_0)}{2}\P(\cC_1 \cap \cC_2) =  \frac{C(m_0)}{2} \P(\cC_1)^2,$$  
for some constant $C(m_0)>0$ since the weight of the orientations of $e$ are independent of the $T_1,T_2$ and their weights.
On the event $\cA$, none of the orientations of $e$ belongs to the MSA of $T_{\text{AGW}}$ and the MSA of $T_{\text{AGW}}$ is just the union of the two MSAs of $T_1$ and $T_2$. Indeed it is easy to see this from the observation that by definition of $M$, CLEB walk started from $v_1$ will never cross the $e$ oriented from $v_1$ to $v_2$ and the same holds for $v_2$, and from the fact that the MSA is a subset of the edges exposed by the CLEB walk algorithm. Therefore, on the event $\cC_{\texttt{join}}$, there are finitely many components of the MSA of $T_{\text{AGW}}$. This event has zero probability since  $T_{\text{AGW}}$ is absolutely continuous with respect to $T_{\text{UGW}}$, and since the MSA on $T_{\text{UGW}}$ has infinitely many components with probability one. Hence, 
$\P(\cC_1) = \P(\cC_2) =0$. Similarly, by defining $\cE_{\texttt{join}} = (\cE_1 \cap \cE_2) \cap \cA, $ we deduce that $\P(\cE_1) = \P(\cE_2) =0$ as well.
We conclude that each $\cC_i \cup \cE_i$ has zero probability to occur, as desired.
\end{proof}

\section{MSA and UST}\label{sec:wilson}
Suppose $G = (V,E)$ be a finite, rooted, oriented multi-graph with no self-loops and we put weights $c(\vec e) = \exp(-\beta U_{\vec e})$ on every oriented edge $\vec e$, where $(U_{\vec e})_{\vec e \in \vec E}$ satisfies \eqref{eq:linear_comb}. Let $\Tu$ denote the \textbf{uniform spanning arborescence} on $G$ with these weights, i.e., for every spanning arborescence $t$ with root $\partial$,
\begin{equation}
\P(\Tu = t ) = \frac{1}{Z} \exp\left(-\beta \sum_{\vec e \in t} U_{\vec e}\right). \label{eq:ust}
\end{equation}
Note here that we are considering the oriented version of the uniform spanning tree, the latter is perhaps more ubiquitous in probability literature, however in keeping with the terminology in this article, we shall call this the Uniform spanning arborescence. The following lemma explains the motivation behind the above model in our context. The proof is clear from the definitions.

\begin{lemma}
As $\beta \to \infty$, $\Tu$ converges in law to the MSA on $G$ with weights $(U_{\vec e})_{\vec e \in E}$.
\end{lemma}
The \textbf{Wilson's algorithm} is a widely popular algorithm for sampling the uniform spanning arborescence. Usually, Wilson's algorithm has found applications in the unoriented setup, but it works in the oriented setup as well, essentially with no change. To describe how to sample a branch of the uniform spanning tree say from a vertex $x$ to $\partial$, we need to describe the \textbf{loop erased random walk (LERW)}. To that end, assume that $(c(\vec e))_{\vec e \in \vec E}$ is a deterministic collection of strictly positive weights.    A simple random walk on this oriented graph is a Markov chain with state space given by the vertices of the graph, with transition probabilities given by 
$$
q(\vec e) = \frac{c(\vec e)}{ \sum_{\vec f \in \cO(\vec e_-) } c(\vec f)} .$$
Note that in our context, the simple random walk is irreducible. The initial state of this chain is called the starting vertex. Note that a simple random walk can be viewed as the sequence of oriented edges $(\vec e_1, \vec e_2,\ldots ) $ which it traverses. We shall call this sequence the \textbf{trajectory} of the simple random walk.

The LERW from $x$ to $\partial$ is sampled as follows. 

\begin{itemize}
\item Start a simple random walk from $x$ and continue until it hits $\partial$.
\item Erase the cycles chronologically. More precisely, having defined a path $Q_i:=(\vec e_0, \ldots, \vec e_i)$ at step $i$, perform a simple random walk step from $(\vec e_i)_+$ to obtain $\vec e_{i+1}$. If $\vec e_{i+1} = (\vec e_j)_+$ for some $j <i$, then define $Q_{i+1}$ to be $(\vec e_0, \ldots, \vec e_j)$ (we erased the loop $(\vec e_{j+1}, \vec e_{j+2}, \ldots, \vec e_i,\vec e_{i+1})$).
\end{itemize}
An edge in the trajectory of a simple random walk (run up to a certain time step) that belongs to at least one loop that is erased is called an \textbf{erased edge}. We emphasize that an edge can be erased multiple times, but when we talk about the \emph{set of all erased edges}, we ignore this multiplicity.

In the CLEB walk, we contract cycles instead of erasing them. Recall from the definition of contraction, that when a cycle is contracted, we remove a certain set of edges, and call the edges that are removed in a contraction step, \textbf{contracted edges}. Note that this set of edges is strictly larger than the edges in the cycle itself: for example, any oriented edge both of whose endpoints are some endpoint of the oriented edges in the cycle contracted, is in this set.
 Unlike erased edges in Wilson's algorithm, an edge is contracted at most once. Also note that once the weights are deterministic, the set of edges contracted is deterministic as well (this is not the case for the edges erased in the USpA). We shall use the notations from \Cref{sec:CLEB_finite} and in particular recall the notations $(V_i,\vec E_i, (U_{i,e})_{e \in \vec E_i})_{i \ge 0}.$

We recall for convenience that for a CLEB walk started at $x$, $S_i \cap \vec E_i$ is an oriented path $P_i$ started at $x$ (before $\partial $ is hit). Denote: $$P_i := (\vec e_{i,1}, \vec e_{i,2}, \ldots, \vec e_{i,k_i}).$$ Recall that every oriented edge in $G_i$ can be naturally identified with an oriented edge in $G_j$ for $j<i$ via inclusion. Let $a_{i,0} = (\vec e_{i,1})_-$ and $a_{i,j} = (\vec e_{i,j})_+$ and recall that all such $a_{i,j}$ is a vertex in $G_i$. The ordered set of cycles which are contracted to create $G_i$ can be partitioned into classes $$((C_1,\ldots, C_{j_0}),(C_{j_0+1}, \ldots, C_{j_1}), \ldots, (C_{j_{m_i-1} +1}, \ldots, C_{j_{m_i}})) $$
where $(C_{j_{k-1} +1} ,\ldots, C_{j_k})$ consists of the cycles contracted into the vertex $a_{i,k}$ for $k \ge 0$ ($j_{-1}=0$) (See  \Cref{fig:S}). If no cycle is contracted to create $a_{i,k}$ (i.e., $a_{i,k} \in V$), then we take its corresponding cycle set to be empty by convention. Let $A_{i,k}$ be the union of the oriented edges in the cycle sets corresponding to $a_{i,k}$ with the convention that $A_{i,k} = \{a_{i,k}\}$ if the cycle set is empty. Recall the notation $V(S)$ which denotes the set of vertices that are endpoints of the oriented edges in $S$. Notice that for each $k$, $V(A_{i,k})$ has a unique vertex in $V$ which is the head of the oriented edge $\vec e_{i,k}$, and another unique vertex in $V$ which is the tail of the oriented edge $\vec e_{i,k+1}$ (here $\vec e_{i,k}, \vec e_{i,k+1}$ are viewed as edges in $G$.) Call these vertices, $\iota(A_{i,k})$ and $o(A_{i,k})$ respectively (for `incoming' and `outgoing' arrows). 

We recall the definition of $N_1$ from \Cref{sec:CLEB_walk}  for the reader's convenience: $N_1$ is the smallest $k$ such that the path $P_i$ is not a singleton for all $i>k$. We say a CLEB walk from $x$ to $\partial $ \textbf{leaves $x$ forever} at step $N_1$. Recall also that two oriented edges are parallel if their head and tail are the same.
\begin{thm}\label{prop:wilson_cleb}
Suppose $(w_{\vec e})_{\vec e \in \vec E}$ denote a collection of weights satisfying \Cref{eq:linear_comb}. Perform a loop erased random walk from $x$ to $\partial $ with weights $c(\vec e) = \exp(-\beta w_{\vec e})$ and  $E^\beta_{\text{LERW}}$ denote the set of all erased edges until the walk leaves $x$ forever. Perform the CLEB walk from $x$ with weights $(c(\vec e))_{\vec e \in \vec E}$. Let $E_{\text{CLEB}}$ denote the set of all edges in the cycles contracted until the CLEB walk leaves $x$ forever. Let $\bar E_{\text{CLEB}}$ denote all the edges that are contracted in the CLEB walk until it leaves $x$ forever. Then 
$$
\lim_{\beta \to \infty}\P(E_{\text{CLEB}} \subseteq E^\beta_{\text{LERW}}  \subseteq  \bar E_{\text{CLEB}})   =1
$$
\end{thm}
The inclusion above can be strict as shown in \Cref{fig:CLEB_wilson_counter}. 
    \begin{figure}[ht]
        \centering
        \includegraphics[scale = 1]{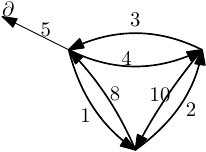}
        \caption{The inclusions in \Cref{prop:wilson_cleb} can be strict. It is not too hard to see in this example that the edge with weight $4$ is in $E^\beta_{\text{LERW}}$ with high probability but is not in $E_{\text{CLEB}}$ and those with weights $8$ and $10$ are not in $E^\beta_{\text{LERW}}$ with high probability.}\label{fig:CLEB_wilson_counter}
    \end{figure}
\begin{proof}
We will use the notations introduced just before the proposition. We will show by induction that if an edge $\vec e$ is added to $S_i$ in the CLEB walk then Wilson's algorithm started from any vertex in $A_{i,k_i}$ will exit through that edge with high probability. Furthermore, if a cycle is created in $G_i$ when this edge is added, then Wilson's algorithm traverses all the edges in this cycle with high probability (in fact it traverses it `infinitely' many times as $\beta$ diverges). It is easy to see that one can deduce the statement in the proposition through these two statements (we prove this below after we record the precise induction statement).

We now state the precise induction hypothesis. Let $\tau_{i,0} < \tau_{i,1}<\ldots <\tau_{i,k_i}$ be the step numbers in the CLEB walk such that at time $\tau_{i,j}$, $a_{i,j}$ is created (i.e., $a_{i,j}$ is in $P_t$ for all $\tau_{i,j}\le  t \le i$). Let $w_{i,j,0}<w_{i,j,1}<\ldots < w_{i,j,n_j}$ be the weights of the outgoing edges from  $a_{i,j}$ in $H_{\tau_{i,j}}$ (recall the notation $H_i$ used to define the sequential CLEB algorithm).  Observe that $w_{i,j,p} - w_{i,j,0}$ is the weight of the outgoing edges from $a_{i,j}$ in $H_{t}$ for $\tau_{i,j}< t \le \tau_{i,k_i}$. Also,  $w_{i,j,0}$ is the weight of $(\vec e_{i,j+1})$ in $G_{\tau_{i,j}}$. Also, the set of outgoing edges from $a_{i,j}$ in $G_{\tau_{i,j}}$, when mapped to $G$, is exactly the set of edges with their tail in $A_{i,j}$ and head outside $A_{i,j}$. Let $\vec e_{i,k_i+1}$ denote the edge with minimal weight out of $a_{i,k_i}$ in $G_{\tau_{i,k_i}}$ (i.e., with weight $w_{i,k_i,0}$). We will  prove by induction that

\begin{claim} \label{claim:induction_loop}For all $i \ge 0$,
For all $0 \le j \le k_i$, the probability that a simple random walk started at any vertex in $V(A_{i,j})$ leaves $V(A_{i,j})$ through the edge with weight $w_{i,j,t}$ is  
$$e^{-\beta (w_{i,j,t} - w_{i,j,0})}) (1+o_\beta(1)), \qquad 0 \le t \le n_{i}.$$
In particular, the above probability for $w_{i,j,0}$ is $1+o_\beta(1)$.
\end{claim}

We first show that the claim proves the theorem. Suppose for $i \ge 1$, $E_\text{CLEB}(i-1)$ be the set of edges in the cycles contracted up to time $\tau_{i-1,k_{i-1}}$ and let $\bar E_{\text{CLEB}(i-1)}$ be the set of edges contracted up to the same time. Let $\cG_{i-1}$ be the event that Wilson's algorithm in $G$ started from $x$ will hit
$A_{i-1,\tau_{i-1,k_{i-1}}}$ and furthermore, the set of vertices visited by Wilson's algorithm until it leaves $A_{i-1,\tau_{i-1,k_{i-1}}}$ for the first time is the same as $\cup_{1 \le z \le k_i}A_{i,z}$ (i.e. the same as those vertices visited by the CLEB walk). Let $t_{i-1}$ be the time at which the simple random walk leaves $A_{i-1,\tau_{i-1,k_{i-1}}}$ for the first time (on $\cG_{i-1}$). Let $E^\beta_\text{LERW}(t_{i-1})$ be the set of edges erased up to time $t_{i-1}$. Assume by induction that $$\P(\cH_{i-1}):= \P(\cG_{i-1} \cap \{E_\text{CLEB}(i-1) \subseteq E^\beta_\text{LERW}(t) \subseteq \bar E_{\text{CLEB}(i-1)}\}) = 1+o_\beta(1).$$ Note the above is trivial for $i=1$. Now we prove the same claim for $P_i$. To that end, we will prove
\begin{equation}
    \P(\cH_i | \cF_{t_{i-1}}\}1_{\cH_{i-1}}\label{eq:ind_LERW1}\\
    =1+o_\beta(1).
\end{equation}
where $\cF_t$ is the $\sigma$-algebra generated by the simple random walk up to time $t$. Note that \eqref{eq:ind_LERW1} is enough for our purposes as the number of steps needed to complete the CLEB walk is bounded, so the error terms cannot blow up. 

Condition on the simple random walk trajectory up to time $t_{i-1}$ and assume it is on $\cH_{i-1}$. Note that by \Cref{claim:induction_loop}, the simple random walk leaves $A_{i-1,\tau_{i-1,k_{i-1}}}$ through the same edge as the CLEB walk with high probability, and that edge is denoted $\vec e_{i-1,k_{i-1}+1}$. If $ (\vec e_{i-1,k_{i-1}+1})_+$ is a vertex never visited before by the CLEB walk, then no edge is contracted in the CLEB walk. Since we are in $\cG_{i-1}$, $ (\vec e_{i-1,k_{i-1}+1})_+$ is never visited before time $t_{i-1}$ by the simple random walk as well, and hence no further loop is erased. Thus in this case, \eqref{eq:ind_LERW1} is trivial.

Suppose now that $ (\vec e_{i-1,k_{i-1}+1})_+$ is a vertex previously visited, then a cycle $$C =  (\vec e_{i-1,j+1}, \ldots \vec e_{i-1,k_{i-1}}, \vec e_{i-1,k_{i-1}+1})$$ for some $j <i$ in $G_i$ is created with edges in $P_i$. Then by \Cref{claim:induction_loop}, the simple random walk from $(\vec e_{i-1,k_{i-1}+1})_+$ in $G$ must trace every edge in $C$ at least once before coming back to $(\vec e_{i-1,k_{i-1}+1})_+$ with probability $1+o_\beta(1)$, thereby creating a loop containing all the edges in $C$, which must be erased.  Furthermore, any other edge erased before leaving $A_{i,k_i}= \cup_{j \le t \le k_{i-1}}A_{i-1,t}$ is contained in $\bar E_{\text{CLEB}(i-1)}$ by definition. By \Cref{claim:induction_loop}, $\cG_i$ also holds with probability $1+o_\beta(1)$. Combining these facts proves \Cref{eq:ind_LERW1} and hence the theorem.

We now prove \Cref{claim:induction_loop}. For brevity, we drop the subscripts $\beta$ in $o_\beta$ below. For $i=0$, the claim is trivial. Indeed, the probability that the simple random walk started at $a_{0,0} = x$ exits through the edge with weight $w_{0,0,t}$ is 
$$
\frac{e^{-\beta w_{0,0,t}}}{e^{-\beta w_{0,0,0}} + o_\beta(e^{-\beta w_{0,0,0}})} = e^{-\beta (w_{0,0,t} - w_{0,0,0})} (1+o(1))
$$
since all weights are distinct by \Cref{eq:linear_comb}.
We now assume that the claim is true for $0 \le t \le i$ and assume that $a_{i,k_i} \neq \partial$ since otherwise we are done. Now assume that the minimal weight edge out of $a_{i,k_i}$ has an endpoint $v$ not in $S_i$. 
Applying the induction hypothesis, we are only left to prove the claim for $a_{i+1,k_{i+1}} = v$. This is similar to the $i=0$ case and we skip the details here.

We now assume that the minimal weight edge out of $a_{i,k_i}$ has endpoint $a_{i,j}$ for some $1 \le j \le k_i$ to create a cycle $C :=(\vec e_{i,j+1}, \ldots \vec e_{i,k_{i}}, \vec e_{i,k_{i}+1})$ in $G_i$.  Thus $a_{i+1,t} = a_{i,t}$ for $1 \le t \le j-1$ and $C$ is contracted to create a new vertex $a_{i+1,k_{i+1}}$. 
We now prove the claim for $A_{i+1,k_{i+1}}$ and this is the hardest case. Observe that $A_{i+1,k_{i+1}}$ in $G$ can be seen as a disjoint union of edge sets in $A_{i,j}, A_{i,j+1}, \ldots, A_{i,k_i}$ and those in $C$. Call these sets of edges $\cE$. Let $1-\ve_{i,t}$ denote the probability that a simple random walk started from $\iota(A_{i,t})$ leaves $A_{i,t}$ through $ \vec e_{i,t+1}$. By the induction hypothesis, the simple random  walk started from some vertex in $A_{i,t}$ leaves $\cE$ through the edge with weight $w_{i,t',p}$ is 
\begin{equation}
q(i,t,t',p):=C(t,t')\sum_{n \ge 0}\left(\prod_{j \le s \le k_i } (1-\ve_{i,s}) \right)^n e^{-\beta (w_{i,t',p} - w_{i,t',0})}(1+o(1))
\end{equation}
where $C(t,t') = 1+o(1)$ is the probability that the always walk escapes $A_{i,s}$ through $ \vec e_{i,s+1}$ until it reaches $A_{i,t'}$.
By induction hypothesis, \begin{equation*}\ve_{i,s} = \sum_{p=1}^{n_s} e^{-\beta(w_{i,s,p}  - w_{i,s,0})} (1+o(1))\end{equation*}

Hence
\begin{equation*}
\sum_{j \le s \le k_i}\ve_{i,s}  =\sum_{j \le s \le k_i}  \sum_{p=1}^{n_s} e^{-\beta(w_{i,s,p}  - w_{i,s,0})} (1+o(1)) 
\end{equation*}

Thus
\begin{equation*}
q(i,t,t',p)=\frac{e^{-\beta ((w_{i,t',p} - w_{i,t',0})}}{\sum_{j \le s \le k_i}  \sum_{p=1}^{n_s} e^{-\beta(w_{i,s,p}  - w_{i,s,0})} } (1+o(1)) =:q'(i,t',p)
\end{equation*}
Let $\cS$ be the set of all pairs $(t',p)$ which correspond to the edges with both endpoints in $\cup_{j\le s \le k_i}V(A_{i,s})$. Let $\cO$ be the set of index pairs for edges coming out of a vertex in $\cup_{j \le s \le k_i} V(A_{i,s})$ which do not belong in $\cS$. The probability that the walk escapes $\cE \cup \cS$ through an edge with weight $w_{i,t',p}$ is
\begin{equation*}
\sum_{n \ge 0}(q'(i,\cS))^nq'(i,t',p) = \frac{e^{-\beta ((w_{i,t',p} - w_{i,t',0}) }}{\sum_{(t',p) \in \cO}e^{-\beta ((w_{i,t',p} - w_{i,t',0}) )}} (1+o(1))
\end{equation*}
where $q'(i,\cS) = \sum_{x \in \cS} q'(i,x)$.
Let $w_{i,s^*,p^*} - w_{i,s^*,0}$ achieves the minimum of $\{w_{i,s,p}  - w_{i,s,0} : (s,p) \in \cO\}$. The above probability can be written as
\begin{equation*}
e^{-\beta ((w_{i,t',p} - w_{i,t',0}) - (w_{i,s^*,p^*} - w_{i,s^*,0}) }  (1+o(1)).
\end{equation*}
Now observe that the weights of the edges coming out of the collapsed vertex $a_{i+1,k_{i+1}}$ (after modification by the CLEB algorithm) are exactly  $(w_{i,t',p} - w_{i,t',0}) - (w_{i,s^*,p^*} - w_{i,s^*,0})$ for $(t',p) \in \cO$. This completes the proof of \Cref{claim:induction_loop}.
\end{proof}

\section{Open questions and simulation}\label{sec:open}

We finish with some open questions. Of course, studying loop contracting random walk is an interesting question in its own right. 
\begin{conjecture} \label{qn:LCRW_transient}
The loop contracting random walk on a transient, vertex-transitive graph is transient at all its vertices with probability one. 
\end{conjecture}

\begin{figure}[ht]
\centering
\includegraphics[width = 5.8cm, height = 5.8cm]{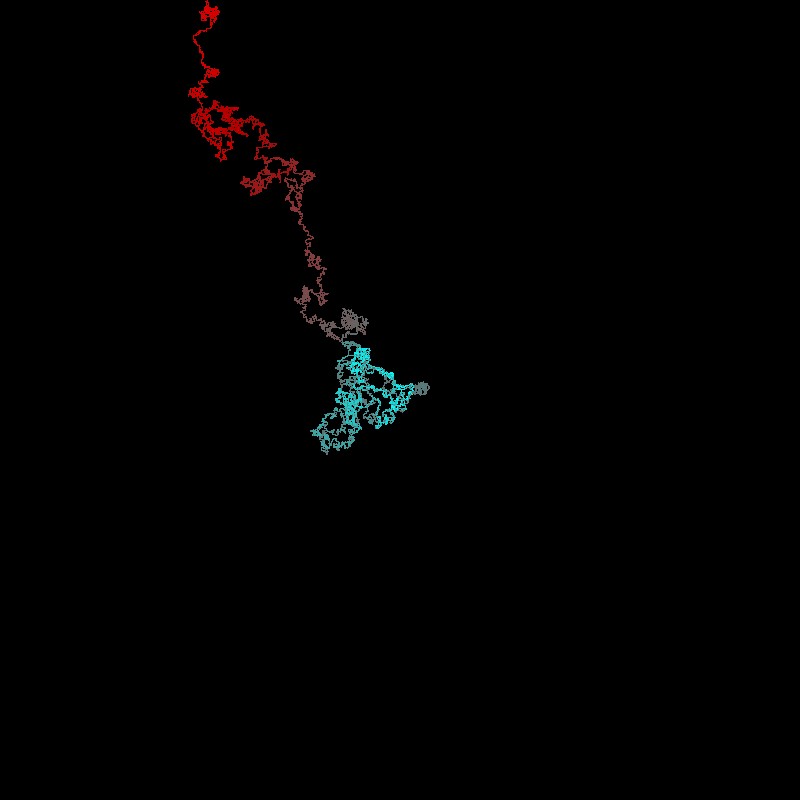}
\includegraphics[width = 5.8cm, height = 5.8cm]{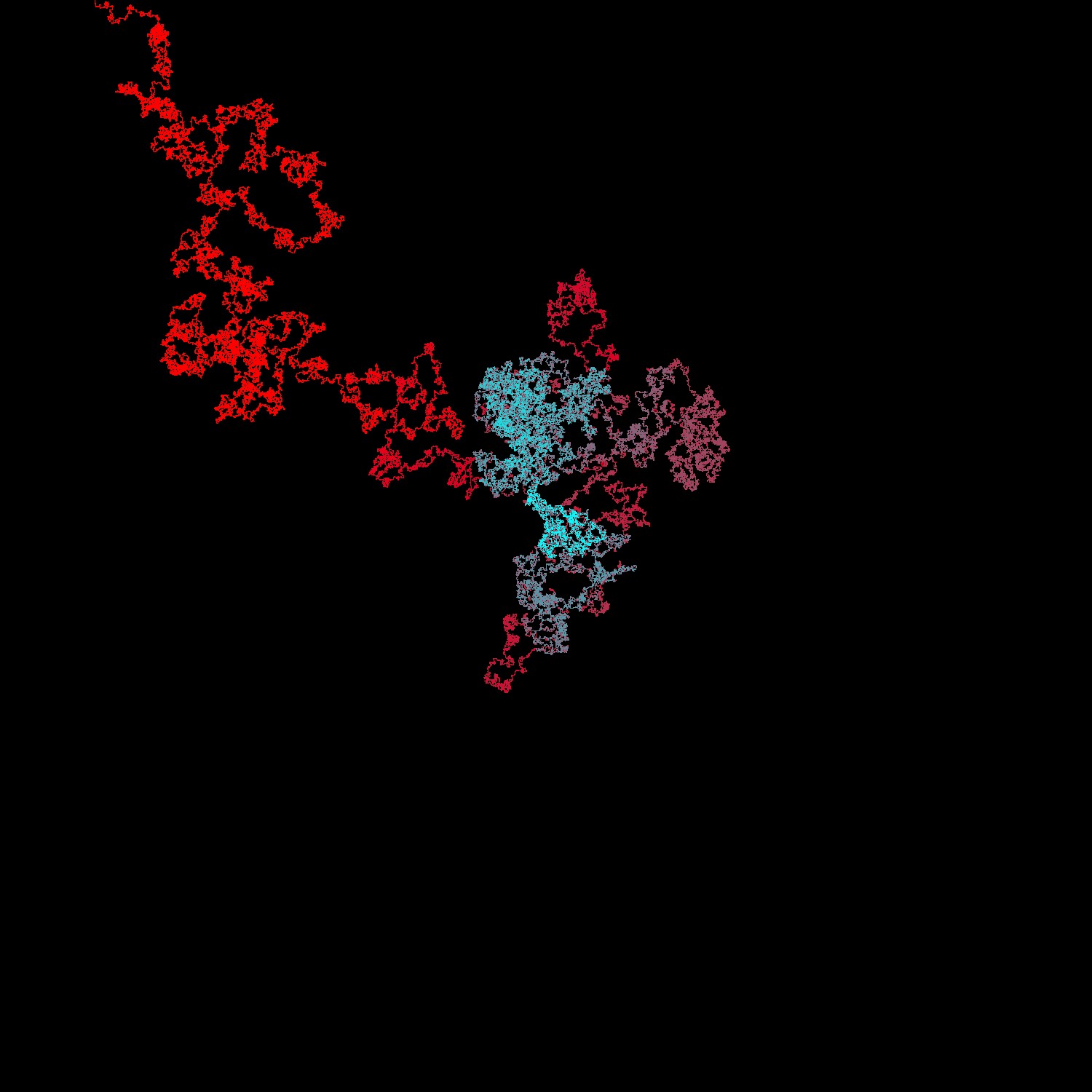}\\
\includegraphics[width = 6.5cm, height = 5cm]{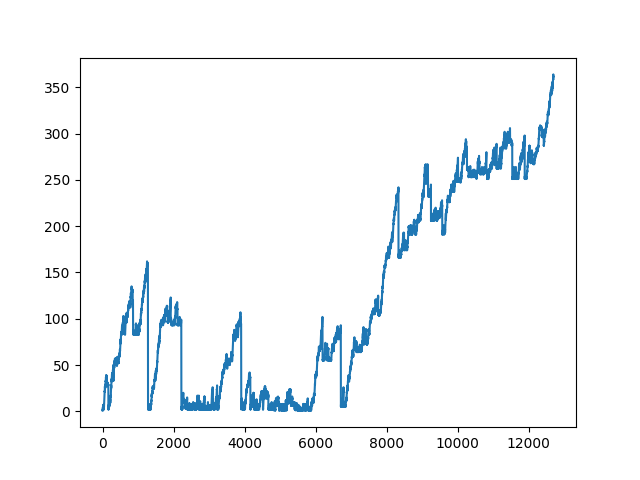}
\includegraphics[width = 6.5cm, height = 4.85cm]{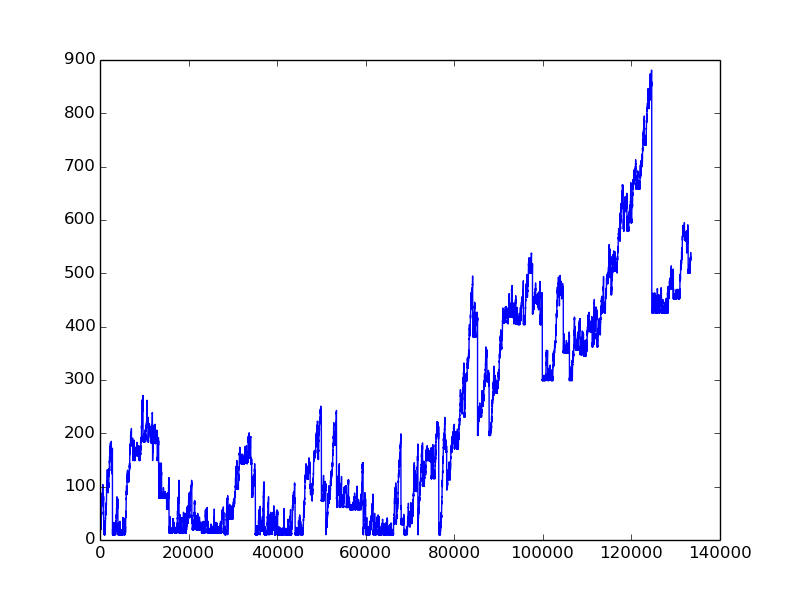}
\caption{The top panel shows simulations of LCRW on an 800 by 800 grid (left) and 1500 by 1500 grid (right) of $\Z^2$. The loops are not contracted and presented as a subset of $\Z^2$ (drawing $S_i$ from \Cref{sec:CLEB_finite}). The colors indicate the step number of the process, changing from turquoise to red. The bottom panel plots for the length of the contracted path $P_i$ for 800 by 800 grid (left) and 1500 by 1500 grid (right).}
\label{simulation_800}
\end{figure}

Also, analogues of \Cref{thm:convergence} is interesting for other distribution of weights:
\begin{conjecture}
 \Cref{thm:convergence} is true for all transient, vertex-transitive graphs with weight distributions that are continuous and supported on subsets of $[0,\infty)$. A particularly interesting case is Uniform $[0,1]$.
\end{conjecture}

We also add a more general question about the transience of the CLEB walk:
\begin{question}\label{qn:0-1}
Is there a 0-1 law for transience and recurrence for the CLEB walk, analogous to Markov chains?
\end{question} 
Loop contracting random walk on $\Z^2$ is in itself an interesting question. We present some simulations in the square lattice in Figure~\ref{simulation_800}. Simulations do indicate that \Cref{qn:LCRW_transient} is true in $\Z^2$ as well.

\bibliographystyle{amsplain}
\bibliography{MSA}

\end{document}